\documentclass[12pt]{amsart}
\usepackage{psfrag,epsfig,url}
\numberwithin{figure}{section}
\numberwithin{table}{section}
\newtheorem{theorem}{Theorem}[section]
\newtheorem*{thrm}{Main Theorem}
\newtheorem{lemma}[theorem]{Lemma}
\newtheorem{prop}[theorem]{Proposition}

\theoremstyle{definition}
\newtheorem{definition}[theorem]{Definition}
\newtheorem{example}[theorem]{Example}
\newtheorem{cor}[theorem]{Corollary}

\theoremstyle{remark}
\newtheorem{remark}[theorem]{Remark}
\numberwithin{equation}{section}

\def \Q{{\mathbb Q}}

\def \A{{\mathfrak A}}
\def \[{[ }
\def \]{] }

\def \l{\langle }
\def \r{\rangle }
\def \B{{\mathcal B}}

\begin{document}

\author{Anna Felikson}
\address{Independent University of Moscow, B. Vlassievskii 11, 119002 Moscow, Russia}
\curraddr{Max-Planck Institute for Mathematics, Vivatsgasse 7, D-53111, Germany}
\email{felikson@mccme.ru}
\thanks{Research supported by grants INTAS YSF-06-10000014-5916 (A.~F.), RFBR 07-01-00390-a (A.~F. and P.~T.), PHY 0555346 and DNS 0800671 (M.~S.), and INTAS YSF-06-10000014-5766 (P.~T.)}

\author{Michael Shapiro}
\address{Department of Mathematics, Michigan State University, East Lansing, MI 48824, USA}
\email{mshapiro@math.msu.edu}

\author{Pavel Tumarkin}
\address{School of Engineering and Science, Jacobs University Bremen, Campus Ring 1, D-28759, Germany} 
\email{p.tumarkin@jacobs-university.de}

\title{Skew-symmetric cluster algebras of finite mutation type}


\begin{abstract}
In the famous paper~\cite{FZ2} Fomin and Zelevinsky obtained Cartan-Killing type classification of
all cluster algebras of finite type, i.e. cluster algebras having only finitely many distinct cluster variables.
A wider class of cluster algebras is formed by cluster algebras of finite mutation type which have finitely many
exchange matrices (but are allowed to have infinitely many cluster variables).
In this paper we classify all cluster algebras of finite mutation type with skew-symmetric exchange matrices.
Besides cluster algebras of rank $2$ and cluster algebras associated with triangulations of surfaces there
are exactly $11$ exceptional skew-symmetric cluster algebras of finite mutation type.
More precisely, $9$ of them are associated with root systems
$E_6$, $E_7$, $E_8$, $\widetilde E_6$, $\widetilde E_7$, $\widetilde E_8$, $E_6^{(1,1)}$, $E_7^{(1,1)}$, $E_8^{(1,1)}$;
two remaining were found by Derksen and Owen in~\cite{DO}.
We also describe a criterion which determines if a skew-symmetric cluster algebra is of finite mutation type,
and discuss growth rate of cluster algebras.
\end{abstract}

\maketitle

\tableofcontents

\section{Introduction}
Cluster algebras were introduced by Fomin and Zelevinsky in the sequel of papers~\cite{FZ1},~\cite{FZ2},~\cite{BFZ3},
~\cite{FZ4}.

We think of cluster algebra as a subalgebra of $\Q(x_1,\dots,x_n)$ determined by generators
(``cluster coordinates''). These generators are collected into $n$-element
groups called {\it clusters} connected by local transition rules which
are determined by an $n\times n$ skew-symmetrizable { exchange matrix} associated with each cluster. For precise definitions see Section~\ref{cluster}.

In~\cite{FZ2}, Fomin and Zelevinsky discovered a deep connection between cluster algebras of {\it finite type} (i.e., cluster algebras containing finitely many clusters)   and Cartan-Killing classification of simple Lie algebras. More precisely, they proved that there is a bijection between Cartan matrices of finite type and cluster algebras of finite type. The corresponding Cartan matrices can be obtained by some symmetrization procedure of exchange matrices.

 Exchange matrices undergo {\it mutations} which are explicitly described locally.
 Collection of all exchange matrices of a cluster algebra form a {\it mutation class} of exchange matrices.
 In particular, mutation class of a cluster algebra of finite type is finite.
 In this paper, we are interested in a larger class of cluster algebras,
 namely, cluster algebras whose exchange matrices form finite mutation class.
 We will assume that exchange matrices are skew-symmetric.

Besides cluster algebras of finite type, there exist other series of
algebras belonging to the class in consideration. One series of
examples is provided by cluster algebras corresponding to Cartan
matrices of affine Kac-Moody algebras with simply-laced Dynkin
diagrams. It was shown in~\cite{BR} that these examples exhaust all
cases of acyclic skew-symmetric cluster algebras of finite mutation
type. Furthermore, Seven in~\cite{S2} has shown that acyclic
skew-symmetrizable cluster algebras of finite mutation type
correspond to affine Kac-Moody algebras.
%

One more large class of infinite type cluster algebras of finite mutation type was studied in the paper~\cite{FST}, where, in particular, was shown that signed adjacency matrices of arcs of a triangulation of a bordered two-dimensional surface have finite mutation class.


In the same paper, Fomin, Shapiro and Thurston discussed the conjecture~\cite[Problem~12.10]{FST} that besides adjacency matrices of triangulations of bordered two-dimensional surfaces and matrices mutation-equivalent to one of the following nine types: $E_6$, $E_7$, $E_8$, $\widetilde E_6$, $\widetilde E_7$, $\widetilde E_8$,  $E_6^{(1,1)}$,  $E_7^{(1,1)}$,  $E_8^{(1,1)}$ (see~\cite[Section~12]{FST}), there exist finitely many skew-symmetric matrices of size at least $3\times 3$ with finite mutation class. Notice that the first three types in the list correspond to cluster algebras of finite type.

In the preprint~\cite{DO} Derksen and Owen found two more skew-symmetric matrices (denoted by $X_6$ and $X_7$) with finite mutation class that are not included in the previous conjecture. The authors also ask if their list of $11$ mutation classes contains all the finite mutation classes of skew-symmetric matrices of size at least $3\times 3$ not corresponding to triangulations.

\emph{The main goal} of this paper is to prove the conjecture by Fomin, Shapiro and Thurston by showing the completeness of the Derksen-Owen list, i.e. to prove the following theorem:

\begin{thrm}[Theorem~\ref{all}]
Any skew-symmetric $n\times n$ matrix, $n\ge 3$, with finite mutation class is either an adjacency matrix of triangulation of a bordered two-dimensional surface or a matrix mutation-equivalent to a matrix of one of the following eleven types: $E_6$, $E_7$, $E_8$, $\widetilde E_6$, $\widetilde E_7$, $\widetilde E_8$,  $E_6^{(1,1)}$,  $E_7^{(1,1)}$, $E_8^{(1,1)}$, $X_6$, $X_7$.

\end{thrm}

\begin{remark} 
The same approach that we used for skew-symmetric matrices is applicable (after small changes) for the more general case of skew-symmetrizable matrices. The complete list of skew-symmetrizable matrices with finite mutation class will be published elsewhere.
\end{remark}

We also show a way to classify all minimal skew-symmetric $n\times n$ matrices with infinite mutation class. In particular, we prove that $n\le 10$. This gives rise to the following criterion for a large skew-symmetric matrix to have finite mutation class:

\setcounter{theorem}{3}
\setcounter{section}{7}
\begin{theorem}
A skew-symmetric $n\times n$ matrix $B$, $n\ge 10$, has finite mutation class if and only if a mutation class of every principal $10\times 10$ submatrix of $B$ is finite.

\end{theorem}
\setcounter{theorem}{0}
\setcounter{section}{1}

As an application of the classification of skew-symmetric matrices of finite mutation type
 we characterize skew-symmetric cluster algebras of polynomial growth, i.e. cluster algebras for which the number of distinct clusters obtained from the initial one by $n$ mutations grows polynomially in $n$.

\medskip

The paper is organized as follows. In Section~\ref{cluster}, we provide necessary background in cluster algebras, and reformulate the classification problem of skew-symmetric matrices in terms of {\it quivers} by assigning to every exchange matrix an oriented weighted graph.

In Section~\ref{main}, we present the sketch of the proof of the Main Theorem. We list all the key steps, and discuss the main combinatorial and computational ideas we use. Sections~\ref{blockdecomp}--\ref{class} contain the detailed proofs.

Section~\ref{blockdecomp} is devoted to the technique of block-decomposable quivers. We recall the basic facts from~\cite{FST} and prove several properties we will heavily use in the sequel. Section~\ref{minimal} contains the proof of the key theorem classifying minimal non-decomposable quivers. Section~\ref{class} completes the proof of the Main Theorem.

In Section~\ref{inf} we provide a criterion for a skew-symmetric matrix to have finite mutation class. Section~\ref{growth} is devoted to growth rates of cluster algebras.

Finally, in Section~\ref{q3} we use the results of the previous section to complete the description of mutation classes of quivers of order $3$.

\medskip

We would like to thank M.~Barot, V.~Fock, S.~Fomin,
C.~Geiss, A.~Goncharov, B.~Keller, A.~Seven, and A.~Zelevinsky for
their interest in the problem and many fruitful discussions, and H.Thomas for explaining us that linear growth of affine
cluster algebras is a corollary of categorification theory. The first and the third authors are grateful to the
University of Fribourg for a great atmosphere during their visit,
and for a partial support by SNF projects 200020-113199 and
200020-121506/1. The second author thanks EPFL, Stockholm
University, and Royal Institute of Technology, the third author
thanks Michigan State University for support during the work on this
paper.

\section{Cluster algebras, mutations, and quivers}
\label{cluster}

\noindent
We briefly remind the definition of coefficient-free cluster algebra.

An integer $n\times n$ matrix $B$ is called \emph{skew-symmetrizable} if there exists an
integer diagonal $n\times n$ matrix $D=diag(d_1,\dots,d_n)$,
such that the product $DB$ is a skew-symmetric matrix, i.e.,
                                       $d_i b_{i,j}=-b_{j,i}d_j$.

\emph{A seed} is a pair $(f,B)$, where $f=\{f_1,\dots,f_n\}$ form a collection of algebraically independent rational functions of $n$ variables
$x_1,\dots,x_n$, and $B$ is a skew-symmetrizable matrix.

The part $f$ of seed $(f,B)$ is called \emph{cluster}, elements $f_i$ are called \emph{cluster variables},
and $B$ is called \emph{exchange matrix}.

\begin{definition}
For any $k$, $1\le k\le n$ we define \emph{the mutation} of seed $(f,B)$ in direction $k$
as a new seed $(f',B')$ in the following way:
\begin{equation}\label{eq:MatrixMutation}
B'_{i,j}=\left\{
           \begin{array}{ll}
             -B_{ij}, & \hbox{ if } i=k \hbox{ or } j=k; \\
             B_{ij}+\frac{|B_{ik}|B_{kj}+B_{ik}|B_{kj}|}{2}, & \hbox{ otherwise.}
           \end{array}
         \right.
\end{equation}

\begin{equation}\label{eq:ClusterMutation}
f'_i=\left\{
           \begin{array}{ll}
             f_i, & \hbox{ if } i\ne k; \\
             \frac{\prod_{B_{ij}>0} f_j^{B_{ij}}+\prod_{B_{ij}<0} f_j^{-B_{ij}}}{f_i}, & \hbox{ otherwise.}
           \end{array}
         \right.
\end{equation}
\end{definition}

\noindent
We write $(f',B')=\mu_k\left((f,B)\right)$.
Notice that $\mu_k(\mu_k((f,B)))=(f,B)$.
We say that two seeds are \emph{mutation-equivalent}
if one is obtained from the other by a sequence of seed mutations.
Similarly we say that two clusters or two exchange matrices are \emph{mutation-equivalent}.

Notice that exchange matrix mutation~\ref{eq:MatrixMutation} depends only on the exchange matrix itself.
The collection of all matrices mutation-equivalent to a given matrix $B$ is called the \emph{mutation class} of $B$.

For any skew-symmetrizable matrix $B$ we define \emph{initial seed} $(x,\!B)$ as
$(\{x_1,\dots,x_n\},\!B)$, $B$ is the \emph{initial exchange matrix}, $x=\{x_1,\dots,x_n\}$ is the \emph{initial cluster}.

{\it Cluster algebra} $\A(B)$ associated with the skew-sym\-met\-ri\-zab\-le $n\times n$ matrix $B$ is a subalgebra of $\Q(x_1,\dots,x_n)$ generated by all cluster variables of the clusters mutation-equivalent
to the initial cluster $(x,B)$.

Cluster algebra $\A(B)$ is called \emph{of finite type} if it contains only finitely many
cluster variables. In other words, all clusters mutation-equivalent to initial cluster contain
totally only finitely many distinct cluster variables.

In~\cite{FZ2}, Fomin and Zelevinsky proved a remarkable theorem that cluster algebras of finite type
can be completely classified. More excitingly, this classification is parallel to the famous Cartan-Killing classification
of simple Lie algebras.


Let $B$ be an integer $n\times n$ matrix. Its \emph{Cartan companion} $C(B)$ is the integer $n\times n$ matrix defined as follows:

\begin{equation*}
  C(B)_{ij}=\left\{
              \begin{array}{ll}
                2, & \hbox{ if } i=j; \\
                 -|B_{ij}|, & \hbox{ otherwise.}
              \end{array}
            \right.
\end{equation*}

\begin{theorem}[\cite{FZ2}]
\label{thm:FinTypeClass}
There is a canonical bijection between the Cartan matrices of
finite type and cluster algebras of finite
type. Under this bijection, a Cartan matrix $A$ of finite type corresponds to the cluster algebra
$\A(B)$, where $B$ is an arbitrary skew-symmetrizable matrix with $C(B) = A$.
\end{theorem}

The results by Fomin and Zelevinsky were further developed in~\cite{S} and~\cite{BGZ}, where the effective criteria for cluster algebras of finite type were given.

A cluster algebra of finite type has only finitely many distinct seeds.
Therefore, any cluster algebra that has only finitely many cluster variables contains only finitely many
distinct exchange matrices. Quite the contrary, the cluster algebra with finitely many exchange matrices
is not necessarily of finite type.

\begin{definition}\label{def:FinMutType} A cluster algebra with only finitely many exchange matrices is called \emph{of finite mutation type}.
\end{definition}

\begin{example} One example of infinite cluster algebra of finite mutation type is the Markov cluster algebra whose
exchange matrix is
\begin{center}
$\begin{pmatrix}
   0 & 2 & -2 \\
   -2 & 0 & 2 \\
   2 & -2 & 0 \\
 \end{pmatrix}$
 \end{center}
 It was described in details in~\cite{FZ1}. Markov cluster algebra is not of finite type, moreover, it is even not finitely generated.
 Notice, however, that mutation in any direction leads simply to sign change of exchange matrix.
 Therefore, the Markov cluster algebra is clearly of finite mutation type.
\end{example}

\begin{remark} 
Since the orbit of an exchange matrix depends on the exchange matrix only, we may speak about skew-symmetrizable matrices of finite mutation type.
\end{remark}

Therefore, the Main Theorem describes  \emph{all skew-symmetric integer matrices whose mutation class is finite}.



For our purposes it is convenient to encode an $n\times n$ skew-symmetric integer matrix $B$ by a finite oriented multigraph without loops and $2-$cycles called {\it quiver}.
More precisely, a {\it quiver} $S$ is a finite $1$-dimen\-sional simplicial complex with oriented weighted edges, where weights are positive integers.

Vertices of $S$ are labeled by $[1,\dots,n]$. If $B_{i,j}>0$, we join vertices $i$ and $j$ by an  edge directed from $i$ to $j$ and assign to this edge weight $B_{i,j}$. Vice versa, any quiver with integer positive weights corresponds to a skew-symmetric integer matrix. While drawing quivers, usually we draw edges of weight $B_{i,j}$ as edges of multiplicity $B_{i,j}$, but sometimes, when it is more convenient, we put the weight on simple edge.

Mutations of exchange matrices induce {\it mutations of quivers}. If $S$ is the quiver corresponding to matrix $B$, and $B'$ is a mutation of $B$ in direction $k$, then we call the quiver $S'$ associated to $B'$ a {\it mutation of $S$ in direction $k$}. It is easy to see that mutation in direction $k$ changes weights of quiver in the way described in the following picture (see e.g.~\cite{K2}):

\begin{figure}[!h]
\begin{center}
\psfrag{a}{\scriptsize $a$}
\psfrag{b}{\scriptsize $b$}
\psfrag{c}{\scriptsize $c$}
\psfrag{d}{\scriptsize $d$}
\psfrag{k}{\scriptsize k}
\epsfig{file=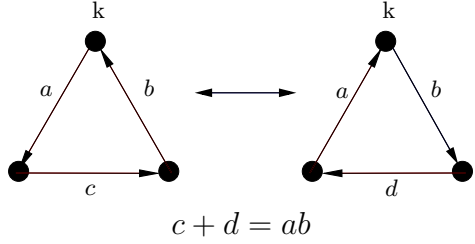,width=0.4\linewidth}\\
$c+d=ab$
\caption{Mutations of quivers}
\label{quivermut}

\end{center}
\end{figure}

Clearly, for given quiver the notion of {\it mutation class} is well-defined. We call a quiver {\it mutation-finite} if its mutation class is finite. Thus, we are able to reformulate the problem of classification of exchange matrices of finite type in terms of quivers: \emph{find all mutation-finite quivers}.

The following criterion for a quiver to be mutation-finite is well-known (see e.g.~\cite[Corollary~8]{DO})

\begin{theorem}
\label{less3}
A quiver $S$ of order at least $3$ is mutation-finite if and only if any quiver in the mutation class of $S$ contains no edges of weight greater than $2$.
\end{theorem}

\medskip

One can use linear algebra tools to describe quiver mutations.

Let $e_1,\dots,e_n$ be a basis of vector space $V$ over field $k$ equipped with a skew-symmetric form
$\Omega$. Denote by $B$  the matrix of the form $\Omega$ with respect to the basis $e_i$, i.e. $B_{ij}=\Omega(e_i,e_j)$.

For each $i\in [1,n]$ we define new basis $e'_1,\dots,e'_n$ in the following way:
\begin{eqnarray*}
  e'_i &=& -e_i \\
  e'_j &=& e_j, \text{ if } \Omega(e_i,e_j)\ge 0 \\
  e'_j &=& e_j-\Omega(e_i,e_j)e_i, \text{ if } \Omega(e_i,e_j) < 0.
\end{eqnarray*}

Note that matrix $B'$ of the form $\Omega$ in basis $e'_k$ is the mutation of matrix $B$ in direction~$i$.

\medskip

From now on, we use language of quivers only. Let us fix some notations we will use throughout the paper.

Let $S$ be a quiver. A {\it subquiver} $S_1\subset S$ is a subcomplex of $S$. The {\it order}  $|S|$ is the number of vertices of quiver $S$.
If $S_1$ and $S_2$ are sub\-quivers of quiver $S$, we denote by $\l S_1,S_2\r$ the sub\-quiver of $S$ spanned by all the vertices of $S_1$ and $S_2$.

Let $S_1$ and $S_2$ be subquivers of $S$ having no common vertices. We say that $S_1$ and $S_2$ are {\it orthogonal} ($S_1\perp S_2$) if no edge joins vertices of $S_1$ and $S_2$.

We denote by ${\rm Val}_{S}(v)$ the unsigned valence of $v$ in $S$ (a double edge adds two to the valence). A {\it leaf} of $S$ is a vertex joined with exactly one vertex in $S$. 

\section{Ideas of the proof}
\label{main}

In this section we present all key steps of the proof.

We need to prove that all mutation-finite quivers except some finite number of mutation classes satisfy some special properties, namely they are block-decomposable (see Definition~\ref{defblock}). In Section~\ref{min} we define a {\it minimal non-decomposable quiver} as a non-decomposable quiver minimal with respect to inclusion (see Definition~\ref{defmin}). By definition, any non-decomposable quiver contains a minimal non-decomposable quiver as a subquiver. First, we prove the following theorem:

\setcounter{theorem}{1}
\setcounter{section}{5}
\begin{theorem}
Any minimal non-decomposable quiver contains at most $7$ vertices.

\end{theorem}

The proof of Theorem~\ref{g8} contains the bulk of all the technical details in the paper. We assume that there exists a minimal non-decomposable quiver of order at least $8$, and investigate the structure of block decompositions of proper subquivers of $S$. By exhaustive case-by-case consideration we prove that $S$ is also block-decomposable. The main tools are Lemmas~\ref{razval1} and~\ref{razval2} which under some assumptions produce a block decomposition of $S$ from block decompositions of proper subquivers of $S$.

The next step is to prove the following key theorem:

\setcounter{theorem}{10}
\setcounter{section}{5}
\begin{theorem}
Any minimal non-decomposable mutation-finite qui\-ver is mutation-equivalent to one of the two quivers $X_6$ and $E_6$ shown below.
\begin{center}
\psfrag{E}{$E_6$}
\psfrag{X}{$X_6$}
\epsfig{file=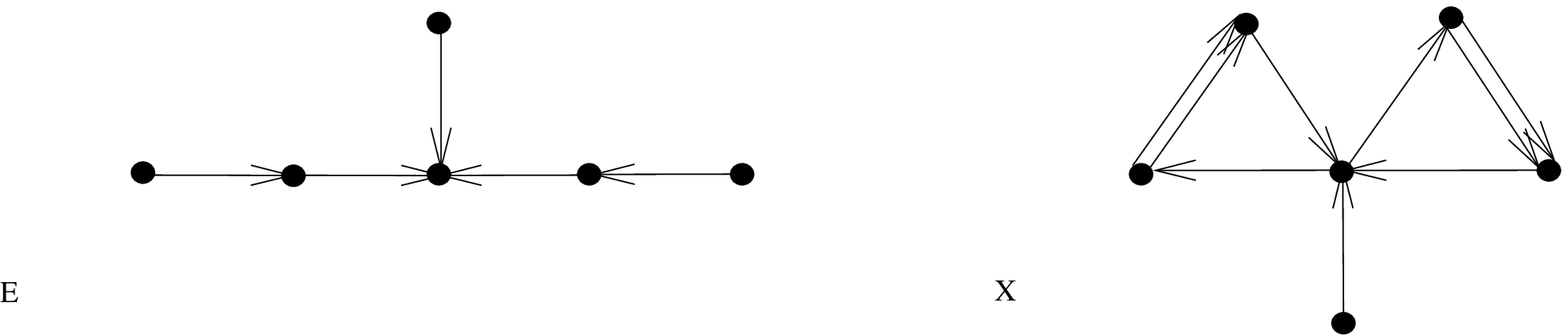,width=0.7\linewidth}
\end{center}
\end{theorem}

The proof is based on the fact that the number of mutation-finite quivers of order at most $7$ is finite, and all such quivers can be easily classified. For that, we use an inductive procedure: we take one representative from each finite mutation class of quivers of order $n$ and attach a vertex by edges of multiplicity at most $2$
in all possible ways (here we use Theorem~\ref{less3}). For each obtained quiver we check if its mutation class is finite (by using Keller's applet for quivers mutations~\cite{K}). In this way we get all the finite mutation classes of quivers of order $n+1$. After collecting all finite mutation classes of order at most $7$, we analyze whether they are block-decomposable. It occurs that all the classes except ones containing $X_6$ and $E_6$ are block-decomposable. The two quivers $X_6$ and $E_6$ are non-decomposable by~\cite[Propositions 4 and 6]{DO}).

Therefore, we proved that each mutation-finite non-decompos\-able quiver contains a subquiver mutation-equivalent to $X_6$ or $E_6$ (Corollary \ref{contmin}). This allows us to use the same inductive procedure to get all the finite classes of non-decomposable quivers.
We attach a vertex to $X_6$ and $E_6$ by edges of multiplicity at most $2$ in all possible ways. In this way we get all the finite mutation classes of non-decomposable quivers of order $7$. More precisely, there are $3$ of them, namely those containing $X_7$, $E_7$ and $\widetilde E_6$. Any mutation-finite non-decomposable quiver of order $8$ should contain a subquiver mutation-equivalent to one of these $3$ quivers due to the following lemma:

\setcounter{theorem}{3}
\setcounter{section}{6}
\begin{lemma}
Let  $S_1$ be a proper subquiver of $S$, let $S_0$ be a quiver mutation-equivalent to $S_1$. Then there exists a quiver $S'$ which is mutation-equivalent to $S$ and contains $S_0$.

\end{lemma}

Using the same procedure, we list one-by-one all the mutation-finite non-decompos\-able quivers of order $8$, $9$ and $10$. The results are the entries from the list by Derksen-Owen (see Fig.~\ref{allfig}). Applying the inductive procedure to a unique mutation-finite non-decomposable quiver $E_8^{(1,1)}$ of order $10$, we obtain no mutation-finite quivers. Now we use the following statement:

\setcounter{theorem}{2}
\setcounter{section}{6}
\begin{cor}
Suppose that for some $d\ge 7$ there are no non-de\-com\-po\-sable  mutation-finite quivers of order $d$. Then order of any non-decomposable  mutation-finite quiver does not exceed $d-1$.

\end{cor}
Corollary~\ref{p1} implies that there is no non-decomposable mutation-finite quiver of order at least $11$, which completes the proof of the Main Theorem.

\setcounter{section}{3}
\section{Block decompositions of quivers}
\label{blockdecomp}

Let us start with definition of block-decomposable quivers (we re\-phrase Definition 13.1 from~\cite{FST}).

\begin{definition}
\label{defblock}
A {\it block} is a quiver isomorphic to one of the quivers with black/white colored vertices shown on Fig.~\ref{bloki}, or to a single vertex. Vertices marked in white are called {\it outlets}. A connected quiver $S$ is called {\it block-decomposable} if it can be obtained from a collection of blocks by identifying outlets of different blocks along some partial matching (matching of outlets of the same block is not allowed), where two edges with same endpoints and opposite directions cancel out, and two edges with same endpoints and same directions form an edge of weight $2$. A non-connected quiver $S$ is called  block-decomposable either if $S$ satisfies the definition above, or if $S$ is a disjoint union of several mutually orthogonal quivers satisfying the definition above.
If $S$ is not block-decomposable then we call $S$ {\it non-decomposable}. Depending on a block, we call it {\it a block of type} $\rm{I}$, $\rm{II}$, $\rm{III}$, $\rm{IV}$, $\rm{V}$, or simply {\it a block of $n$-th type}.

\begin{figure}[!h]
\begin{center}
\psfrag{1}{$\B_{\rm{I}}$}
\psfrag{2}{$\B_{\rm{II}}$}
\psfrag{3a}{$\B_{\rm{IIIa}}$}
\psfrag{3b}{$\B_{\rm{IIIb}}$}
\psfrag{4}{$\B_{\rm{IV}}$}
\psfrag{5}{$\B_{\rm{V}}$}
\epsfig{file=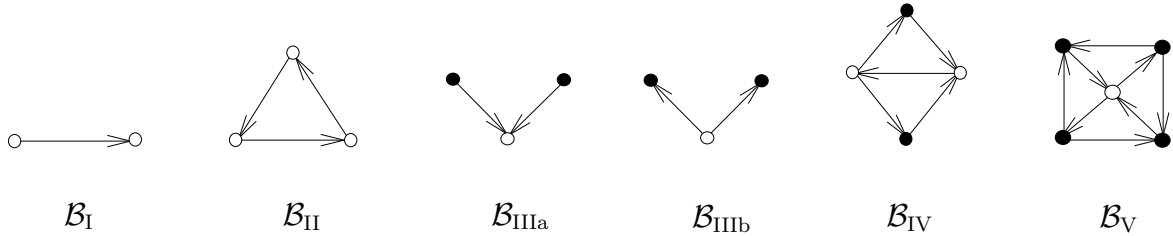,width=0.999\linewidth}
\caption{Blocks. Outlets are colored white, dead ends are black.}
\label{bloki}
\end{center}
\end{figure}

\end{definition}

We denote by $\B_{\rm{I}}$, $\B_{\rm{II}}$ etc. the isomorphism classes of blocks of types $\rm{I}$, $\rm{II}$, etc. respectively. For a block $B$ we write $B\in\B_{\rm{I}}$ if B is of type $\rm{I}$.

\begin{remark}
It is shown in~\cite{FST} that block-decomposable quivers have a nice geometrical interpretation:
they are in one-to-one correspondence with adjacency matrices of arcs of ideal (tagged) triangulations of bordered two-dimensional surfaces with marked points (see~\cite[Section~13]{FST} for the detailed explanations). Mutations of block-decomposable quivers correspond to flips of triangulations.

\begin{figure}[!h]
\begin{center}
\epsfig{file=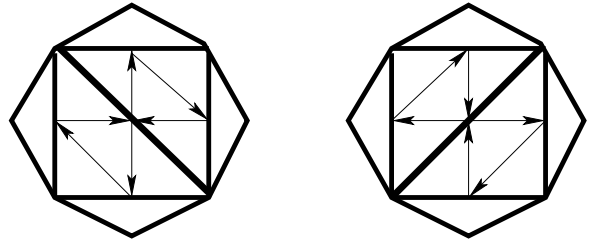,width=0.499\linewidth}
\caption{Fat sides denote arcs of triangulations. Arrows form corresponding quiver.}
\end{center}
\end{figure}

In particular, this description implies that mutation class of any block-decomposable quiver is finite (indeed, the absolute value of an entry of adjacency matrix can not exceed $2$). Another immediate corollary is that any subquiver of a block-decomposable is block-decomposable too.

\end{remark}

\begin{remark}
As the following example shows, a block decomposition of quiver (if exists) may not be unique.
\end{remark}

\begin{example}
There are two ways to decompose an oriented triangle into blocks, see Fig.~\ref{dif_decomposition}.

\begin{figure}[!h]
\begin{center}
\psfrag{b}{\small $B$}
\psfrag{b1}{\small $B_1$}
\psfrag{b2}{\small $B_2$}
\psfrag{b3}{\small $B_3$}
\epsfig{file=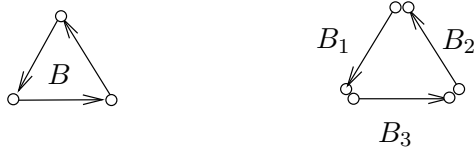,width=0.399\linewidth}
\caption{Two different decompositions of an oriented triangle into blocks}
\label{dif_decomposition}
\end{center}
\end{figure}

\end{example}

We say that a vertex of a block is a {\it dead end} if it is not an outlet.
\begin{remark}\label{rmrk:DeadEndEdge}
Notice that if $S$ is decomposed into blocks, and $u\in S$ is a dead end of some block $B$, then any edge of $B$ incident to $u$ can not cancel with any other edge of another block and therefore must appear in $B$ with weight $1$.
\end{remark}

We call a vertex $u\in S$  an {\it outlet of $S$} if $S$ is block-decomposable and there exists a block decomposition of $S$ such that $u$ is contained in exactly one block $B$, and $u$ is an outlet of $B$. 

We use the following notations. For two vertices $u_i,u_j$ of quiver $S$ we denote by
$(u_i,u_j)$ a directed arc connecting $u_i$ and $u_j$ which may or may not belong to $S$. It may be directed  either way. By $(u_i,u_j,u_k)$ we denote oriented triangle with vertices $u_i,u_j,u_k$ which is oriented either way and whose edges also may or may not belong to $S$.
We use standard notation $\l u_i,u_j\r$ for an edge of $S$.

While drawing quivers, we keep the following notation:

\begin{itemize}
\item
a non-oriented edge is used when orientation does not play any role in the proof;

\item
\psfrag{u}{\tiny $u$}
\psfrag{v}{\tiny $v$}
an edge \ \ \epsfig{file=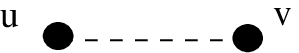,width=0.15\linewidth} \ \ is an edge of a block containing $u$ and $v$, where $u$ and $v$ are not joined in the quiver. The figure assumes a fixed block decomposition;
\
\item
\psfrag{x}{\tiny $x$}
\psfrag{a}{\tiny $a$}
an edge \ \ \epsfig{file=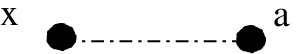,width=0.15\linewidth} \ \ means that $x$ is joined with $a$ by some edge.

\end{itemize}


\begin{prop}
\label{razval1}
Let $S$ be a connected quiver with $n$ vertices, and let $b$ be a vertex of $S$ satisfying the following properties:

$(0)$ $S\setminus b$ is not connected;

$(1)$ for any $u\in S$ the quiver $S\setminus u$ is block-decomposable;

$(2)$ at least one connected component of $S\setminus b$ has at least $3$ vertices;

$(3)$ each connected component of $S\setminus b$ has at most $n-3$ vertices.

\smallskip
\noindent
Then $S$ is block-decomposable.

\end{prop}

\begin{proof}
We divide $S\setminus b$ into two parts $S_1$ and $S_2$ in the following way: $S_1$ is any connected component of $S\setminus b$ with at least $3$ vertices (it exists by assumption $(2)$), and $S_2=S\setminus \l b,S_1\r$. Notice that assumption $(3)$ implies that $|S_2|\ge 2$.

Now choose points $a_1\in S_1$ and $a_2\in S_2$ satisfying the following conditions: $S\setminus a_i$ is connected,
and $S\setminus a_i$ does not contain leaves attached to $b$ and belonging to $S_1$. We always can take as $a_2$ a
vertex of $S_2$ at the maximal distance from $b$. To choose $a_1\in S_1$, we look at the vertices on maximal distance from $b$ in $\l S_1,b\r$. 
If the maximal distance from $b$ in $S_1\cup b$ is greater than $2$, then we may take as $a_1$ any vertex of $S_1$ at the maximal distance from $b$: in this case $\l S_1,b\r\setminus a_1$ does not contain leaves attached to $b$.  If $S_1$ contains a leaf of $S$, then we can take as $a_1$ this leaf. This does not produce leaves of $\l S_1,b\r\setminus a_1$ since $S_1$ is connected and $|S_1|\ge 3$. Finally, if the maximal distance from $b$ in $S_1\cup b$ is $2$, and each vertex on distance $2$ is not a leaf, we take as $a_1$ any neighbour of $b$ with minimal number of neighbours in $S_1$. Again, this does not produce leaves of $\l S_1,b\r\setminus a_1$.


We will prove now that each $\l S_i,b\r$ is block-decomposable with outlet $b$. Since $b$ is the only common vertex of $\l S_1,b\r$ and $\l S_2,b\r$, this will imply that $S$ is block-decomposable.

Consider the quiver $S\setminus a_2$. It is block-decomposable by assumption $(1)$. Choose any its decomposition into blocks. Let us prove that for any block $B$ either $B\cap S_1=\emptyset$ or $B\cap (S_2\setminus a_2)=\emptyset$. In particular, this will imply that $S_1$ is block-decomposable, and $b$ is an outlet (since $|S_2|\ge 2$ and $S\setminus a_2$ is connected). Suppose that for some $B$ both intersections $B\cap S_1$ and $B\cap (S_2\setminus a_2)$ are not empty. We consider below all possible types of block $B$.
Table~\ref{razval1_} illustrates the plan of the proof.

\begin{table}[!h]
\begin{center}
\caption{To the proof of Proposition~\ref{razval1}}
\label{razval1_}
\psfrag{1}{\bf 1}
\psfrag{2}{\bf 2}
\psfrag{3}{\bf 3}
\psfrag{4}{\bf 4}
\psfrag{B=}{$B=$}
\psfrag{or}{\scriptsize or}
\psfrag{b}{\scriptsize $b$}
\epsfig{file=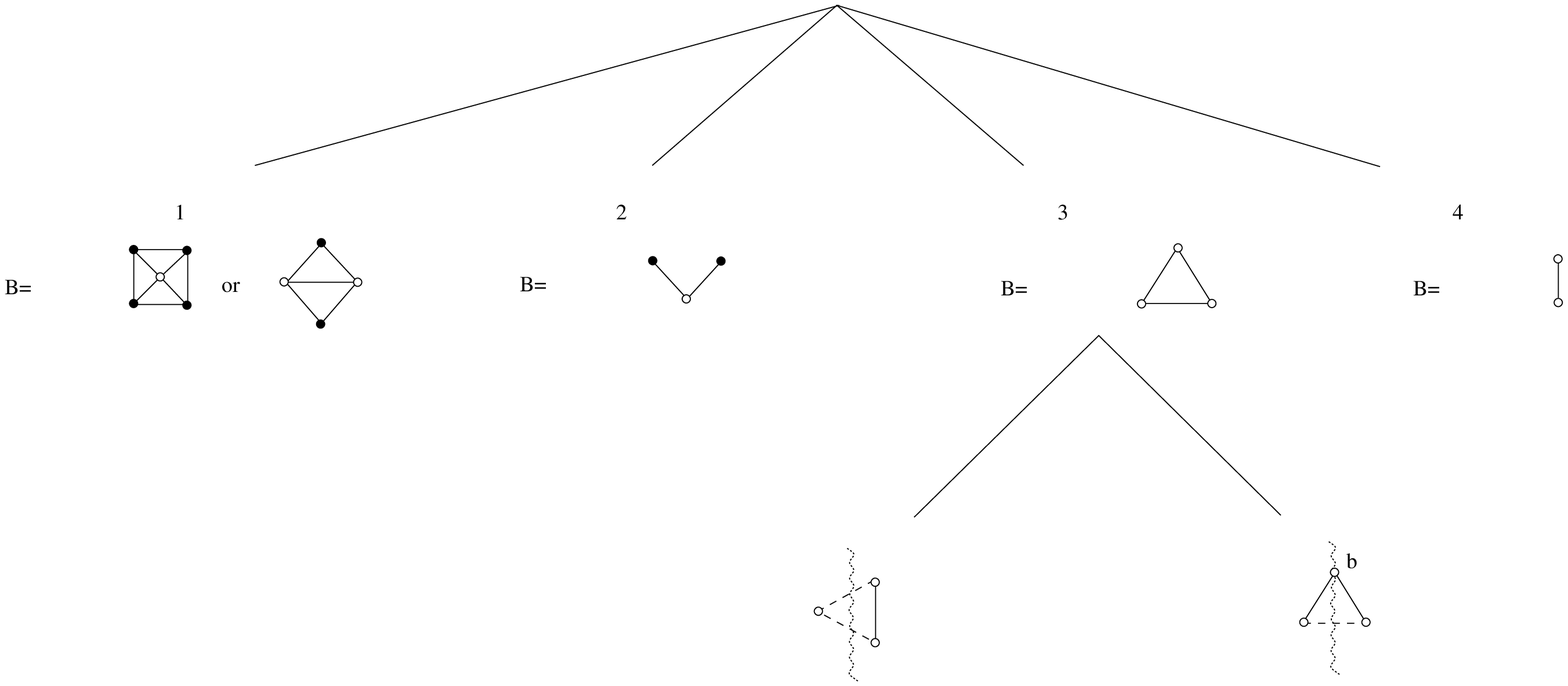,width=1.0\linewidth}
\end{center}
\end{table}

\smallskip
\noindent
{\bf Case 1:} $B\in\B_{\rm{IV}}$, or $B\in\B_{\rm{V}}$. \  At least one edge of $B$ having a dead end runs from $S_1$ to $S_2\setminus a_2$. By Remark~\ref{rmrk:DeadEndEdge} this edge appears in $S$ contradicting assumption $S_1\perp S_2$.

\smallskip
\noindent
{\bf Case 2:} $B\in\B_{\rm{III}}$. \  A unique outlet of $B$ cannot coincide with $b$ due to the way of choosing $a_1$ and $a_2$. Therefore, some edge of $B$ joins vertices of $S_1$ and $S_2\setminus a_2$, which is impossible since both edges of $B$ have dead ends.

\smallskip
\noindent
{\bf Case 3:} $B\in\B_{\rm{II}}$. \ Suppose first that $b$ is not a vertex of $B$. Then one vertex of $B$ (say $w$) belongs to one part of  $S\setminus a_2$ (i.e. in either $S_2\setminus a_2$ or $S_1$), and the remaining two ($w_1$ and $w_2$) to the other.
Since $S$ does not contain edges $(w,w_1)$ and $(w,w_2)$ these edges must cancel out with edges from other blocks of block decomposition. Since all these additional blocks contain $w$, the only way is to attach a block $B_1$ of second type along all the three outlets of $B$. This results in three vertices $w$, $w_1$
and $w_2$ unjoined with all other vertices of $S\setminus a_2$. In particular, $\l S_1,b\r$ is not connected (since $|S_1|\ge 2$), which implies that $S$ is not connected either.

Now suppose that $b$ is a vertex of $B$. Then we may assume that other vertices $w_1$ and $w_2$ of $B$ belong to  $S_1$ and $S_2\setminus a_2$ respectively. $S$ does not contain edge $(w_1,w_2)$. The only way to avoid it in $S$ is to glue an edge $(w_1,w_2)$ (which is a block $B_1$ of the first type) to $B$ (since blocks of all other types are already prohibited by Cases 1,2 and the past of this one). Then $w_1$ is a leaf of $\l S_1,b\r$ attached to $b$ in contradiction with the way of choosing of $a_2$.

\smallskip
\noindent
{\bf Case 4:} $B\in\B_{\rm{I}}$. \ Let $B=(w_1,w_2)$, $w_1\in S_1$ and $w_2\in S_2\setminus a_2$. The only way to avoid this edge in $S$ is to glue another edge $(w_1,w_2)$ (which is a block $B_1$ of the same type) to $B$ (all other blocks are
already prohibited by previous cases). Then $\l S_1,b\r$ is not connected, which implies that $S$ is not connected.

\medskip
Since all the four cases are done, we obtain that $S_1$ is block-decompos\-able with outlet $b$. Considering $S\setminus a_1$ instead of $S\setminus a_2$, in a similar way we conclude that $S_2$ is also block-decomposable with outlet $b$. Gluing these decompositions together along $b$ we obtain a block decomposition of $S$.

\end{proof}

\begin{remark}
In all the situations where we will apply Proposition~\ref{razval1} connectedness of $S$ and assumption $(1)$ will be stated in advance. Usually it is sufficient only to point out the vertex $b$ (in this case we say that $S$ is block-decomposable by Proposition~\ref{razval1} applied to $b$) and all the assumptions are evidently satisfied.
Only in the proofs of Lemma~\ref{no3} and Theorem~\ref{g8} assumption $(3)$ requires additional explanations.
\end{remark}

\begin{prop}
\label{razval2}

Let $S$ be a connected quiver $S=\l S_1,b_1,b_2,S_2\r$, where $S_1\perp S_2$, and $S$ has at least $8$ vertices. Suppose that

$(0)$ $b_1$ and $b_2$ are not joined in $S$;

$(1)$ for any $u\in S$ the quiver $S\setminus u$ is block-decomposable;

$(2)$ there exist $a_1\in S_1,a_2\in S_2$ such that

${}$\phantom{w}  $(2a)$ $S\setminus a_i$ is connected;

${}$\phantom{w}    $(2b)$ either $\l S_i,b_1,b_2\r\setminus a_i$ or $\l S_j,b_1,b_2\r$ (for $i,j=1,2,\ j\ne i$)  contains no leaves attached to $b_1$;

\qquad\quad similarly, either $\l S_i,b_1,b_2\r\setminus a_i$ or $\l S_j,b_1,b_2\r$ (for $j\ne i$) contains no leaves attached to $b_2$;


${}$\phantom{w}  $(2c)$ if $a_i$ is joined with $b_j$ (for $i,j=1,2$), then there is another vertex $w_i\in S_i$ attached to $b_j$.

\smallskip

\noindent
Then $S$ is block-decomposable.

\end{prop}

\begin{proof}
The plan is similar to the proof of Proposition~\ref{razval1}. The idea is to prove that each $S_i$ together with those of $b_1,b_2$ which are attached to $S_i$ is block-decomposable with outlets $b_1$ and $b_2$ (or just one of them if the second is not joined
with $S_i$). Then we combine together these block decompositions to obtain a decomposition of $S$.  First we show that for any block decomposition of  $S\setminus a_2$ any block $B$ is contained entirely either in $\l S_1,b_1,b_2\r$ or in $\l S_2,b_1,b_2\r\setminus a_2$. For this, we consider any block decomposition of $S\setminus a_2$, assuming that for a block $B$ both intersections $B\cap S_1$ and $B\cap (S_2\setminus a_2)$ are not empty. We consider all possible types of block $B$ (see Table~\ref{razval2_}) and obtain contradiction for each type.

\begin{table}[!h]
\begin{center}
\caption{To the proof of Proposition~\ref{razval2}}
\label{razval2_}
\psfrag{1}{\bf 1}
\psfrag{2}{\bf 2}
\psfrag{3}{\bf 3}
\psfrag{4}{\bf 4}
\psfrag{B=}{$B=$}
\psfrag{or}{or}
\psfrag{a2}{\tiny $a_2$}
\psfrag{b1}{\tiny $b_1$}
\psfrag{b2}{\tiny $b_2$}
\psfrag{S1}{\tiny $S_1$}
\psfrag{S2}{\tiny $S_2$}
\psfrag{w1}{\tiny $w_1$}
\psfrag{w2}{\tiny $w_2$}
\psfrag{t1}{\tiny $t_1$}
\psfrag{t2}{\tiny $t_2$}
\psfrag{S1=w1}{\scriptsize $S_1=\{w_1\}$}
\psfrag{perp}{\scriptsize {$B_1$ is the only block  containing $b_2$}}
\psfrag{neperp}{\scriptsize $\exists$ block $B_2\ne B_1: b_2\in B_2$}
\psfrag{3.1}{\small \bf 3.1}
\psfrag{3.2}{\small \bf 3.2}
\psfrag{3.1.}{\scriptsize $B_2\subset \langle b_2,S_1 \rangle$  }
\psfrag{3.2.}{\scriptsize $B_2\subset \langle b_2,S_2\setminus a_2 \rangle$  }
\psfrag{3.1.1}{{\scriptsize \bf 3.1.1}}
\psfrag{3.1.2}{{\scriptsize \bf 3.1.2}}
\psfrag{3.1.1.}{{\scriptsize $t_1$ is a leaf}}
\psfrag{3.1.2.}{{\scriptsize $t_1$ is not a leaf}}
\psfrag{B}{\tiny $B$}
\psfrag{B1}{\tiny $B_1$}
\psfrag{B2}{\tiny $B_2$}
\psfrag{Table}{\ \ \ \ \small{Table~\ref{razv2_32}}}
\epsfig{file=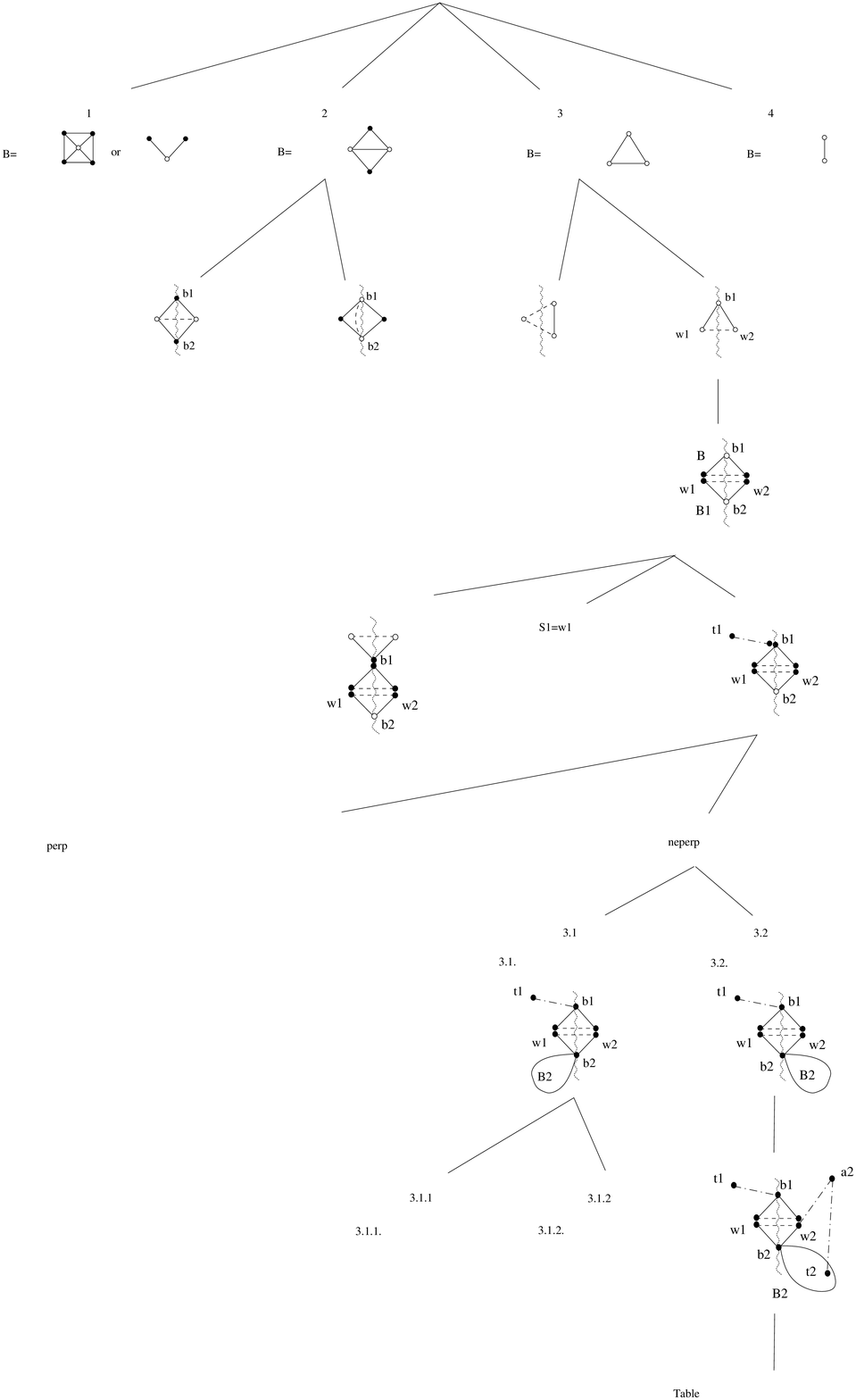,width=0.8\linewidth}
\end{center}
\end{table}

Notice that the assumption $(2c)$ implies that $|S_i|\ge 2$. We will refer this fact as assumption $(3)$.

\medskip
\noindent
{\bf Case 1:} $B\in\B_{\rm{V}}$, or $B\in\B_{\rm{III}}$. \ The proof is the same as in Proposition~\ref{razval1}.

\medskip
\noindent
{\bf Case 2:} $B\in\B_{\rm{IV}}$. \  Evidently, both $b_1,b_2\in B$. We may assume that both $b_1$ and $b_2$
are either dead ends of $B$, or outlets of $B$ (otherwise, by Remark~\ref{rmrk:DeadEndEdge} $S$ contains
simple edge $(b_1,b_2)$  contradicting assumption $(0)$).
First suppose that $b_1$ and $b_2$ are dead ends of $B$. Then we may assume that the
remaining vertices $w_1,w_2$ of $B$ lie in $S_1$ and $S_2\setminus a_2$ respectively. The edge $(w_1,w_2)$ of $B$
must be canceled out by an edge of another block $B_1$, otherwise $(w_1,w_2)$ appears in $S$ contradicting $S_1\perp S_2$.
Block $B_1$ cannot be of type ${\rm{IV}}$ (otherwise none of its vertices is $b_1$, so the block directly connects
$S_1$ and $S_2$). Therefore it is either of type ${\rm{II}}$ or ${\rm{I}}$.
If $B_1$ is of type ${\rm{II}}$ (i.e. a triangle with vertices $w_1,w_2$ and $w$, $w\ne b_i$),
then we note that edges $(w_1, w)$ and $(w,w_2)$ are not canceled out by other blocks and appear in $S$. Therefore
either $(w_1 w)$ or $(ww_2)$ connects $S_1$ to $S_2$, and, hence, contradicts the assumption $S_1\perp S_2$.
If we glue an edge $(w_1,w_2)$ as a block of the first type,
then $w_1$ is not joined with any other vertex of $S_1$. Since $b_1$
and $b_2$ are dead ends of $B$, they are not joined with any other vertex of $S_1$ either. Thus, due to
assumption $(3)$ the subquiver $\l S_1,b_1,b_2 \r$ is not connected, so $S$ is not connected either.

Now suppose that  $b_1$ and $b_2$ are outlets of $B$. Denote by $w_1\in S_1$ and
$w_2\in S_2\setminus a_2$ the other vertices of $B$. To avoid the edge $(b_1,b_2)$ in $S$,
some block should be glued along this edge. If we glue a block of first or forth
type, then we obtain a quiver with respectively $4$ and $6$ vertices without outlets.
Since $S$ has at least $8$ vertices the complement of this quiver in $S\setminus a_2$ is nonempty
and $S\setminus a_2$ is not connected
contradicting the choice of $a_2$. If we glue a block of the second type (a
triangle $b_1b_2w$), then we obtain a quiver with $5$ vertices, the only outlet is $w$.
Therefore, all the remaining vertices of $S\setminus a_2$ are not joined with vertices of $B$.
In particular, $w\in S_1$ (otherwise $S_1$ consists of $w_1$ only), so
$S_2$ consists of $w_2$ and $a_2$. Hence, $S$ is block-decomposable by Proposition~\ref{razval1} applied to $w$.

\medskip
\noindent
{\bf Case 3:} $B\in\B_{\rm{II}}$. \ We may assume that vertices of $B$ are $b_1$, $w_1\in S_1$ and
$w_2\in S_2\setminus a_2$.
Since edge $(w_1,w_2)$ is not in $S$ it must be canceled by a block $B_1$. It is either of type $\rm{I}$ or $\rm{II}$ (since all other types are already excluded above). $B_1$ is not of the first type, otherwise $w_1$ and $w_2$ are
leaves in $S_1$ and $S_2$, resp., attached to $b_1$, contradicting (2b).
Therefore, $B_1$ is of type $\rm{II}$, and the remaining vertex of $B_1$ is either $b_1$ or $b_2$.
If it is $b_1$ then $S$ is not connected.
We conclude that $b_2\in B_1$.

Any other block $B_1'$ with vertex $b_1$ is contained in either $\l S_1,b_1\r$ or $\l b_1,S_2\setminus a_2\r$
(otherwise, if $B_1'$ containing $b_1$ has nonempty intersection with both $S_1$ and $S_2\setminus a_2$,
then $B_1'$ is again of type ${\rm{II}}$. As above there exists $B_2'$ containing $b_2$ completing $B_1'$
in such a way that $\l B, B_1, B_1', B_2'\r$ form a six vertex subquiver without outlets. Since
$|S|\ge 8$, this implies that $S\setminus a_2$ is not connected.)
Similarly, any block with vertex $b_2$ (other than $B_1$) is contained
either in $\l b_1,b_2,S_1\r$ or in $\l b_1,b_2,S_2\setminus a_2\r$.
Further, no vertex of $S\setminus a_2$ except $b_1,b_2$ is joined with $w_1$ and $w_2$.

Therefore, $S\setminus a_2$ consists of $\l b_1,b_2,w_1,w_2\r$, blocks attached to $b_1$ and $b_2$,
and vertices not joined with $\l b_1,b_2,w_1,w_2\r$.
Combining that with assumption $(3)$, we see that there is at least one vertex $t_1\in S_1$
distinct from $w_1$ which is joined with at least one of $b_1$ and $b_2$ by a simple edge.
We may assume that  $t_1$ is attached to $b_1$.
Since $b_1$ is contained in two blocks of $S\setminus a_2$, no vertex of $S_2$ except possibly $a_2$
is joined with $b_1$.

Suppose that no block except $B_1$ contains $b_2$.
Then the subquiver $S\setminus \l b_1,b_2,w_1,w_2,a_2\r$ belongs to $S_1$
and is joined with $b_1$ only. Applying  Proposition~\ref{razval1} to $b_1$ we conclude that $S$ is block-decomposable.

Let us consider the case when some block $B_2$ is attached to $b_2$.
As we have already shown, $B_2$ is entirely contained either in  $\l S_1,b_1,b_2\r$ or in $\l b_1,b_2,S_2\setminus a_2\r$.

\smallskip


\noindent
{\bf Case 3.1:} $B_2$ is contained in $\l S_1,b_2\r$. \
In this case $S_2$ consists of $w_2$ and $a_2$ only,
so $|S_2|=2$. Recall that $w_1$ is joined with $b_1$ and $b_2$ only,
and define new decomposition of $S=\l S'_1,b_1,b_2,S'_2\r$,
where $S'_1=S_1\setminus w_1$, and $S'_2=\l S_2,w_1\r$.
We will show that this decomposition satisfies all the assumptions
of Proposition~\ref{razval2}. Since $|S_2|=2$ and $|S|\ge 8$, this decomposition satisfies $|S'_1|,|S'_2|\ge 3$. Therefore we may avoid Case 3.1.

Clearly, assumptions $(0)$ and $(1)$ hold. We need to choose vertices $a_1'$ and $a_2'$
satisfying conditions $(2a)$--$(2c)$. We keep $a_2'=a_2$.
Recall that $t_1\in S'_1$ is attached to $b_1$, and some vertex of $B_2$ (say $t_2$)
is attached to $b_2$.  So, if $a_2$ is not a leaf of $S$ attached to exactly one of $b_1$ and $b_2$
we may choose as $a_1'$ any vertex  of $S'_1$ different from both $t_1$ and $t_2$
and satisfying condition $(2a)$, see Fig.~\ref{razv2_31}.

\begin{figure}[!h]
\begin{center}
\psfrag{w2}{\tiny $w_2$}
\psfrag{w1}{\tiny $w_1$}
\psfrag{b1}{\tiny $b_1$}
\psfrag{b2}{\tiny $b_2$}
\psfrag{t1}{\tiny $t_1$}
\psfrag{t2}{\tiny $t_2$}
\psfrag{a2}{\tiny $a_2$}
\psfrag{S1}{\small $S_1$}
\psfrag{S2}{\small $S_2$}
\psfrag{S1'}{\small $S_1'$}
\psfrag{S2'}{\small $S_2'$}
\epsfig{file=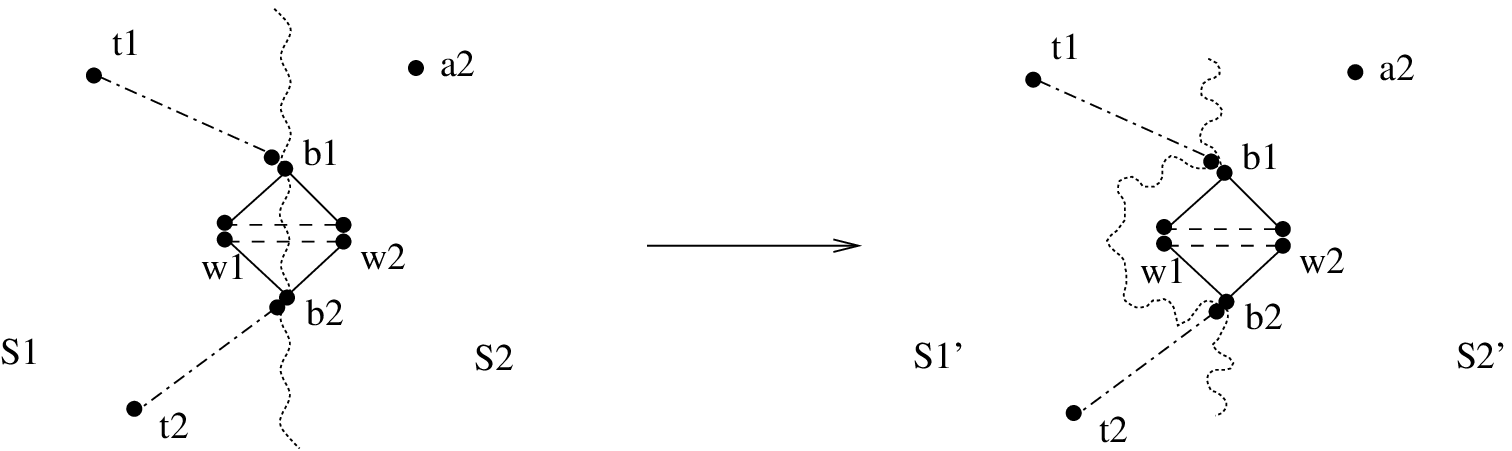,width=0.9\linewidth}
\caption{To the proof of Proposition~\ref{razval2}, Case 3.1}
\label{razv2_31}
\end{center}
\end{figure}

Suppose that  $a_2$ is a leaf of $S$ attached to exactly one of $b_1$ and $b_2$,
say to $b_1$. Notice that there exists at most one leaf of $S$ in $S'_1$ attached
to $b_1$, otherwise assumption $(2b)$ does not hold for the initial decomposition of $S$.
We may assume that if there is a leaf of $S$ in $S'_1$ attached to $b_1$, then it is $t_1$.
Consider the following two cases.

\smallskip
\noindent
{\bf Case 3.1.1:} $t_1$ is a leaf of $S$. \ If there exists another vertex of $S'_1$ attached to $b_1$, then $a_1'=t_1$ satisfies
all the assumptions. So, $t_1$ is the only vertex of $S'_1$ attached to $b_1$. 
Further, as we have seen above, $S\setminus\l t_1,a_2\r$ is block-decomposable with outlet $b_1$,
vertices $t_1$ and $a_2$ are joined with $b_1$ only. Therefore, depending on the orientation of edges $\l b_1,t_1\r$ and $\l b_1,a_2\r$,
we may glue either a block of the third type composed by $b_1,t_1$ and $a_2$,
or a composition of a second type block with an extra edge $(t_1,a_2)$ in appropriate direction, which implies that $S$ is block-decomposable.

\smallskip

\noindent
{\bf Case 3.1.2:} $t_1$ is not a leaf of $S$. \
Denote by $r_1$ any vertex of $S'_1$ attached to $t_1$,
and consider the following quiver $S'\!=\!\l r_1,t_1,b_1,w_1,w_2,a_2\r$.
As a proper subquiver of $S$, $S'$ should be block-decomposable.
However, using the {\rm Java} applet~\cite{K} by Keller one can easily check that
the mutation class of $S'$ is infinite, so $S'$ is non decomposable.

\begin{table}[!h]
\begin{center}
\caption{To the proof of Proposition~\ref{razval2}, Case 3.2}
\label{razv2_32}
\psfrag{1}{\small \bf 3.2.1}
\psfrag{2}{\small \bf 3.2.2}
\psfrag{3}{\small \bf 3.2.3}
\psfrag{4}{\small \bf 3.2.4}
\psfrag{B3}{\tiny $B_3$}
\psfrag{B4}{\tiny $B_4$}
\psfrag{w2}{\tiny $w_2$}
\psfrag{u2}{\tiny $u_2$}
\psfrag{b1}{\tiny $b_1$}
\psfrag{b2}{\tiny $b_2$}
\psfrag{only}{{\scriptsize $B_3$ is the only block}}
\psfrag{containing}{\scriptsize containing $w_2$}
\psfrag{3.2.2.1}{{\scriptsize \bf 3.2.2.1}}
\psfrag{3.2.2.2}{{\scriptsize \bf 3.2.2.2}}
\psfrag{3.2.4.1}{{\scriptsize \bf 3.2.4.1}}
\psfrag{3.2.4.2}{{\scriptsize \bf 3.2.4.2}}
\psfrag{3.2.4.3}{{\scriptsize \bf 3.2.4.3}}
\psfrag{B=}{\scriptsize $B_4=$}
\epsfig{file=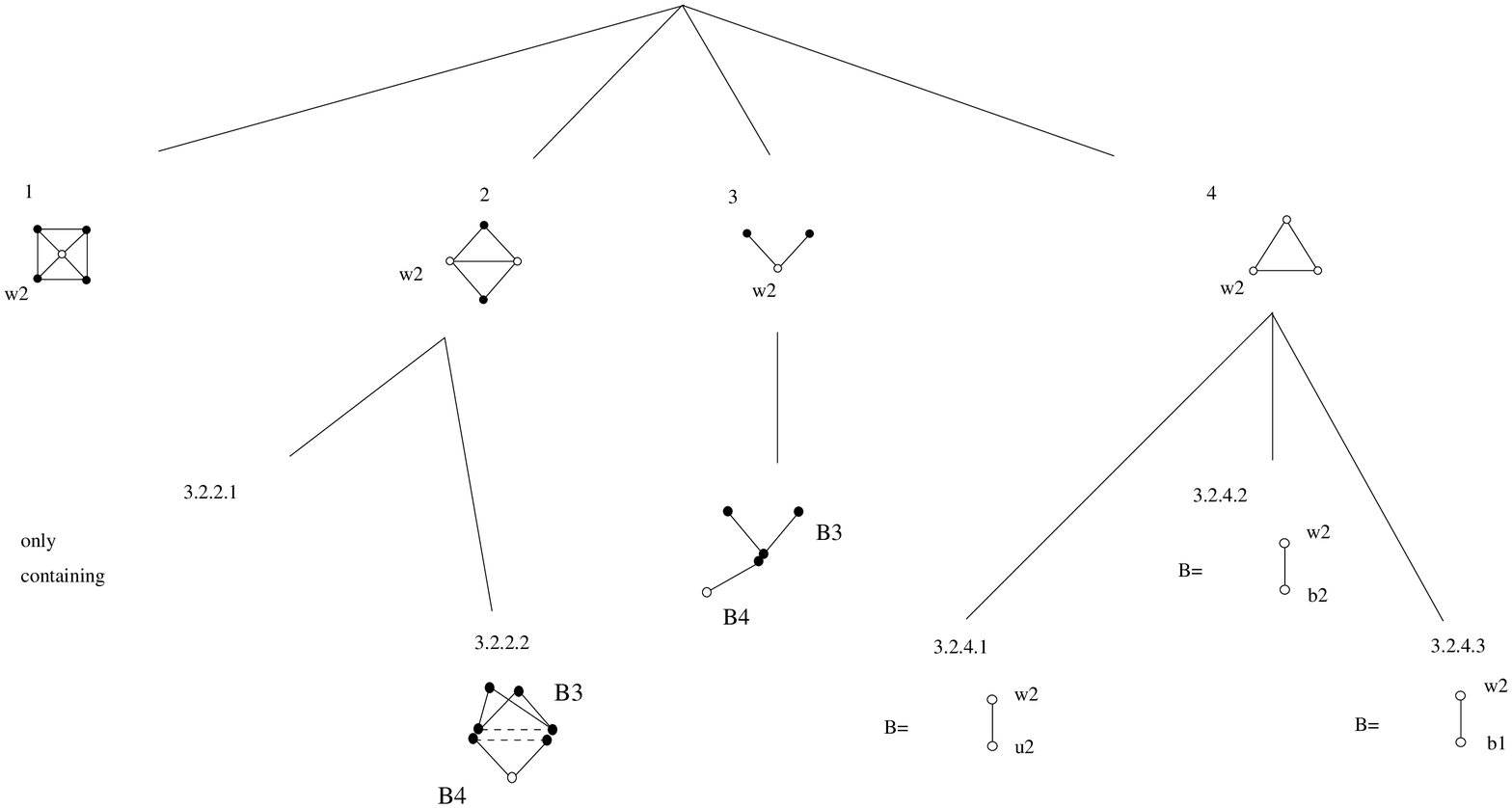,width=1.0\linewidth}
\end{center}
\end{table}

\smallskip
\noindent
{\bf Case 3.2:} $B_2$ is contained in $\l b_2,S_2\setminus a_2\r$. \
If $t_1$ is not a leaf of $S$ then applying Proposition~\ref{razval1} to $b_1$ we see that
$S$ is block-decomposable. Thus, we may assume that  $t_1$ is a leaf of $S$.
In this case $S_1$ consists of $w_1$ and $t_1$ only, so $|S_1|=2$.
If $a_2$ is joined with neither $b_1$ nor $w_2$, then $S$ is block-decomposable by
Proposition~\ref{razval1} applied to $b_2$.
By the same reason, $a_2$ is joined with some vertex $t_2\in S\setminus\l t_1,b_1,w_1,w_2,b_2\r$.
If $a_2$ is not joined with $w_2$, then switching $S_1$ and $S_2$ leads us back to Case~3.1.
Thus, we may assume that $a_2$ is attached to $w_2$ by a simple edge.

Now take any block decomposition of $S\setminus t_1$ and consider all possible types
of blocks containing vertex $w_2$ (see Table~\ref{razv2_32}).
Recall that the valence ${\rm Val}_{S\setminus t_1}(w_2)=3$.

\smallskip
\noindent
{\bf Case 3.2.1:} $w_2$ lies in block $B_3$ of type ${\rm{V}}$. \
In this case $w_2$ is a dead end of $B_3$ (due to its valence), one of $b_1$
and $b_2$ is a dead end of $B_3$, and another one is an outlet
(since the orientations of edges $\l w_2,b_1\r$ and $\l w_2,b_2\r$ differ, see Fig.~\ref{razv2_321}).
Then the edge $(b_1,b_2)$ in $B_3$ is not canceled out by any other edge. Therefore, $b_1$ and $b_2$ are joined in $S$
contradicting assumption $(0)$.

\begin{figure}[!h]
\begin{center}
\psfrag{b1}{\tiny $b_1$}
\psfrag{b2}{\tiny $b_2$}
\psfrag{w1}{\tiny $w_1$}
\psfrag{w2}{\tiny $w_2$}
\psfrag{or}{\small or}
\epsfig{file=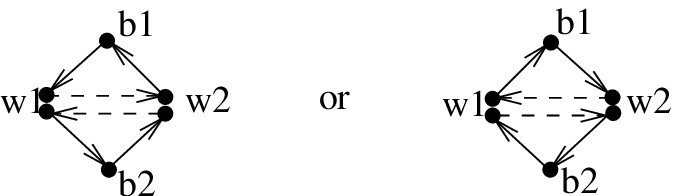,width=0.55\linewidth}
\caption{To the proof of Proposition~\ref{razval2}, Case 3.2.1. Orientations of the edges $\l w_2,b_1\r$ and $\l w_2,b_2\r$ are different.}
\label{razv2_321}
\end{center}
\end{figure}

\smallskip
\noindent
{\bf Case 3.2.2:} $w_2$ is contained in block $B_3$ of type ${\rm{IV}}$. \
In this case $w_2$ is an outlet of $B_3$ (since ${\rm Val}_{S\setminus t_1}(w_2)=3$). Consider two cases.

\smallskip
\noindent
{\bf Case 3.2.2.1:} $w_2$ is contained in block $B_3$ only. \
Then one of $b_1$ and $b_2$ is a dead end of $B_3$, and
another one is an outlet (since the orientations of edges $\l w_2,b_1\r$ and $\l w_2,b_2\r$ are different).
Therefore, $b_1$ is joined with $b_2$, which contradicts  assumption $(0)$.

\smallskip
\noindent
{\bf Case 3.2.2.2:} $w_2$ is contained simultaneously in two blocks $B_3$ and $B_4$, $B_4\ne B_3$. \
Since ${\rm Val}_{S\setminus t_1}(w_2)=3$, $B_4$ is of second type.
Again, the orientations of edges $\l w_2,b_1\r$ and $\l w_2,b_2\r$ are different,
so $a_2$ is a dead end of $B_3$. This implies that valence of $a_2$ is $2$,
so only $t_2$ can be outlet of $B_3$. The second dead end of $B_3$ should be joined
with both $t_2$ and $w_2$. Since $b_1$ is not joined with $t_2$, $b_2$ is a dead end of $B_3$.
But this contradicts existence of the edge joining $b_2$ and $w_1$.

\smallskip
\noindent
{\bf Case 3.2.3:} $w_2$ is contained in block $B_3$ of type ${\rm{III}}$. \
Since ${\rm Val}_{S\setminus t_1}(w_2)=3$, vertex $w_2$ is the outlet of $B_3$. By the same reason at least
one of $b_1$ and $b_2$ is a dead end of $B_3$, hence a leaf of $S\setminus t_1$. But neither $b_1$ nor $b_2$
is a leaf and such $B_3$ does not exist.

\smallskip
\noindent
{\bf Case 3.2.4:} $w_2$ is contained in block $B_3$ of type ${\rm{II}}$. \
Recall that ${\rm Val}_{S\setminus t_1}(w_2)=3$, so in this case $w_2$ is also contained in block $B_4$ of first type.
There are exactly three ways to place vertices $w_2,b_1,b_2,a_2$ into two blocks.

\smallskip
\noindent
{\bf Case 3.2.4.1:} $B_4=(w_2,a_2)$. \
Then $B_3$ has an edge $(b_1,b_2)$ which must cancel out in $S\setminus t_1$.
Since $b_1$ is joined with $w_1$, the only way to cancel out edge $(b_1,b_2)$ in $S\setminus t_1$
is to attach along this edge a second type block $B_5=(b_1,b_2,w_1)$.
Then $b_2\in B_3\cap B_5$, so it must be disjoint from $B_2$. Hence $B_2=\emptyset$.

\smallskip
\noindent
{\bf Case 3.2.4.2:} $B_4=(w_2,b_2)$. \ Then $B_3=(w_2,b_1,a_2)$.
Since $w_1$ is not joined with $a_2$ and ${\rm Val}_{S\setminus t_1}(b_1)\le 3$,
vertices $b_1$ and $w_1$ compose a block $B_5$ of first type.
Since ${\rm Val}_{S\setminus t_1}(w_1)=2$,
vertices $w_1$ and $b_2$ also compose a block $B_6$ of first type.
Again, this implies that $b_2\in B_3\cap B_6$, so $B_2=\emptyset$.

\smallskip
\noindent
{\bf Case 3.2.4.3:} $B_4=(w_2,b_1)$. \ Then $B_3=(w_2,b_2,a_2)$.
The proof is similar to the previous case. Since $w_1$ is not joined with $a_2$, and
${\rm Val}_{S\setminus t_1}(b_1)\le 3$, vertices $b_1$ and $w_1$ compose a block $B_5$
of first type. Since ${\rm Val}_{S\setminus t_1}(w_1)=2$,
vertices $w_1$ and $b_2$ also compose a block $B_6$ of first type.
This implies that $b_2\in B_3\cap B_6$, so again $B_2=\emptyset$.

\medskip
\noindent
{\bf Case 4:} $B\in\B_{\rm{I}}$. \ The proof is the same as in Proposition~\ref{razval1}.

\medskip

We call the connected component of $\l S_i,b_1,b_2\r$ containing $S_i$ \emph{the closure} of $S_i$ and denote it by $\widetilde S_i$.
We proved above that any block in the decomposition of $S\setminus a_2$ is entirely contained in exactly
one of $\widetilde S_1$ and $\widetilde S_2$.
Consider the union of all the blocks with vertices from $\widetilde S_1$ only.
They form a block decomposition either of $\widetilde S_1$,
or of  $\widetilde S_1\cup (b_1,b_2)$, i.e. $\widetilde S_1$ with edge $(b_1,b_2)$.
Due to assumption $(2c)$, in both cases vertices $b_1$ and $b_2$ are outlets.
Similarly, considering a block decomposition of $S\setminus a_1$,
we obtain a block decomposition either of $\widetilde S_2$, or of $\widetilde S_2\cup (b_1,b_2)$
where both $b_1$ and $b_2$ are outlets.

Suppose that in the way described above we got block decompositions
of $\widetilde S_1$ and $\widetilde S_2$.
Then we can glue these decompositions
to obtain a block decomposition of $S$.
Now we will prove that in all other cases $S$ is also block-decomposable.

Now suppose that for one of $\widetilde S_1$ and $\widetilde S_2$ (say $\widetilde S_1$)
we got a block decomposition of $\widetilde S_1$ with an edge $(b_1,b_2)$.
Consider the corresponding decomposition of $S\setminus a_2$.
Clearly, there is a block $B_1$ in the decomposition of $\widetilde S_1$ with an edge $(b_1,b_2)$
containing both $b_1$ and $b_2$ as outlets.
Since  $b_1$ and $b_2$ are not joined in $S$, there exists a block $B_2$
with vertices from $\widetilde S_2$ containing the edge $(b_1,b_2)$,
again both $b_1,b_2$ are outlets. Notice that $B_1$ and $B_2$ are blocks
of second or fourth type (block of third or fifth type has one outlet only;
if $B_i$ is a block of first type then no vertex of $S_i\setminus a_2$
can be attached to $b_j$, so $|S_i|\le 1$).

First, we prove that if $S$ is not block-decomposable then $|S_1|=2$.
Indeed, no vertex of $S_1$ except vertices of $B_1$ can be attached to $b_1$ or $b_2$.
If $B_1$ is of fourth type, then both its vertices belonging to $S_1$ are dead ends,
so $|S_1|=2$. If $B_1$ is of second type with third vertex $v_1$,
then $S$ is block-decomposable by Proposition~\ref{razval1} applied to $v_1$ unless $|S_1|=2$.

Finally, we look at the type of $B_2$.  Again, no vertex of $S_2\setminus a_2$ except vertices of $B_2$
can be attached to $b_1$ or $b_2$.  If $B_2$ is of fourth type, we see that  $|S_2|\le 3$,
so $|S|<8$, which contradicts assumptions of the proposition.
Therefore, we may assume that $B_2$ is of second type with third vertex $v_2$.
In particular, $v_2$ is the only vertex of $S_2\setminus a_2$ joined with $b_1$ and $b_2$.
We will prove that $S$ is block-decomposable.

If $a_2$ is joined with neither $b_1$ nor $b_2$, then $S$ is block-decomposable by
Proposition~\ref{razval1} applied to $v_2$.
So, we may assume that $a_2$ is joined with one of $b_1$ and $b_2$, say $b_1$.
If $a_2$ is joined with no vertex of $S_2\setminus v_2$,
then again $S$ is block-decomposable by  Proposition~\ref{razval1} applied to $v_2$.
Hence, there is $t_2\in S_2\setminus v_2$ attached to $a_2$.
Further, since $|S_1|=2$ and $|S|\ge 8$, there exists a vertex $u_2\in S_2\setminus\l v_2,a_2\r$
joined with $v_2$ (see Fig.~\ref{razv2_e1}), otherwise $S$ is block-decomposable by Proposition~\ref{razval1} applied to $a_2$.

\begin{figure}[!h]
\begin{center}
\psfrag{v2}{\tiny $v_2$}
\psfrag{v1}{\tiny $v_1$}
\psfrag{t2}{\tiny $t_2$}
\psfrag{t2=}{\tiny $t_2=u_2$}
\psfrag{b1}{\tiny $b_1$}
\psfrag{b2}{\tiny $b_2$}
\psfrag{a1}{\tiny $a_1$}
\psfrag{a2}{\tiny $a_2$}
\psfrag{u1}{\tiny $u_1$}
\psfrag{u2}{\tiny $u_2$}
\psfrag{or}{\small or}
\epsfig{file=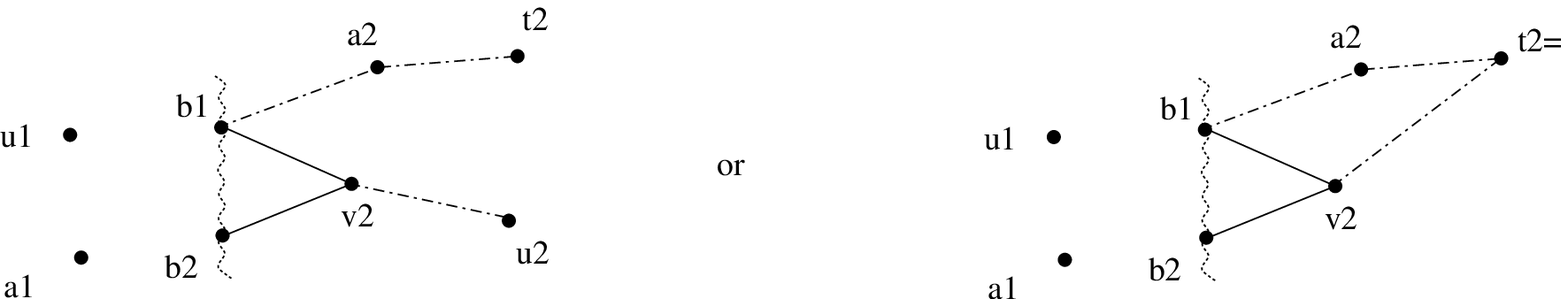,width=0.9\linewidth}
\caption{To the proof of Proposition~\ref{razval2}.}
\label{razv2_e1}
\end{center}
\end{figure}

Take $a_1\in S_1$, and consider a block decomposition of $S\setminus a_1$.
Denote by $u_1$ the remaining vertex of $S_1$, and consider all blocks containing $b_1$.
Since ${\rm Val}_S(b_1)\le 4$ and no block contains vertices from $S_1$ and $S_2$ simultaneously,
$b_1$ does not belong to a block of fifth type.
Since $a_2$ and $v_2$ are not leaves,  $b_1$ does not belong to a block of third type.
Moreover, $b_1$ does not belong to a block of fourth type:
in this case  $a_2$ and $v_2$ are dead ends contradicting existence of $u_2$.

Block $(u_1,b_1,b_2)$ can not exist, otherwise $b_1$ enters three blocks simultaneously.
Therefore, $(b_1,u_1)$ is a block of first type while
$b_1,a_2,v_2$ compose a block of second type (since $S\setminus a_1$ is connected).
Notice that $v_2$ is the only vertex of $S_2\setminus a_2$ joined
with $b_1$ and $b_2$, which implies that either $(b_2,v_2)$ is a block of first type, or $(b_2,v_2,a_2)$ is a block of second type
(see Fig.~\ref{razv2_e2}).
In both cases $v_2$ is contained in two blocks,
so it cannot be attached to $u_2$. This contradiction completes the proof of the Proposition.

\begin{figure}[!h]
\begin{center}
\psfrag{v2}{\tiny $v_2$}
\psfrag{v1}{\tiny $v_1$}
\psfrag{t2}{\tiny $t_2$}
\psfrag{b1}{\tiny $b_1$}
\psfrag{b2}{\tiny $b_2$}
\psfrag{a1}{\tiny $a_1$}
\psfrag{a2}{\tiny $a_2$}
\psfrag{u1}{\tiny $u_1$}
\psfrag{u2}{\tiny $u_2$}
\psfrag{or}{\small or}
\epsfig{file=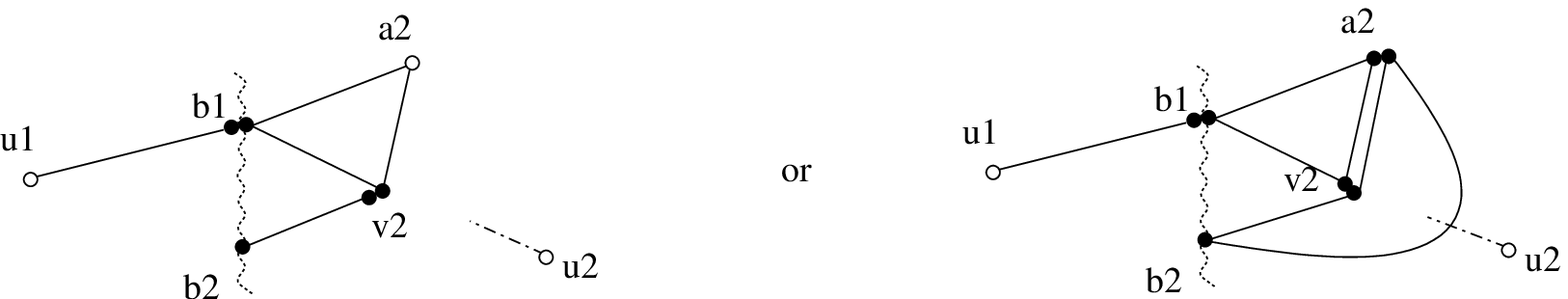,width=0.95\linewidth}
\caption{To the proof of Proposition~\ref{razval2},
$v_2$ belongs to two blocks and cannot be joined with $u_2$.}
\label{razv2_e2}
\end{center}
\end{figure}

\end{proof}

\begin{cor}
\label{after2}
Suppose that $S=\l S_1,b_1,b_2,S_2\r$ satisfies all the assumptions of Proposition~\ref{razval2} except $(2)$. Suppose also that $|S_1|\ge 2,|S_2|\ge 3$, and there exists $c_1\in S_1$ such that the following holds:

$(a)$ $S_1\setminus c_1$ is connected;

$(b)$ $S_1$ contains no leaves of $S$ attached to  $b_1$ or $b_2$,  and $S_1\setminus c_1$ contains no leaves of $S\setminus c_1$ attached to  $b_1$ or $b_2$;

$(c)$ $S_1\setminus c_1$ is attached to both $b_1$ and $b_2$.

Then $S$ is block-decomposable.

\end{cor}

\begin{proof}
We will show how to choose $a_1$ and $a_2$ to fit into assumption $(2)$ of  Proposition~\ref{razval2}.

As $a_1$ we can always take $c_1$. Clearly, assumptions $(a)$--$(c)$ imply corresponding assumptions $(2a)$--$(2c)$ of the Proposition~\ref{razval2} for $a_1$.

To choose $a_2$, we look how $S_2$ is attached to $b_1$ and $b_2$. If either $S_2$ is not attached to one of them (say $b_2$) or there is a vertex $v_2\in S_2$ joined with both $b_1$ and $b_2$, then we take as $a_2$ any vertex of $S_2$ being at the maximal distance from $b_1$; otherwise, we fix two vertices $v_1, v_2\in S_2$ joined with $b_1$ and $b_2$ respectively, and take as $a_2$ any vertex of $S_2\setminus\l v_1,v_2\r$ being at the maximal distance from $b_1$.

\end{proof}

\bigskip

\section{Minimal non-decomposable quivers}
\label{minimal}

Our aim is to prove that any non-decomposable quiver contains a subquiver of relatively small order which is non-decomposable either.

\begin{definition}
\label{defmin}
A \emph{minimal non-decomposable quiver} $S$ is a quiver that
\begin{itemize}
\item is non-decomposable;
\item for any $u\in S$ the quiver $S\setminus u$ is block-decomposable.
\end{itemize}
\end{definition}
Notice that a minimal non-decomposable quiver is connected. Indeed, if $S$ is non-connected and non-decomposable, then at least one connected component of $S$ is non-decomposable either.

\begin{theorem}
\label{g8}

Any minimal non-decomposable quiver contains at most $7$ vertices.

\end{theorem}

The plan of the proof is the following. We assume that there exists a quiver $S$ of order at least $8$ satisfying the assumptions of Theorem~\ref{g8}, and show for each type of block that if a block-decomposable subquiver $S
\setminus u$ contains block of this type then $S$ is also block-decomposable.

Throughout this section we assume that $S$ satisfies the assumptions of Theorem~\ref{g8}. Here we emphasize that we do not assume the mutation class of $S$ to be finite.

A {\it link} $L_S(v)$ of vertex $v$ in $S$ is a subquiver of $S$ spanned by all neighbors of $v$. If $S$ is block-decomposable, we introduce for a given block decomposition a quiver $\Theta_{S}(v)$ obtained by gluing all blocks either containing $v$ or
having at least two points in common with $L_{S}(v)$. Notice that $\Theta_{S}(v)$ may not be a subquiver of $S$.
Clearly, $L_{S}(v)$ is a subquiver of $\Theta_{S}(v)$ for any block decomposition of $S$.

\begin{lemma}
\label{no5}
For any $x\in S$ any block decomposition of $S\setminus x$ does not contain blocks of type ${\rm{V}}$.

\end{lemma}

To prove the lemma we use the following proposition.

\begin{prop}
\label{5is5}
Suppose that $S\setminus x$ contains a subquiver $S_1$ consisting of a block $B$ of type ${\rm{V}}$ (with dead ends $v_1,\dots,v_4$ and outlet $v$) and a vertex $t$ joined with $v$ (and probably with some of $v_i$). Then for any $u\in S\setminus S_1$
and any block decomposition of $S\setminus u$ a subquiver $\l v,v_1,\dots,v_4\r$ is contained in one block of type ${\rm{V}}$. In particular, $t$ does not attach to any of $v_i$, $i=1,\dots,4$.

\end{prop}

\begin{proof}
Take any $u\in S\setminus S_1$ and consider any block decomposition of $S_2=S\setminus u$. Since valence of $v$ in $S_2$ is at least $5$, $v$ is contained in exactly two blocks $B_1$ and $B_2$, at least one of which is of the type ${\rm{V}}$ or ${\rm{IV}}$. Suppose that none of $B_1$ and $B_2$ is of the type ${\rm{V}}$, and let $B_1$ be of the type ${\rm{IV}}$. Then for any choice of $B_2$ the number of vertices of $S_2$ which are neighbors of $v$ and  have valence at least three in $S_2$ does not
exceed $3$. However, there are at least four such vertices $v_1,\dots,v_4$, so the contradiction implies that we may assume $B_1$ to be of the type ${\rm{V}}$ with outlet $v$.

Since block $B$ of $S\setminus x$ is of type ${\rm{V}}$, the subquiver $\l v_1,v_2,v_3,v_4\r\subset S$ is a cycle. At the same time, the link $L_{S_2}(v)$ is a disjoint union of a cycle of order $4$ (composed by dead ends of $B_1$) and another
quiver with at most $4$ vertices (composed by vertices of $B_2\setminus v$). If we assume that $v_1,v_2,v_3,v_4$ are not contained in one block $B_1$ (or $B_2$) in $S_2$, then $v_1,v_2,v_3,v_4$ do not compose a cycle, and we come to a
contradiction. To complete the proof it is enough to notice that only block of type ${\rm{V}}$ contains chordless cycle of length $4$.

\end{proof}

\begin{proof}[Proof of Lemma~\ref{no5}]
Suppose that a block decomposition of $S\setminus x$ contains a block $B$ of type ${\rm{V}}$ with dead ends $v_1,\dots,v_4$ and outlet $v$. Consider two cases:
either ${\rm{Val}}_S(v)\ge 5$ or ${\rm{Val}}_S(v)=4$.

\smallskip
 \noindent
{\bf Case 1:} ${\rm Val}_S(v)\ge 5$. \ Then there exists $u\in S$, $u\notin B$ that is joined with $v$.
Denote $S_1=\l v, v_1,v_2,v_3,v_4,u\r$, and consider any block decomposition of
$S_2=S\setminus w_2$  for any $w_2\notin S_1$.
By Proposition~\ref{5is5}, $\l
v,v_1,v_2,v_3,v_4\r$ form a block $B_1$ of type ${\rm{V}}$ with the outlet $v$.
Therefore, no vertex of $\l u, S_2\setminus S_1\r $ is joined with $v_1,v_2,v_3,v_4$.
Since $|S|\ge 8$, we have $S\setminus \l w_2,S_1\r\ne \emptyset$.
Consider a block decomposition of $S_3=S\setminus w_3$ for some $w_3\in S\setminus\l S_1,w_2\r$.
No vertex of $S_3\setminus S_1$ is joined with $v_1,v_2,v_3,v_4$.
In particular, we obtain that none of  $w_2,\ w_3,\ u$ is joined with $v_1,v_2,v_3,v_4$.
Moreover, since $w_2$ and $w_3$ are arbitrary vertices of $S\setminus S_1$,
this implies that no vertex of $\l u,S\setminus S_1\r$ is joined with $v_1,v_2,v_3,v_4$.
Thus,  $S$ is block-decomposable by Proposition~\ref{razval1} applied to $v$.

\smallskip
\noindent
{\bf Case 2:} ${\rm Val}_S(v)=4$. \ Fix a block decomposition of $S\setminus x$ containing $B$.
Since ${\rm Val}_S(v)=4$ and $v_1,\dots,v_4$ are dead ends,
no vertex of $S\setminus\{x\cup B\}$ is joined with vertices of $B$. Again, $B$ is block-decomposable by Proposition~\ref{razval1} applied to $x$.

\end{proof}

\begin{lemma}
\label{no4}
For any $x\in S$ no block decomposition of $S\setminus x$ contains blocks of type ${\rm{IV}}$.

\end{lemma}

\begin{proof}
Consider any block decomposition of $S\setminus x$.
Any vertex is contained in at most two blocks.
By Lemma~\ref{no5}, any block decomposition of $S\setminus x$ does not contain
blocks of type ${\rm{V}}$.
This implies that the valence
${\rm Val}_{S\setminus x}(v)$ does not exceed $6$ for any $v\in S\setminus x$.
Thus, ${\rm Val}_{S}(v)\le 8$ (recall that any proper subquiver of $S$, and therefore $S$ itself, does not contain edges of multiplicity greater than $2$ due to Theorem~\ref{less3}).

Now let $B$ be a block of type ${\rm{IV}}$ in some block decomposition of $S\setminus x$,
denote by $v_1$ and $v_2$ the outlets of $B$,
and assume that ${\rm Val}_{S\setminus x}(v_2)\le {\rm Val}_{S\setminus x}(v_1)\le 6$.
If ${\rm Val}_{S\setminus x}(v_2)= {\rm Val}_{S
\setminus x}(v_1)$, we assume that ${\rm Val}_{S}(v_2)\le {\rm Val}_{S}(v_1)$.
We analyze the situation case by case with respect to the valence ${\rm{Val}}_{S\setminus x}(v_1)$ decreasing.
Each case splits in two: either $v_1$ is joined with
$x$ or not (see Table~\ref{bl4}).

\begin{table}[!h]
\begin{center}
\caption{To the proof of Lemma~\ref{no4}}
\label{bl4}
\psfrag{1}{\bf 1}
\psfrag{1.}{\scriptsize ${\rm{Val}}_{S\setminus x}(v_1)=6$}
\psfrag{2}{\bf 2}
\psfrag{2.}{\scriptsize ${\rm{Val}}_{S\setminus x}(v_1)=5$}
\psfrag{3}{\bf 3}
\psfrag{3.}{\scriptsize ${\rm{Val}}_{S\setminus x}(v_1)=4$}
\psfrag{4}{\bf 4}
\psfrag{4.}{\scriptsize ${\rm{Val}}_{S\setminus x}(v_1)=3$}
\psfrag{5}{\bf 5}
\psfrag{5.}{\scriptsize ${\rm{Val}}_{S\setminus x}(v_1)=2$}

\psfrag{1.1}{\small \bf 1.1}
\psfrag{1.1.} {\scriptsize $x\not\perp v_1$}
\psfrag{1.2.} {\scriptsize $x\perp v_1$}

\psfrag{1.2}{\small \bf 1.2}
\psfrag{2.1}{\small \bf 2.1}
\psfrag{2.2}{\small \bf 2.2}
\psfrag{3.1}{\small \bf 3.1}
\psfrag{3.2}{\small \bf 3.2}
\psfrag{4.1}{\small \bf 4.1}
\psfrag{4.2}{\small \bf 4.2}

\psfrag{1.1.1}{\scriptsize {\bf 1.1.1}}
\psfrag{1.1.2}{\scriptsize {\bf 1.1.2}}
\psfrag{1.2.1}{\scriptsize {\bf 1.2.1}}
\psfrag{1.2.2}{\scriptsize {\bf 1.2.2}}
\psfrag{2.2.1}{\scriptsize {\bf 2.2.1}}
\psfrag{2.2.2}{\scriptsize {\bf 2.2.2}}
\psfrag{3.2.1}{\scriptsize {\bf 3.2.1}}
\psfrag{3.2.2}{\scriptsize {\bf 3.2.2}}
\psfrag{4.2.1}{\scriptsize {\bf 4.2.1}}
\psfrag{4.2.2}{\scriptsize {\bf 4.2.2}}

\psfrag{2.2.2.1}{\tiny {\bf 2.2.2.1}}
\psfrag{2.2.2.2}{\tiny {\bf 2.2.2.2}}
\psfrag{B2}{\tiny $B_2$}
\psfrag{w3}{\tiny $w_3$}
\psfrag{w4}{\tiny $w_4$}
\psfrag{v1}{\tiny $v_{1}$}
\psfrag{v2}{\tiny $v_2$}
\psfrag{v3}{\tiny $v_3$}
\psfrag{u1}{\tiny $u_1$}
\psfrag{u2}{\tiny $u_2$}
\psfrag{x}{\tiny $x$}
\psfrag{or}{\tiny or}
\psfrag{v=v}{\tiny $v_3=v_2$}
\psfrag{vv}{\tiny $v_3\ne v_2$}
\psfrag{v=u}{\tiny $u_2=v_2$}
\psfrag{vu}{\tiny $u_2\ne v_2$}
\psfrag{n}{\tiny no block contains $u_1, u_2, v_2$}
\psfrag{sim}{\tiny simultaneously}
\psfrag{c1}{\tiny $B_2$ contains $u_1, u_2, v_2$}
\epsfig{file=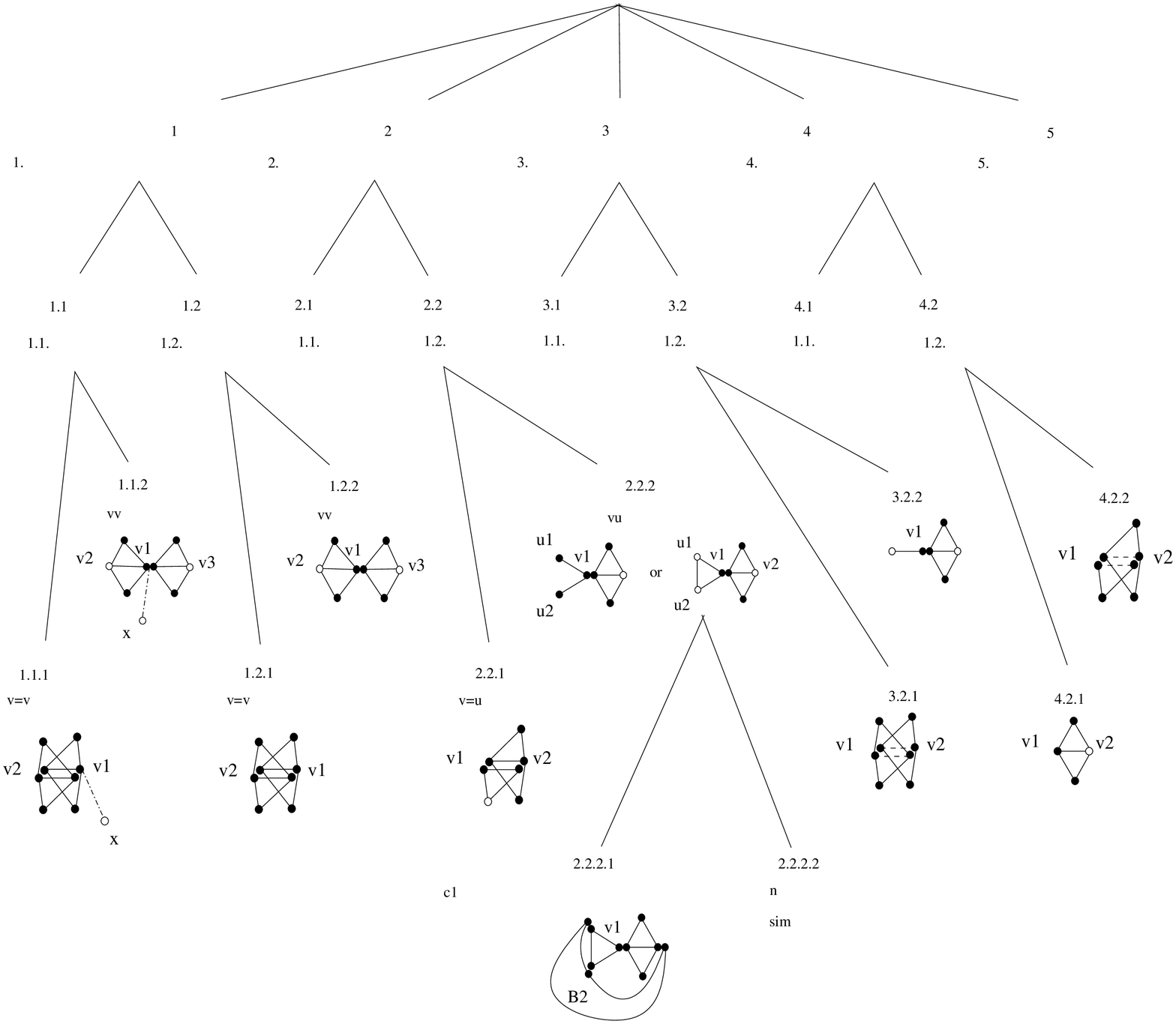,width=1.0\linewidth}
\end{center}
\end{table}

Notice that $x$ may be joined with $v_1$ by a simple edge only.
Indeed, suppose that $x$ is joined with $v_1$ by a double edge.
Denote by $w_1$ and $w_2$ dead ends of $B$.
Since the subquiver $\l x,v_1,w_1 \r$ is decomposable and  its mutation class is finite,
$x$ is joined with $w_1$ and the triangle composed by $x,v_1$ and $w_1$ is oriented (this follows from the fact that
 the mutation class of a chain of a double edge and a simple edge is infinite independently
of the orientations of edges, and the mutation class of a non-oriented triangle containing
a double edge is also infinite). By the same reason, $v_2$ is joined with both $x$ and $v_1$, and
the triangle composed by $x,v_1$ and $v_2$ is oriented.
Thus, directions of edges $\l v_1, w_1\r$ and $\l v_1, v_2\r$ induce opposite orientations of edge $\l x, w_1\r$
(see Fig.~\ref{bl4_}). Obtained contradiction shows that ${\rm Val}_{S}(v_1)\le {\rm Val}_{S\setminus x}(v_1)+1$.

\begin{figure}[!hb]
\begin{center}
\psfrag{v1}{\tiny $v_1$}
\psfrag{v2}{\tiny $v_2$}
\psfrag{w1}{\tiny $w_1$}
\psfrag{w2}{\tiny $w_2$}
\psfrag{x}{\tiny $x$}
\psfrag{or}{\small or}
\epsfig{file=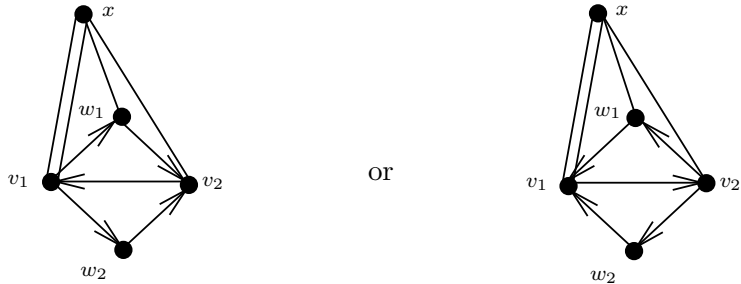,width=0.65\linewidth}
\caption{To the proof of Lemma~\ref{no4}. One of the triangles $xv_1v_2$ and $xv_1w_1$ is non-oriented.}
\label{bl4_}
\end{center}
\end{figure}

\smallskip
\noindent
{\bf Case 1:} ${\rm Val}_{S\setminus x}(v_1)=6$. \

\smallskip
\noindent
{\bf Case 1.1:}   $x\not\perp v$, hence ${\rm Val}_S(v_1)=7$. \
Then $v_1$ is contained in block $B_1$ of type ${\rm{IV}}$,
and $v_1$ is joined with $x$. Denote the dead ends of $B_1$ by $w_3$ and $w_4$,
and the remaining outlet by $v_3$.

\smallskip
\noindent
{\bf Case 1.1.1:} $v_3$ coincides with $v_2$. \
Since ${\rm Val}_S(v_1)=7$, $v_1$ and $v_2$ are joined by a double edge.
Consider $S_1=S\setminus w_1$ with some block decomposition. Clearly, ${\rm Val}_{S_1}(v_1)=6$,
and subquiver $\l v_1,v_2,x,w_2,w_3,w_4\r\subset S_1$ is obtained by
gluing two blocks of fourth type along the edge $(v_1,v_2)$.
In particular, $x$ is a dead end of one of these blocks,
so $x$ is joined with $v_2$ and is not joined with $w_2,w_3,w_4$.
Similarly, considering $S_2=S\setminus w_2$, we see that $x$ is not joined with $w_1$ either.

Now consider subquiver $S'=\l v_1,x,w_1,w_2,w_3,w_4\r$.
The only edges in $S'$ are those joining $v_1$ with other vertices.
Using Keller's applet~\cite{K} we can check that mutation class of $S'$ is infinite.
Recall that $S'$ is a subquiver of block-decomposable quiver and, hence,  is
block-decomposable itself. Therefore, its mutation class must be finite, and obtained
contradiction eliminates the considered case.

\smallskip
\noindent
{\bf Case 1.1.2:} $v_3$ does not coincide with $v_2$. Consider $S_1=S\setminus v_3$. Since ${\rm Val}_{S_1}(v_1)=6$, subquiver $\l v_1,v_2,x,w_1,w_2,w_3,w_4\r\subset S_1$ is obtained by gluing of two blocks of fourth type at $v_1$. Vertices $w_3$ and
$w_4$ are not joined, so $x$ is an outlet of block $\l v_1,x,w_3,w_4\r$. In particular, $x$ is joined with $w_3$ and $w_4$ and is not joined with $w_1$ and $w_2$. Now, considering $S_2=S\setminus v_2$, we obtain in a similar way that $x$ is joined with
$w_1$ and $w_2$ and is not joined with $w_3$ and $w_4$, so we come to a contradiction.

\smallskip
\noindent
{\bf Case 1.2:} $x\perp v_1$, hence ${\rm Val}_S(v_1)=6$. \ As in the previous case, $v_1$ is contained in block $B_1$ of type ${\rm{IV}}$. Denote the dead ends of $B_1$ by $w_3$ and $w_4$, and the remaining outlet by $v_3$.

\smallskip
\noindent
{\bf Case 1.2.1:} $v_3$ coincides with $v_2$. Since ${\rm Val}_S(v_1)=6$, $v_1$ and $v_2$ are joined by a double edge. The only outlets of $B$ and $B_1$ are $v_1$ and $v_2$, and they are already contained in two blocks each. Therefore, the quiver
spanned by $B$ and $B_1$ has no outlets, so any other vertex of $S$ except $x$ is not joined with $S'=\l v_1,v_2,w_1,w_2,w_3,w_4\r$. Now take any vertex $u$ distinct from $x$ which is not contained in $B$ or $B_1$, and consider any block
decomposition of $S_1=S\setminus u$. Since ${\rm Val}_{S_1}(v_1)={\rm Val}_{S_1}(v_2)=6$, the subquiver $S'\subset S_1$ is again obtained by gluing of two blocks of fourth type along the edge $(v_1,v_2)$. By the reasons described above, no vertex of $S$
except $u$ is joined with vertices of $S'$. This implies that neither $x$ nor $u$ is attached to $S'$, so no vertex of $S\setminus S'$ is attached to $S'$ and $S$ is not connected.

\smallskip
\noindent
{\bf Case 1.2.2:} $v_3$ does not coincide with $v_2$.
Recall that none of vertices $w_1,w_2,w_3,w_4$ is joined with vertices of
$S\setminus\{x\cup (B_1\cap B_2)\}$.
Let us prove that they are not joined with $x$ either.

There are two options for link $L_{S\setminus x}(v_1)$:
either $v_2$ is joined with $v_3$ or not. Notice that if two vertices
$u$ and $v$ are dead ends of a block of type ${\rm{IV}}$
with outlet $v_1$ in some block decomposition of any quiver $S'$, then $u$ and $v$
are leaves of the link $L_{S'}(v)$, and the distance between them equals two.

Consider the quiver $S_1=S\setminus w_1$ with some block decomposition.
Since ${\rm Val}_{S_1}(v_1)=5$, subquiver $\l v_1,v_2,v_3,w_2,w_3,w_4\r\subset S_1$
is obtained by gluing a block of fourth type and
a block of second or third type at $v_1$. Looking at
the link $L_{S_1}(v_1)$, we see that there is only one pair
of leaves at distance two, namely $w_3$ and $w_4$.
Therefore, vertices $v_1,v_3,w_3,w_4$ are contained in one block of fourth type.
In particular, $x$ is not joined with $w_3$ and $w_4$.
Similarly, considering $S_2=S\setminus w_3$,
we obtain that $x$ is not joined with $w_1$ and $w_2$.

Let us take another one look at block decomposition of $S_1$.
Since $w_2$ is joined with $v_2$, vertices $v_1,v_2,w_2$ are
contained in a block of second type, i.e. triangle.
Since none of $w_1$ and $w_2$ is joined with any other vertex of $S$ than $v_1$ and $v_2$
we can replace triangle $v_1v_2w_2$ by block $(v_1,v_2,w_1,w_2)$
of type ${\rm{IV}}$ to obtain a block decomposition of $S$.

\medskip
\noindent
{\bf Case 2:} ${\rm Val}_{S\setminus x}(v_1)=5$. \

\smallskip
\noindent
{\bf Case 2.1:} $x\not\perp v_1$, ${\rm Val}_S(v_1)=6$. \  Since $  |S|\ge 8$, there exists $y\in S$ which is not joined with $v_1$. Since  ${\rm Val}_{S\setminus y}(v_1)=6$, in any block decomposition of $S\setminus y$ vertex $v_1$ is an outlet of a block  of type ${\rm{IV}}$, so we may refer to Case~1.2.

\smallskip
\noindent
{\bf Case 2.2:} $x\perp v_1$, ${\rm Val}_S(v_1)=5$.
\  Vertex $v_1$ is contained in block $B$ and
in block $B_1$ of type ${\rm{II}}$ or ${\rm{III}}$.
Vertices $w_1$ and $w_2$ are dead ends of $B$,
so they are joined in $S\setminus x$ with $v_1$ and $v_2$ only.

Denote by $u_1$ and $u_2$ the remaining vertices of $B_1$, and consider the quiver $\Theta_{S\setminus x}(v_1)$. Since the union of $B$ and $B_1$ has at most $3$ outlets, $\Theta_{S\setminus x}(v_1)$ consists of blocks $B$, $B_1$ and probably
some block $B_2$ containing at least two of vertices $v_2,u_1$ and $u_2$. Consider the following three cases.

\smallskip
\noindent
{\bf Case 2.2.1:} Vertex $v_2$ coincides with $u_2$.
In this case $v_1$ and $v_2$ are joined by a double edge,
$B_1$ is a block of second type (which implies
that $u_1$ is joined with $v_2$, so ${\rm Val}_S(v_2)\ge 5$),
and the union of $B$ and $B_1$ has a unique outlet $u_1$.
Thus, $\l L_{S}(v_1),v_1\r\setminus u_1$ may be joined with $x$ only.
Since ${\rm Val}_S(v_2)\le {\rm Val}_S(v_1)=5$, $x$ is not joined
with $v_1$ and $v_2$. If $x$ is not joined with $w_1$ and $w_2$ either,
then we can apply Proposition~\ref{razval1} to $u_1$.
Therefore, we may assume that $x$ is joined with at least one of $w_1$ and $w_2$, say $w_1$.

Now take any $y\notin\l B,u_1,x\r$ and consider $S_1=S\setminus y$ with some block decomposition.
Recall that $y$ cannot be joined with any vertex of $B$.
Since ${\rm Val}_{S_1}(v_1)=5$, $v_1$ is contained
in some fourth type block of this decomposition together with $v_2$
and two of $w_1,w_2,u_1$. But $w_1$ is joined with $x$,
so it cannot be a dead end of the block.
Hence, $v_1,v_2,w_2,u_1$ compose a block of type ${\rm{IV}}$
with outlets $v_1,v_2$ and dead ends $w_2,u_1$.
In particular, neither $w_2$ nor $u_1$ is attached to $x$.
If $y$ is not joined with $u_1$, then we can apply Proposition~\ref{razval1} to $w_1$.
Therefore, we may assume that $y$ is joined with $u_1$.

By assumption, $|S|\ge 8$. Thus, we can take a vertex $z$ which does not coincide
with any of preceding ones, and consider $S_2=S\setminus z$ (see Fig.~\ref{bl4_221}).
As explained above, any block decomposition of
subquiver $\l B,u_1\r$ is a union of two blocks with one outlet only.
However, $w_1$ is joined with $x$, and $u_1$ is joined with $y$,
so we come to a contradiction.

\begin{figure}[!h]
\begin{center}
\psfrag{v1}{\tiny $v_1$}
\psfrag{v2}{\tiny $v_2$}
\psfrag{w1}{\tiny $w_1$}
\psfrag{w2}{\tiny $w_2$}
\psfrag{x}{\tiny $x$}
\psfrag{u1}{\tiny $u_1$}
\psfrag{y}{\tiny $y$}
\psfrag{z}{\tiny $z$}
\epsfig{file=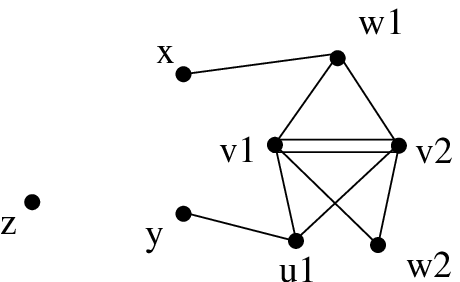,width=0.4\linewidth}
\caption{To the proof of Lemma~\ref{no4}, Case 2.2.1}
\label{bl4_221}
\end{center}
\end{figure}

\smallskip
\noindent
{\bf Case 2.2.2:} Neither $u_1$ nor $u_2$ coincides with $v_2$.
The case also splits into the following two.

\smallskip
\noindent
{\bf Case 2.2.2.1:} $\Theta_{S\setminus x}(v_1)$ consists of three blocks $B$,
$B_1$, $B_2$, where $B_2$ is a triangle with vertices $v_2,u_1$ and $u_2$.
\ In this case the quiver $\Theta_{S\setminus x}
(v_1)$ has no outlets, so no vertex of $S\setminus\Theta_{S\setminus x}(v_1)$
except $x$ is joined with $\Theta_{S\setminus x}(v_1)$.

Take any vertex $y\in S\setminus\Theta_{S\setminus x}(v_1)$ distinct from $x$,
and consider quiver $S_1=S\setminus y$. The quiver $\Theta_{S\setminus y}(v_1)$
is spanned by $\Theta_{S\setminus x}(v_1)$ and probably $x$.
Since ${\rm Val}_{S_1}(v_1)=5
$, three of vertices $w_1,w_2,v_2,u_1,u_2$
should compose a block $B'$ of type ${\rm{IV}}$ together with $v_1$.
Since none of $w_1,w_2$ is joined with any of $u_1,u_2$,
that block contains $v_2$. Vertex $v_1$ is also contained in some block $B''$
of type ${\rm{II}}$ or ${\rm{III}}$ which contains
two remaining vertices of  $\Theta_{S\setminus x}(v_1)$.
Similarly, the same two vertices lie in some block $B'''$
of type ${\rm{II}}$ or ${\rm{III}}$ containing $v_2$.
In particular, 
all vertices of $L_S(v_1)$ are either dead ends of block $B'$,
or are already contained in two blocks, so $x\notin \Theta_{S\setminus y}(v_1)$,
and, moreover, $x$ is not joined with $\Theta_{S\setminus y}(v_1)$
implying that $S$ is not connected.

\smallskip
\noindent
{\bf Case 2.2.2.2:} no block contains  vertices $v_2,u_1$ and $u_2$ simultaneously.
\ In this case $L_S(v_1)$ contains exactly two leaves on distance two from each other,
namely $w_1$ and $w_2$, see Table~\ref{bl4_2222}.
This means that for any $y\notin L_S(v_1)$ distinct from $v_1$ and $x$ and any block decomposition of $S_1=S\setminus y$ vertices $v_1,v_2,w_1,w_2$ compose a block of forth type, which implies that $x$ is not joined with $w_1$ and $w_2$.

\begin{table}[!h]
\begin{center}
\caption{To the proof of Lemma~\ref{no4}, Case 2.2.2.2}
\label{bl4_2222}
\psfrag{v1}{\tiny $v_1$}
\psfrag{v2}{\tiny $v_2$}
\psfrag{w1}{\tiny $w_1$}
\psfrag{w2}{\tiny $w_2$}
\psfrag{u2}{\tiny $u_2$}
\psfrag{u1}{\tiny $u_1$}
\psfrag{T}{\small $\Theta_{S\setminus x}(v_1)$}
\psfrag{L1}{\small $L_{S\setminus x}(v_1)$}
\psfrag{L2}{\small $L_{S\setminus w_1}(v_1)$}
\epsfig{file=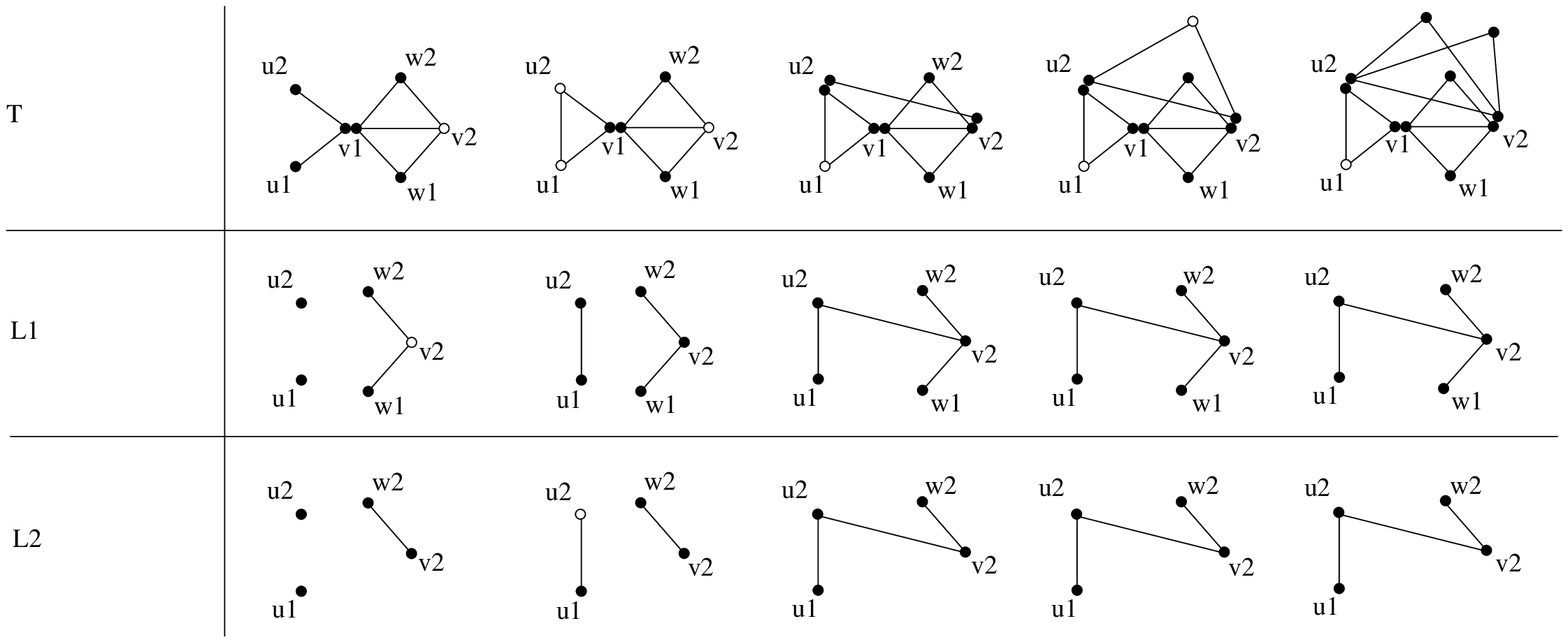,width=0.99\linewidth}
\end{center}
\end{table}

Now consider $S_2=S\setminus w_1$ with a block decomposition.
Clearly, ${\rm Val}_{S_2}(v_1)=4$. Looking at possible quivers $\Theta_{S\setminus x}(v_1)$
(see Table.~\ref{bl4_2222}),
one can notice that $L_{S_2}(v_1)$ does not contain a pair of
leaves on distance two,
so $v_1$ is contained in two blocks of second or third type.
Since $w_2$ is joined with $v_1$ and $v_2$ only,
it is easy to see that $w_2,v_1,v_2$ compose a block of type ${\rm{II}}$.
Now recall that neither $w_1$ nor $w_2$ are
joined with any vertex of $S$ except $v_1$ and $v_2$,
so we can replace triangle $v_1,v_2,w_2$ by block $v_1,v_2,w_1,w_2$ of type ${\rm{IV}}$
to obtain a block decomposition of $S$.

Notice that we do not show in Table.~\ref{bl4_2222} quivers $\Theta_{S\setminus x}(v_1)$ in which vertices $u_1$ and $u_2$ are contained in two blocks simultaneously. The case of a triangle with vertices $u_1,u_2$ and $v_2$ is treated in Case~2.2.2.1, and it is easy to see that all the others do not produce new leaves.

\medskip
\noindent
{\bf Case 3:}  ${\rm Val}_{S\setminus x}(v_1)=4$. \

\smallskip
\noindent
{\bf Case 3.1:} $x\not\perp v_1$, ${\rm Val}_S(v_1)=5$.
\  The proof is the same as in Case 2.1.
Namely, since $ |S|\ge 8$, there exists $y\in S$ which is not joined with $v_1$.
Since  ${\rm Val}_{S\setminus y}(v_1)=5$,
in any block decomposition of $S\setminus y$ vertex $v_1$ is an outlet
of a block  of type ${\rm{IV}}$, so we may refer to Case~2.2.

\smallskip
\noindent
{\bf Case 3.2:} $x\perp v_1$, ${\rm Val}_S(v_1)=4$. \  In this case vertex $v_1$ is
contained in block $B$ and in block $B_1$ of type ${\rm{I}}$ or ${\rm{IV}}$.
Consider these two cases separately.

\smallskip
\noindent
{\bf Case 3.2.1:} Block $B_1$ is of type ${\rm{IV}}$. \ The case is similar to Case~1.2.1,
the only difference is in the orientation of $B_1$: instead of getting double edge,
the edge $(v_1,v_2)$ cancels out. Vertex $v_2$ is also contained in $B_1$,
denote by $w_3$ and $w_4$ dead ends of $B_1$. The only vertex joined
with $B$ and $B_1$ is $x$. We want to show that $x$ is not joined
with $L_S(v_1)$, which will imply that $S$ is not connected.

Take any vertex $y\notin\l L_S(v_1),v_1,x\r$ and consider $S_1=S\setminus y$
with some block decomposition. If $v_1$ is contained in a block of fourth type,
then the second block containing $v_1$ is also of forth type
(otherwise the remaining block is an
edge, and the link $L_{S_1}(v_1)=L_S(v_1)$ contains at most $3$ edges,
contradicting the fact that $L_S(v_1)$ contains $4$ edges), and we see that $x$ is not joined
with $L_S(v_1)$. Therefore, $v_1$ is contained in two blocks of type
${\rm{II}}$ or ${\rm{III}}$. More precisely,
since valence of all neighbors of $v_1$ is at least two, $v_1$ is contained
in two blocks $B'$ and $B_1'$ of type ${\rm{II}}$.

Block $B'$ contains one of $w_1,w_2$ and one of $w_3,w_4$ (due to orientation, see Fig.~\ref{bl4_321} a)), assume that it contains $w_1$ and $w_3$. To avoid the edge $(w_1,w_3)$ which does not appear in $S$, another block $B_1'$ should contain $w_1$ and $w_3$. Since
$w_1$ and $w_3$ are joined with $v_2$, $B_1'$ is either of second or fourth type. In the latter case $v_2$ is a dead end of $B_1'$, but $v_2$ is joined with $w_2$ and $w_4$ also. Hence, $B_1'$ is a triangle containing
$v_2,w_1,w_3$, see Fig.~\ref{bl4_321} b). Similarly, $w_2,w_4$ and $v_2$ are contained in block $B_1''$ of second type. In particular, all the vertices of $L_S(v_1)$ are already contained in two blocks, so $x$ is not joined with $L_S(v_1)$.

\begin{figure}[!h]
\begin{center}
\psfrag{v1}{\tiny $v_1$}
\psfrag{v2}{\tiny $v_2$}
\psfrag{w1}{\tiny $w_1$}
\psfrag{w2}{\tiny $w_2$}
\psfrag{w3}{\tiny $w_3$}
\psfrag{w4}{\tiny $w_4$}
\psfrag{or}{\small or}
\psfrag{a}{\small a)}
\psfrag{b}{\small b)}
\epsfig{file=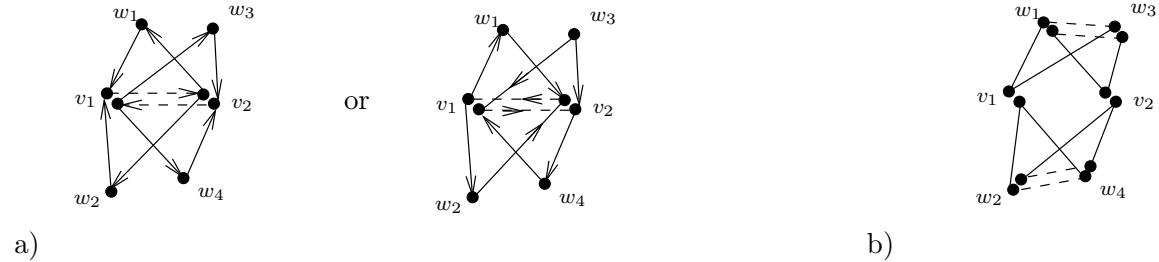,width=0.99\linewidth}
\caption{To the proof of Lemma~\ref{no4}, Case 3.2.1}
\label{bl4_321}
\end{center}
\end{figure}

\smallskip
\noindent
{\bf Case 3.2.2:} Block $B_1$ is of type ${\rm{I}}$.
\ Denote by $u$ the second vertex of $B_1$. If $u$ coincides with $v_2$, then $\l L_S(v_1),v_1\r$ has no outlets, so we can apply Proposition~\ref{razval1} to $x$. Now we may assume that $u\ne v_2$.

If $u$ is joined with $v_2$, then the edge $\l u, v_2\r$ must form a block of first type, otherwise
${{\rm{Val}}}_S(v_2)>{{\rm{Val}}}_S(v_1)$. Therefore, no vertex except $x$ is attached
to $\l L_S(v_1),v_1\r$. Since $|S|\ge 8$, in this case we can apply Proposition~\ref{razval1} to $x$.
Thus, we may assume that $u$ is not joined with $v_2$.

Take any vertex $y\notin\l L_S(v_1),v_1,x\r$ and consider $S_1=S\setminus y$ with some block decomposition. It is easy to see that $v_1$ should be contained in a block of type ${\rm{IV}}$. Furthermore, looking at $L_S(v_1)$ one can notice that there is
exactly one pair of leaves on distance two, namely $w_1$ and $w_2$, which implies that vertices $v_1,v_2,w_1,w_2$ belong to one block, and $w_1$ and $w_2$ are dead ends. Therefore, $x$ is not joined with $w_1$ and $w_2$, so only $v_1$ and $v_2$ are attached to $w_1$ and $w_2$.

Now we proceed as in Case~2.2.2.2. We consider quiver $S_1=S\setminus w_1$ with some block decomposition and show that $w_2,v_1,v_2$ compose a block of type ${\rm{II}}$. Since neither $w_1$ nor $w_2$ are not joined with any vertex of $S$ except
$v_1$ and $v_2$, we can replace triangle $v_1,v_2,w_2$ by block $v_1,v_2,w_1,w_2$ of type ${\rm{IV}}$ to obtain a block decomposition of $S$.

\medskip
\noindent
{\bf Case 4:}  ${\rm Val}_{S\setminus x}(v_1)=3$. \

\smallskip
\noindent
{\bf Case 4.1:} $x\not\perp v_1$, ${\rm Val}_S(v_1)=4$. \ Since ${\rm Val}_{S\setminus x}(v_2)\le {\rm Val}_{S\setminus x}(v_1)$, we have ${\rm Val}_{S\setminus x}(v_2)=3$.
This means that vertices of $B$ may be joined with $x$ only. Thus, we may apply Proposition~\ref{razval1} to $x$.

\smallskip
\noindent
{\bf Case 4.2:} $x\perp v_1$, ${\rm Val}_S(v_1)=3$. \  In this case vertex $v_1$
is contained either in block $B$ only, or in block $B$ and in second type block $B_1$.
Consider the two cases.

\smallskip
\noindent
{\bf Case 4.2.1:} Vertex $v_1$ is contained in one block. \ This implies that $v_1$ and $v_2$ are joined in $S$, so no vertex except $x$ is attached to $B$. We apply Proposition~\ref{razval1} to $x$.

\smallskip
\noindent
{\bf Case 4.2.2:} Vertex $v_1$ is contained in two blocks. \ In this case the second block $B_1$ is of type ${\rm{II}}$, it contains $v_1$, $v_2$ and some vertex $u$. The case is similar to Case~2.2.1, the difference is in orientation of $B_1$.

The union of $B$ and $B_1$ has a unique outlet $u$. Thus, $\l L_{S}(v_1),v_1\r\setminus u$ may be joined with $x$ only. Further, $x$ is not joined with $v_1$ and $v_2$ (since ${\rm{Val}}_S(v_2)\le {\rm{Val}}_S(v_1)$. If $x$ is not joined with $w_1$ and $w_2$ either, then we can apply Proposition~
\ref{razval1} to $u$. Therefore, we may assume that $x$ is joined with one of $w_1$ and $w_2$, say $w_1$. Moreover, we may assume that some vertex $y\notin\l B,u,x\r$ is joined with $u$, otherwise we
can apply Proposition~\ref{razval1} to $x$.

Now take any $z\notin\l B,u,x,y\r$ and consider $S_1=S\setminus z$ with some block decomposition. Vertex $v_1$ is contained either in a blocks of fourth and second type, or in blocks of second and first type. In the first case, due to orientations of edges
(see Fig.~\ref{bl4_422}) block of type ${\rm{IV}}$ should contain all vertices of $B$, which is impossible since dead end $w_1$ is joined with $x$. Hence, $v_1$ is contained in a triangle $B'$ and an edge $B_1'$. Again, because of orientations of edges,
$w_1$ and $w_2$ cannot belong to $B'$ simultaneously, so $B'$ contains $u$ and $w_i$. To avoid the edge $(w_i,u)$ these two vertices should be contained in some block $B_2'$. Since $u$ is joined with $v_2$ and $y$, $B_2'$ contains $v_2$ and $y$
also. Therefore, $B_2'$ is a block of fourth type with outlets $u,w_i$ and dead ends $v_2,y$. But this implies that $y$ attaches to $w_i$ which is impossible.

\begin{figure}[!h]
\begin{center}
\psfrag{v1}{\tiny $v_1$}
\psfrag{v2}{\tiny $v_2$}
\psfrag{w1}{\tiny $w_1$}
\psfrag{w2}{\tiny $w_2$}
\psfrag{wi}{\tiny $w_i$}
\psfrag{x}{\tiny $x$}
\psfrag{y}{\tiny $y$}
\psfrag{u}{\tiny $u$}
\psfrag{or}{\small or}
\psfrag{a}{\small a)}
\psfrag{b}{\small b)}
\epsfig{file=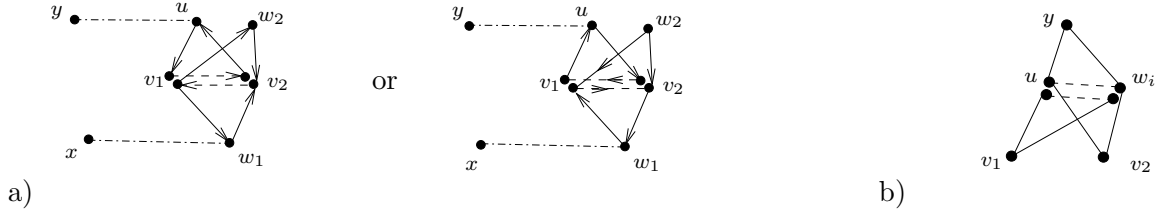,width=0.99\linewidth}
\caption{To the proof of Lemma~\ref{no4}, Case 4.2.2}
\label{bl4_422}
\end{center}
\end{figure}

\medskip
\noindent
{\bf Case 5:}  ${\rm Val}_{S\setminus x}(v_1)=2$. \   In this case vertex $v_1$ is contained in block $B$ and block $B_1$ of type ${\rm{I}}$, where $B_1$ is an edge joining $v_1$ and $v_2$ with suitable orientation. The union of $B$ and $B_1$ has no outlets, so we apply Proposition~\ref{razval1} to $x$.

\medskip

Clearly, valence of $v_1$ in $S\setminus x$ is at least two, so all cases are studied and the lemma is proved.

\end{proof}

\begin{cor}
\label{valle4}
Valence of any vertex $v$ of a minimal non-decomposable quiver  $S$ does not exceed $4$.

\end{cor}

\begin{proof}

The proof is evident. Indeed, take any $x$ which is not joined with $v$, and consider any block decomposition of $S\setminus x$. Vertex $v$ is contained in at most two blocks, valence of any vertex of blocks does not exceed two.

\end{proof}



\begin{lemma}
\label{no3}
Let $v\in S$ be a vertex of valence $4$. Then for any non-neighbor $x$ of $v$ and any block decomposition of $S\setminus x$ vertex $v$ is not contained in a block of third type.

\end{lemma}

\begin{proof}
Denote by $v_1$ and $w_1$ dead ends of block $B$ of type ${\rm{III}}$ with outlet $v$, and denote by $v_2$ and $w_2$ vertices of block $B_1$ with outlet $v$. Clearly, $v_1$ and $w_1$ are not joined with $v_2$ and $w_2$. Denote $S'=\l
v_1,v_2,v,w_1,w_2\r$. If no vertex of $S\setminus \l S',x\r$ is joined with $v_2$ and $w_2$, then we apply Proposition~\ref{razval1} to $x$. Thus, we may assume that some vertex $u_2$ attaches to one of $v_2$ and $w_2$, say $v_2$, see Fig.~\ref{wr4bl3} a). In particular, this implies that $B_1$ is a block of type ${\rm{II}}$.

Suppose that $x$ is not joined with $v_1$ and $w_1$. Since ${\rm Val}_S(v)=4$, the quiver $S\setminus v$ has at most $4$ connected components. Two of them are $v_1$ and $w_1$. The remaining two (or one) contain at least $5$ vertices (due to $|S|\ge 8$), so at least one connected component has at least $3$ vertices. Hence, we can apply Proposition~\ref{razval1} to $v$.

Therefore, we assume that $x$ attaches to at least one of $v_1$ and $w_1$, say $w_1$. We want to prove that $S$ is block-decomposable by applying Proposition~\ref{razval2} to $S=\left\l S_1=\l v_1,w_1\r,b_1=v,b_2=x,S_2=S\setminus\l S_1,v,x\r\right\r$, see Fig.~\ref{wr4bl3}b. For this take  $a_1=v_1$, and try to choose $a_2$. The choice of $a_2$ will depend on $\Theta_{S\setminus x}(v)$.

If $v_2$ and $w_2$ are joined in $S$, then we choose from non-attached to $x$ vertices of $S_2$ (if they do exist) those which is at maximal distance from $v$ in $S$. Clearly, such a vertex can be taken as $a_2$. If each vertex of $S_2$ is joined with $x$, we take as $a_2$ any vertex of $S_2\setminus\l v_2,w_2\r$.

\begin{figure}[!h]
\begin{center}
\psfrag{v}{\tiny $v$}
\psfrag{v1}{\tiny $v_1$}
\psfrag{v2}{\tiny $v_2$}
\psfrag{w1}{\tiny $w_1$}
\psfrag{w2}{\tiny $w_2$}
\psfrag{u2}{\tiny $u_2$}
\psfrag{x}{\tiny $x$}
\psfrag{a1=v1}{\tiny $a_1=v_1$}
\psfrag{u}{\tiny $u$}
\psfrag{c}{\small c)}
\psfrag{a}{\small a)}
\psfrag{b}{\small b)}
\epsfig{file=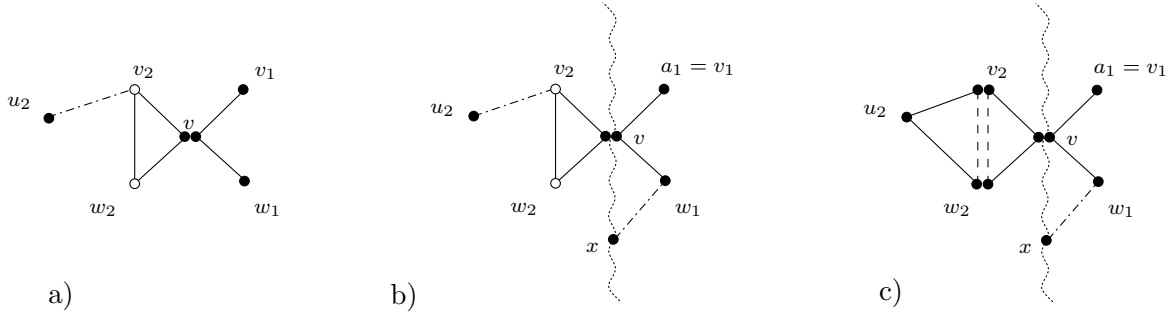,width=0.99\linewidth}
\caption{To the proof of Lemma~\ref{no3}}
\label{wr4bl3}
\end{center}
\end{figure}

Now suppose that $v_2$ and $w_2$ are not joined in $S$ (in particular, $\Theta_{S\setminus x}(v)$ contains also some block $B_2$ composed by $v_2,w_2$ and probably some other vertex of $S_2$). $B_2$ cannot be of first type since $v_2$ attaches to $u_2$. Thus, $w_2$ is joined with $u_2$,  see Fig.~\ref{wr4bl3} c). Moreover, no vertex of $S\setminus x$ attaches to any of $v_2$ and $w_2$
since both of them belong already to two blocks of considered block decomposition of $S\setminus x$. So, $S\setminus v_2$ is connected and we may take $a_2=v_2$.

\end{proof}

\begin{lemma}
\label{wrong4}
Let $v\in S$ be a vertex of valence $4$. Then for any non-neighbor $x$ of $v$ and any block decomposition of $S\setminus x$ the diagram $\Theta_{S\setminus x}(v)$ consists of exactly two blocks of type ${\rm{II}}$ having the only vertex $v$ in common.

\end{lemma}

\begin{proof}

Lemmas~\ref{no3},~\ref{no4}, and~\ref{no5}
rule out blocks of types $\rm{III},\rm{IV}$, and $\rm{V}$ from
decomposition of ${S\setminus x}$. The only possibility left is that
$v_1$ is contained in two blocks $B_1$ and $B_2$ of second type.
Clearly, they have at most $4$ outlets, so $\Theta_{S\setminus x}(v)$ may contain at most $2$ additional blocks. Denote by $v_1,u_1$ and $v_2,u_2$ remaining vertices of $B_1$ and $B_2
$ respectively, and consider the following cases (see Table~\ref{wrongval4}).

\begin{table}[!h]
\begin{center}
\caption{To the proof of Lemma~\ref{wrong4}}
\label{wrongval4}
\psfrag{q}{$)$ }
\psfrag{1}{\bf 1}
\psfrag{1.}{\scriptsize $|B_1\cap B_2|\ge 2$}
\psfrag{2}{\bf 2}
\psfrag{2.}{\scriptsize $|B_1\cap B_2|=1$}
\psfrag{1.1}{\small \bf 1.1}
\psfrag{1.1.} {\scriptsize $|B_1\cap B_2|=2$}
\psfrag{1.2}{\small \bf 1.2}
\psfrag{1.2.}{\scriptsize $|B_1\cap B_2|=3$}
\psfrag{2.0}{\tiny  $\Theta_{S\setminus x}(v)=B_1\cup B_2$}
\psfrag{2.1}{\small \bf 2.1}
\psfrag{2.1.}{\tiny $\Theta_{S\setminus x}(v)=B_1\cup B_2\cup B_3$}
\psfrag{2.2}{\small \bf 2.2}
\psfrag{2.2.}{\tiny $\Theta_{S\setminus x}(v)=B_1\cup B_2\cup B_3\cup B_4$}
\psfrag{2.1.1}{\small {\bf 2.1.1}}
\psfrag{2.1.2}{\small {\bf 2.1.2}}
\psfrag{2.2.1}{\small {\bf 2.2.1}}
\psfrag{2.2.2}{\small {\bf 2.2.2}}
\psfrag{2.2.3}{\small {\bf 2.2.3}}
\psfrag{2.2.1.1}{\scriptsize {\bf 2.2.1.1}}
\psfrag{2.2.1.2}{\scriptsize {\bf 2.2.1.2}}
\psfrag{2.1.2.1}{\scriptsize {\bf 2.1.2.1}}
\psfrag{2.1.2.2}{\scriptsize {\bf 2.1.2.2}}
\psfrag{2.1.1.1}{\scriptsize {\bf 2.1.1.1}}
\psfrag{2.1.1.2}{\scriptsize {\bf 2.1.1.2}}
\psfrag{2.1.1.3}{\scriptsize {\bf 2.1.1.3}}
\psfrag{B1}{\tiny $B_1$}
\psfrag{B2}{\tiny $B_2$}
\psfrag{B3}{\tiny $B_3$}
\psfrag{B4}{\tiny $B_4$}
\psfrag{v}{\tiny $v$}
\psfrag{w3}{\tiny $w_3$}
\psfrag{w4}{\tiny $w_4$}
\psfrag{B3=}{\scriptsize $B_3=$}
\psfrag{B4=}{\scriptsize $B_4=$}
\psfrag{2.2.1.1.1}{{\tiny  \bf 2.2.1.1.1}}
\psfrag{2.2.1.1.2}{{\tiny  \bf 2.2.1.1.2}}
\psfrag{2.2.1.2.1}{{\tiny  \bf 2.2.1.2.1}}
\psfrag{2.2.1.2.2}{{\tiny  \bf 2.2.1.2.2}}
\psfrag{=}{\tiny ($w_3=w_4$)}
\psfrag{ne}{\tiny ($w_3\ne w_4$)}
\epsfig{file=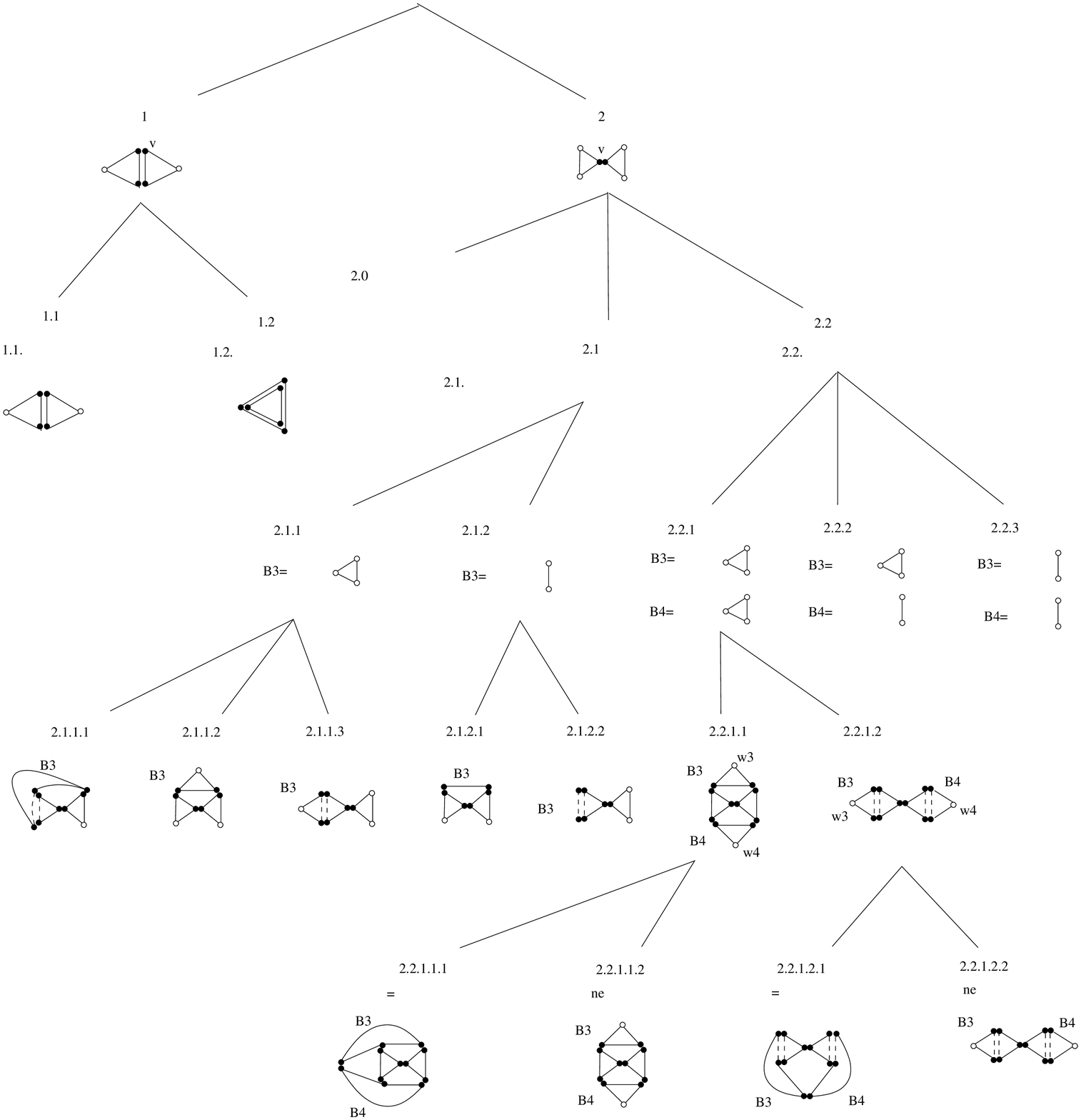,width=1.0\linewidth}
\end{center}
\end{table}

\medskip
\noindent
{\bf Case 1:} $B_1$ and $B_2$ have at least two vertices in common. \

\smallskip
\noindent
{\bf Case 1.1:} $B_1$ and $B_2$ have exactly two vertices in common.
\ We may assume that $v_1=v_2$.
Then ${\rm Val}_S(v)={\rm Val}_S(v_1)=4$ is the maximal possible
valence in $S$, so $x$ is not attached to $v$ and $v_2$.
Consider  all options for $\Theta_{S\setminus x}(v)$ (see Fig.~\ref{wr4_11}).

\begin{figure}[!h]
\begin{center}
\psfrag{v}{\tiny $v$}
\psfrag{v1}{\tiny $v_1=v_2$}
\psfrag{w}{\tiny $w$}
\psfrag{u1}{\tiny $u_1$}
\psfrag{u2}{\tiny $u_2$}
\epsfig{file=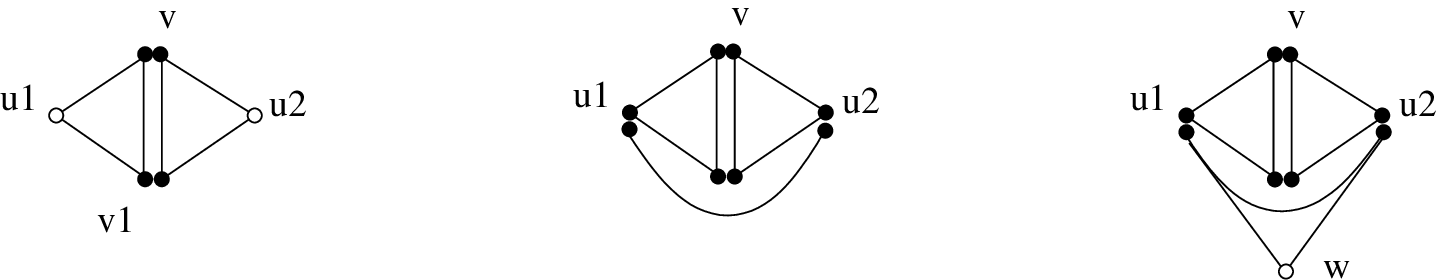,width=0.85\linewidth}
\caption{To the proof of Lemma~\ref{wrong4}, Case~1.1}
\label{wr4_11}
\end{center}
\end{figure}

If $u_1$ and $u_2$ are not joined
(i.e. $\Theta_{S\setminus x}(v)$ consists of $B_1$ and $B_2$ only,
see the left picture of Fig.\ref{wr4_11}), then $S$ is block-decomposable by Corollary~\ref{after2} applied to $S\!=\!\left\l S_1\!=\!\l v,v_2\r,b_1\!=\!u_1,b_2\!=\!u_2,S_2\!=\!S\setminus\l S_1,u_1,u_2\r\right\r$  with $c_1=v$.

If the edge $\l u_1, u_2\r$ forms a block of first type
(see Fig.\ref{wr4_11}, the middle picture),
then they can be connected only with vertex $x\in S$. Hence,
 all the vertices of $\Theta_{S\setminus x}(v)$ are dead ends, so $S$ is block-decomposable by Proposition~\ref{razval1} applied to $x$.

If $u_1$ and $u_2$ are contained in a block of second type
with additional vertex $w$ (see Fig.\ref{wr4_11}, right),
then $u_1$ and $u_2$ have valence $4$, so they are not joined with $x$. Therefore, $S$ is block-decomposable by Proposition~\ref{razval1} applied to $w$.

\smallskip
\noindent
{\bf Case 1.2:} $B_1$ and $B_2$ have three vertices in common. \ In this case union of $B_1$ and $B_2$ has no outlets, so $S$ is block-decomposable by Proposition~\ref{razval1} applied to $x$.

\medskip
\noindent
{\bf Case 2:} $B_1$ and $B_2$ intersect at $v$ only.
\ The quiver $\Theta_{S\setminus x}(v)$ consists of two, three, or four blocks. We are going to prove that there are no
other blocks in $\Theta_{S\setminus x}(v)$
except $B_1$ and $B_2$, this will imply
our lemma. For that we consider the two remaining cases
and find a contradiction.

\smallskip
\noindent
{\bf Case 2.1:}
$\Theta_{S\setminus x}(v)$ consists of three blocks $B_1$, $B_2$ and $B_3$. \ We go through different types of $B_3$ and the way it attaches to $B_1$ and $B_2$.

\smallskip
\noindent
{\bf Case 2.1.1:} $B_3\in \B_{\rm{II}}$ \

\smallskip
\noindent
{\bf Case 2.1.1.1:} $B_3$ has $3$ points in common with the union of $B_1$ and $B_2$. \ Let $v_1,u_1,v_2$ be vertices of $B_3$.
Either $v_1$ and $u_1$ are joined by a double edge or they are no joined at all. If they are joined by a double edge, then valence of $u_1$ equals $4$, so we are in assumptions of Case~1, which implies that $S$ is block-decomposable. Hence, we may assume that
$v_1$ and $u_1$ are not joined in $S$. Thus,  $\Theta_{S\setminus x}(v)$ is the quiver shown on Fig.~\ref{wr4_2111}. The only outlet is $u_2$. We may assume that some $y\notin \l \Theta_{S\setminus x}(v),x\r$ is joined with $u_2$, otherwise $S$ is block-decomposable by Proposition~\ref{razval1} applied to $x$.

\begin{figure}[!h]
\begin{center}
\psfrag{v}{\tiny $v$}
\psfrag{y}{\tiny $y$}
\psfrag{B3}{\tiny $B_3$}
\psfrag{v1}{\tiny $v_1$}
\psfrag{v2}{\tiny $v_2$}
\psfrag{u1}{\tiny $u_1$}
\psfrag{u2}{\tiny $u_2$}
\epsfig{file=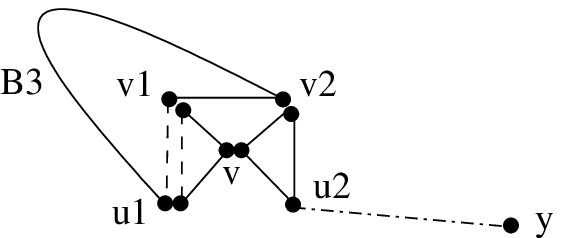,width=0.305\linewidth}
\caption{To the proof of Lemma~\ref{wrong4}, Case~2.1.1.1}
\label{wr4_2111}
\end{center}
\end{figure}

Consider any $z\notin\l \Theta_{S\setminus x}(v),x,y\r$ and some
block decomposition of $S\setminus z$.
By Lemmas~\ref{no3},~\ref{no4}, and~\ref{no5}, $v$ is contained in two blocks of second type.
One of these blocks contains $v,v_2$ and one of $u_1,u_2,v_1$. Another one
contains $v$ and two remaining vertices from $u_1,u_2,v_1$. Similarly, $v_2$ is also contained in block of second type with vertices $v_2$ and two of $u_1,u_2,v_1$. This imply that only two of $u_1,u_2,v_1$ are dead ends
of $\l v,v_1,v_2,u_1,u_2\r$, and only one of them is outlet.

Note that no other vertex than $u_1,u_2,v_1,v$ is joined with $v_2$ otherwise ${\rm{Val}}_S(v_2)>4$.
If $u_2$ is an outlet, then $x$ may be joined with $u_2$ only in $\Theta_{S\setminus x}(v)$, so $S$ is block-decomposable by Proposition~\ref{razval1} applied to $u_2$.

If $u_1$ or $v_1$ is an outlet, then $u_2$ is dead end, but $y$ is attached to $u_2$, so we get a contradiction.

\smallskip
\noindent
{\bf Case 2.1.1.2:} $B_3$ has exactly $1$ point in common with each of $B_1$ and $B_2$. \ We may assume that vertices of $B_3$ are $v_1,v_2$ and $w$. Since ${\rm Val}_S(v_1)={\rm Val}_S(v_2)=4$, no vertex of $S\setminus\l \Theta_{S\setminus x}(v),w\r$ is
joined with $v_1$ or $v_2$.

Consider the quiver $S_1=S\setminus v_1$ with a block decomposition. Since ${\rm Val}_{S_1}(v)=3$ and $v_2$ is joined with $u_2$ in $S$, $v$ is contained in one block of second type and in one of first type. But $u_1$ is joined neither with $v_2$ nor with $u_2$,
so $v,v_2,u_2$ are vertices of one block, and $v,u_1$ compose another one block. Looking at vertex $v_2$ we see that $v_2$ and $w$ compose a block of first type, too. Replace the block $v_2 w$ by $B_3$, and block  $v u_1$ by $B_1$, and obtain a block decomposition of $S$.

\smallskip
\noindent
{\bf Case 2.1.1.3:} $B_3$ does not intersect one of $B_1$ and $B_2$. \ In this case we may assume that vertices of $B_3$ are $v_1,u_1$ and $w$. Similarly to Case~2.1.1.1, we conclude that either this situation is already considered in Case 1, or $v_1$ and $u_1$ are not joined in $S$. Take any $y\notin \l \Theta_{S\setminus x}
(v),x\r$ and consider $S_1=S\setminus y$.
Since $v_1$ and $u_1$ are not joined with $v_2$ and $u_2$, vertices $v,v_2$ and $u_2$ compose one block
(otherwise in order to cancel two edges joining $\l v_1,u_1\r$ with $\l v_2,u_2\r$ we have to glue in two blocks such that neither of them contains simultaneously $u_2$ and
$v_2$. Then both $u_2$ and $v_2$ are dead ends with no edge between them contradicting the fact $u_2$ and $v_2$ are joined in $S$). Therefore, $v,v_1$ and $u_1$ compose a block, so $v_1,u_1$ and $w$ also form a block. In particular, $v_1$ and $u_1$ are dead ends of $\l u_1,w,v_1,v\r$, and $x$ is not
attached to  $v_1$ and $u_1$. Recall also that no vertex except $x$, $v$, $w$ could be attached to $v_1$ and $u_1$ since $v_1$ and $u_1$ are dead ends of $\Theta_{S\setminus x}(v)$. Hence, both $u_1$ and $v_1$ are joined with $v$ and $w$ only.

Consider $S_2=S\setminus v_1$ with some block decomposition. Similarly to Case~2.1.1.2, it easy to see that $v,v_2,u_2$ compose one block of type {\rm{II}}, and $u_1,v$ form another block of type {\rm{I}}. Since $u_1$ is not joined with any vertex except $v$ and $w$,
$\l u_1,w\r$ is a block. Notice that blocks $\l u_1,w,v_1 \r$ and $\l u_1,v_1,v\r$ are oriented so that sides $\overrightarrow{v_1 u_1}$ of one triangle cancels out with the side $\overrightarrow{u_1 v_1}$ of the other. This yields that one of the edges
$\l u_1, w\r$ and $\l u_1, v\r$ is directed towards $u_1$ while the other is directed from $u_1$. Replacing $(u_1,w)$ by $B_3$, and $(u_1,v)$ by $B_1$, we obtain a block decomposition of $S$.

\smallskip
\noindent
{\bf Case 2.1.2:}
$B_3\in \B_{\rm{I}}$. \ There are two possibilities to attach $B_3$ to $B_1$ and $B_2$.

\smallskip
\noindent
{\bf Case 2.1.2.1:} $B_3$ has exactly one point in common with each of $B_1$ and $B_2$. \ We may assume that vertices of $B_3$ are $v_1$ and $v_2$ (see Fig.~\ref{wr4_2121}a). If $x$ is not joined with $v_1$ and $v_2$, then $S$ is block-decomposable by Corollary~\ref{after2} applied to
$S=\left\l S_1=\l v,v_1,v_2\r,b_1=u_1,\right.$ $\left.b_2=u_2,S_2=S\setminus\l S_1,u_1,u_2\r\right\r$ with $c_1=v_1$.  So, we may assume that $x$ is joined with at least one of $v_1$ and $v_2$, say $v_1$.

Take any $y\notin \l \Theta_{S\setminus x}(v),x\r$ and consider $S_1=S\setminus y$. Valence of $v_1$ is equal to four, which means that $v_1$ is contained in two blocks of second type. Since $v_1$ is joined with $v$, one of these blocks (call it $B'$) contains both $v_1$ and $v$. The third vertex of $B'$ is either $v_2$ or $u_1$ (since ${\rm{Val}}_S(v)={\rm{Val}}_S(v_1)=4$ is the maximal possible and only $u_1$ and $v_2$ are joined with both $v_1$ and $v$).

If $B'$ contains $v_2$, then vertices $u_1,v_1,x$ compose one block. The second block containing $v$ is composed by $v,u_1,u_2$. In particular, $u_1$ is joined with $u_2$, which contradicts the assumption, see Fig.~\ref{wr4_2121} b) . Therefore, $B'$ is composed by $u_1,v_1$
and $v$. The second block containing $v$ is composed by $v,v_2,u_2$, and the second block containing $v_1$ is composed by $v_1,v_2,x$,  see Fig.~\ref{wr4_2121} c) We obtain a quiver considered above in Case~2.1.1.2.

\begin{figure}[!h]
\begin{center}
\psfrag{a}{\scriptsize a)}
\psfrag{b}{\scriptsize b)}
\psfrag{c}{\scriptsize c)}
\psfrag{v1}{\tiny $v_1$}
\psfrag{v2}{\tiny $v_2$}
\psfrag{v}{\tiny $v$}
\psfrag{x}{\tiny $x$}
\psfrag{w4}{\tiny $w_4$}
\psfrag{u1}{\tiny $u_1$}
\psfrag{u2}{\tiny $u_2$}
\epsfig{file=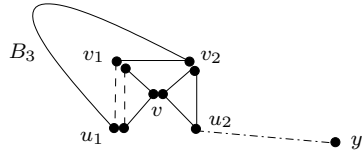,width=0.85\linewidth}
\caption{To the proof of Lemma~\ref{wrong4}, Case~2.1.2.1}
\label{wr4_2121}
\end{center}
\end{figure}

\smallskip
\noindent
{\bf Case 2.1.2.2:} $B_3$ does not intersect one of $B_1$ and $B_2$. \ The proof repeats the proof of Lemma~\ref{no3}.

\smallskip
\noindent
{\bf Case 2.2:}
$\Theta_{S\setminus x}(v)$ consists of four blocks $B_1$, $B_2$, $B_3$, and $B_4$.
\ We go through all the different types of $B_3$ and $B_4$ and all the ways to assemble $\Theta_{S\setminus x}(v)$ from them.

\smallskip
\noindent
{\bf Case 2.2.1:} Both $B_3\in\B_{\rm{II}}$ and $B_4\in\B_{\rm{II}}$.
\ There are two possibilities to attach $B_3$ and $B_4$ to $B_1$ and $B_2$.

\smallskip
\noindent
{\bf Case 2.2.1.1:} Each of $B_3$ and $B_4$ has exactly one vertex in common with each of $B_1$ and $B_2$. \ Denote by $w_3$ and $w_4$ the remaining vertices of $B_3$ and $B_4$ respectively, and consider two possibilities.

\smallskip
\noindent
{\bf Case 2.2.1.1.1:} Vertices $w_3$ and $w_4$ coincide. \ In this case valences of all the $6$ vertices of $\Theta_{S\setminus x}(v)$ are equal to $4$ yielding that
$\l\Theta_{S\setminus x}(v)\r\perp (S\setminus\l\Theta_{S\setminus x}(v)\r)$ and $S$
is not connected.

\smallskip
\noindent
{\bf Case 2.2.1.1.2:} Vertices $w_3$ and $w_4$ do not coincide. \ We may assume that $B_3$ contains vertices $w_3,u_1$ and $u_2$. Consider $S_1=S\setminus u_1$ with some block decomposition.  Since valences of $v,v_1,u_1,v_2,u_2$ are equal
to $4$, no vertex from $S\setminus\Theta_{S\setminus x}(v)$ attaches to any of these $5$ vertices.  We want to prove that edges $\l w_3,u_2\r$ and $\l v_1,v\r$ are blocks of first type of $S_1$, see Fig.~\ref{wr4_22112}. Then replacing $(w_3,u_2)$ and $(v_1,v)$ by $B_3$ and $B_1$ respectively we get
a block decomposition of $S$.

\begin{figure}[!h]
\begin{center}
\psfrag{a}{\scriptsize a)}
\psfrag{b}{\scriptsize b)}
\psfrag{v}{\tiny $v$}
\psfrag{x}{\tiny $x$}
\psfrag{v1}{\tiny $v_1$}
\psfrag{v2}{\tiny $v_2$}
\psfrag{w3}{\tiny $w_4$}
\psfrag{w4}{\tiny $w_3$}
\psfrag{u1}{\tiny $u_1$}
\psfrag{u2}{\tiny $u_2$}
\epsfig{file=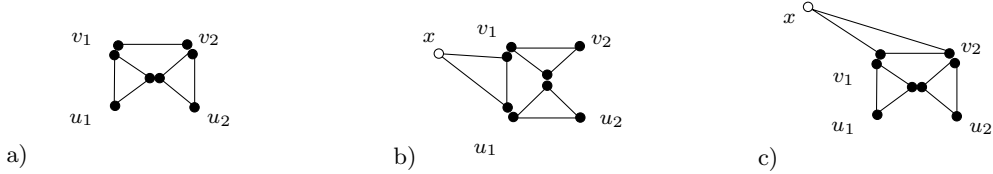,width=0.45\linewidth}
\caption{To the proof of Lemma~\ref{wrong4}, Case~2.2.1.1.2}
\label{wr4_22112}
\end{center}
\end{figure}

Since ${\rm Val}_{S_1}(u_2)=3$, $u_2$ is contained in one block of second type and one of first type. The edge $\l v_2,u_2\r$ belongs to block of type ${\rm{II}}$ because of ${\rm Val}_{S_1}(v_2)=4$. By the same reason, $w_3$ and $v_2$ are not contained in one
block. Therefore, $v_2,v,u_2$ compose one block, and $w_3,u_2$ compose another one block. Looking at vertex $v_2$ we see that the second block containing $v_2$ is spanned by $v_2,v_1$ and $w_4$. Recall that vertices $v$ and $v_1$ are not
joined with any vertex of $S_1$ except $w_3,v_2$ and $u_2$. Thus, $v$ and $v_1$ compose a block, which completes the case.

\smallskip
\noindent
{\bf Case 2.2.1.2:} One of $B_3$ and $B_4$ (say $B_3$) does not intersect $B_2$ and the other (namely, $B_4$) does not intersect $B_1$. \ Denote by $w_3$ and $w_4$ the remaining vertices of $B_3$ and $B_4$ respectively. Taking into account Case~1, we may assume that $u_2$ is not joined with $v_2$ in $S$, and $u_1$ does not attach to $v_1$. We consider two possibilities.

\smallskip
\noindent
{\bf Case 2.2.1.2.1:} Vertices $w_3$ and $w_4$ coincide. \ Notice that each vertex of $\Theta_{S\setminus x}(v)$ is already contained in two blocks, so no vertex of $S\setminus\l\Theta_{S\setminus x}(v)\r$ except $x$ is attached to $\l\Theta_{S\setminus x}(v)\r$.
To show that $x$ is not joined with $\l\Theta_{S\setminus x}(v)\r$ either, take any $y\notin\l\Theta_{S\setminus x}(v),x\r$ (such $y$ does exist since $|S|\ge 8$) and consider $S\setminus y$ with some block decomposition.
Valences of $v$ and $w_3$ are equal to $4$, so they belong to two blocks of second type each. Further, suppose that some pair of $u_1,u_2,v_1,v_2$ compose a block with $w_3$. To avoid an
edge between this pair of vertices, they should compose also a block with $v$.  Therefore, each vertex of $\Theta_{S\setminus x}(v)$ is contained in two blocks, so no vertex of $S\setminus\l\Theta_{S\setminus x}(v)\r$ except $y$ attaches to $\l\Theta_{S\setminus x}(v)\r$. In particular, $x\perp \l\Theta_{S\setminus x}(v)\r$.

Thus, $S\setminus\l\Theta_{S\setminus x}(v)\r\perp \l\Theta_{S\setminus x}(v)\r$, and $S$ is not connected.

\smallskip
\noindent
{\bf Case 2.2.1.2.2:} Vertices $w_3$ and $w_4$ do not coincide. \ Clearly, no vertex of $(S\setminus x)\setminus\Theta_{S\setminus x}(v)$ is joined with $v$ or any of its neighbors, see Fig.~\ref{wr4_22122} a). Suppose that $x$ is joined with any of neighbors of $v$, say with $u_1$. Consider
$S_1=S\setminus w_4$ with some block decomposition. Since $u_1$ does not attach to any of neighbors of $v$ and ${\rm Val}_{S_1}(u_1)=3$, the edge $\l u_1,v\r$ is not contained in block of second type. However,
 ${\rm{Val}}_{S_1}(v)=4$
 so the
edge $\l u_1,v\r$ should be contained in some block of second type. The contradiction shows that $x\perp u_1$. Similarly, $x$ is not joined with  any other neighbor of $v$ (and with $v$ itself, of course).

\begin{figure}[!h]
\begin{center}
\psfrag{a}{\scriptsize a)}
\psfrag{b}{\scriptsize b)}
\psfrag{v}{\tiny $v$}
\psfrag{x}{\tiny $x$}
\psfrag{v1}{\tiny $v_1$}
\psfrag{v2}{\tiny $v_2$}
\psfrag{w3}{\tiny $w_3$}
\psfrag{w4}{\tiny $w_4$}
\psfrag{u1}{\tiny $u_1$}
\psfrag{u2}{\tiny $u_2$}
\epsfig{file=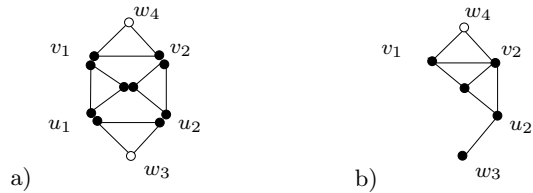,width=0.85\linewidth}
\caption{To the proof of Lemma~\ref{wrong4}, Case~2.2.1.2.2}
\label{wr4_22122}
\end{center}
\end{figure}

Now consider $S_2=S\setminus u_1$,  see Fig.~\ref{wr4_22122} b).
Since
$v_1\perp v_2$, $v_1\perp u_2$
and $v$ is the only common neighbor of $v_1$ and $v_2$ (or $v_1$ and $u_2$), a block containing
the edge $\l v, v_1\r$ contains neither $v_2$ nor $u_2$. Therefore, $v_1$ and $v$ compose a block of first type. As we have proved above,
${\rm{Val}}_S(v_1)=2$
so $v_1$ and $w_3$ also compose a block of first type. Now replacing the edge $(v_1,v)$ by $B_1$, and $(v_1,w_3)$ by $B_3$, we obtain a block decomposition
of $S$.

\smallskip
\noindent
{\bf Case 2.2.2:}
$B_3\in \B_{\rm{II}}$, and  $B_4\in \B_{\rm{I}}$. \ Denote by $w_3$ the remaining vertex of $B_3$. If $B_4$ does not intersect one of $B_1$ or $B_2$ (say $B_1$), then
 it must coincide with the edge $\l u_2,v_2\r$ of $B_2$, and
we get a situation described in Lemma~\ref{no3} (notice that we did not use orientations of edges while proving Lemma~\ref{no3}). Hence, we may assume that $B_4$ intersects both $B_1$ and $B_2$. This implies that $B_3$ does the same. Let $u_1$ and $u_2$ be vertices of $B_4$. Then $\Theta_{S\setminus x} (v)$ is a quiver shown on Fig.~\ref{wr4_222}. Consider  $\Theta_{S\setminus x} (v_2)$. Notice, that ${\rm{Val}}_{S\setminus x}(v_2)=4$ and  $\Theta_{S\setminus x} (v_2)$ was treated in Case~2.1.1.2.

\begin{figure}[!h]
\begin{center}
\psfrag{v}{\tiny $v$}
\psfrag{v1}{\tiny $v_1$}
\psfrag{v2}{\tiny $v_2$}
\psfrag{w3}{\tiny $w_3$}
\psfrag{u1}{\tiny $u_1$}
\psfrag{u2}{\tiny $u_2$}
\epsfig{file=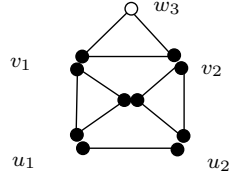,width=0.2\linewidth}
\caption{To the proof of Lemma~\ref{wrong4}, Case~2.2.2}
\label{wr4_222}
\end{center}
\end{figure}

\smallskip
\noindent
{\bf Case 2.2.3:} Both $B_3, B_4\in \B_{\rm{I}}$.
\ In this case all the vertices of $\Theta_{S\setminus x}(v)$ are dead ends, so $S$ is block-decomposable by Proposition~\ref{razval1} applied to $x$.

\medskip

Since all possibilities are exhausted, the lemma is proved.

\end{proof}

\begin{lemma}
\label{wrong3}
Let $v\in S$ be a vertex of valence $3$. Then for any non-neighbor $x$ of $v$ and any block decomposition of $S\setminus x$ the diagram $\Theta_{S\setminus x}(v)$ consists of exactly two blocks, one of type ${\rm{II}}$ and the other of type $\rm{I}$ having  only vertex $v$ in common.

\end{lemma}

\begin{proof}
Since valence of $v$ equals $3$, $v$ is contained in block $B_1$
of second or third type with two other vertices $v_1$ and $u_1$,
and in block $B_2$ of first type with second vertex $v_2$.
We consider both types of $B_1$ and all possible quivers
$\Theta_{S\setminus x}(v)$  below (see Table~\ref{wr3}).

\medskip
\noindent
{\bf Case 1:} $B_1\in\B_{\rm{III}}$. \ In this case $v_1$ and $u_1$ are
dead ends of the union of $B_1$ and $B_2$, so
$\{v_1,u\}\perp (S\setminus\{x,v\})$,
and $\Theta_{S\setminus x}(v)$ consists of $B_1$ and $B_2$ only.
If $x$ is not joined with $v_1$ and $u_1$, then $S$ is block-decomposable by Proposition~\ref{razval1} applied to $v_2$. Therefore, we may assume that $x$ is joined with at least one of $v_1$ and $u_1$, say $u_1$.
If $v_2\perp S\setminus\l\Theta_{S\setminus x}(v),x\r$, then again $S$ is block-decomposable by Proposition~\ref{razval1} applied to $x$.
Thus, we can assume that $v_2$ is joined with some vertex distinct from $v$ and $x$, see Fig.~\ref{wr3_1}.  Now we can apply Corollary~\ref{after2} to $S=\left\l S_1=\l u_1,v_1,v\r,b_1=v_2,b_2=x,S_2=S\setminus\l S_1,x,v_2\r\right\r$ with $c_1=v_1$ to show that $S$ is block-decomposable.

\begin{figure}[!h]
\begin{center}
\psfrag{v}{\tiny $v$}
\psfrag{v1}{\tiny $v_1$}
\psfrag{v2}{\tiny $v_2$}
\psfrag{u1}{\tiny $u_1$}
\psfrag{x}{\tiny $x$}
\epsfig{file=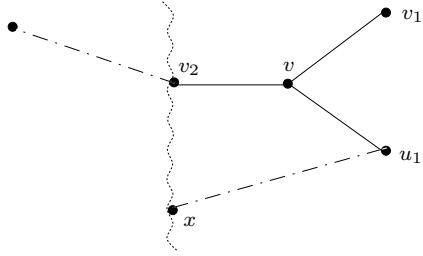,width=0.36\linewidth}
\caption{To the proof of Lemma~\ref{wrong3}, Case 1}
\label{wr3_1}
\end{center}
\end{figure}

\begin{table}[!h]
\begin{center}
\caption{To the proof of Lemma~\ref{wrong3}}
\label{wr3}
\psfrag{1}{\bf 1}
\psfrag{B=}{\small $B_1=$}
\psfrag{2}{\bf 2}
\psfrag{2.1}{\small \bf  2.1}
\psfrag{2.1.}{\small $|B_1\cap B_2|=2$}
\psfrag{2.2}{\bf 2.2}
\psfrag{2.2.}{\small $|B_1\cap B_2|=1$}
\psfrag{2.2.0.}{\small  (No $B_3$)}
\psfrag{2.2.1}{\small \bf 2.2.1}
\psfrag{2.2.2}{\small \bf 2.2.2}
\psfrag{B3=}{\small $B_3=$}
\psfrag{2.2.1.1}{{\scriptsize  \bf 2.2.1.1}}
\psfrag{2.2.1.2}{{\scriptsize \bf 2.2.1.2}}
\psfrag{2.2.2.1}{{\scriptsize \bf 2.2.2.1}}
\psfrag{2.2.2.2}{{\scriptsize \bf 2.2.2.2}}
\psfrag{2.2.2.3}{{\scriptsize \bf 2.2.2.3}}
\psfrag{B1}{\tiny $B_1$}
\psfrag{B2}{\tiny $B_2$}
\psfrag{B3}{\tiny $B_3$}
\psfrag{v}{\tiny $v$}
\psfrag{v1}{\tiny $v_1$}
\psfrag{v2}{\tiny $v_2$}
\psfrag{w}{\tiny $w$}
\psfrag{u1}{\tiny $u_1$}
\epsfig{file=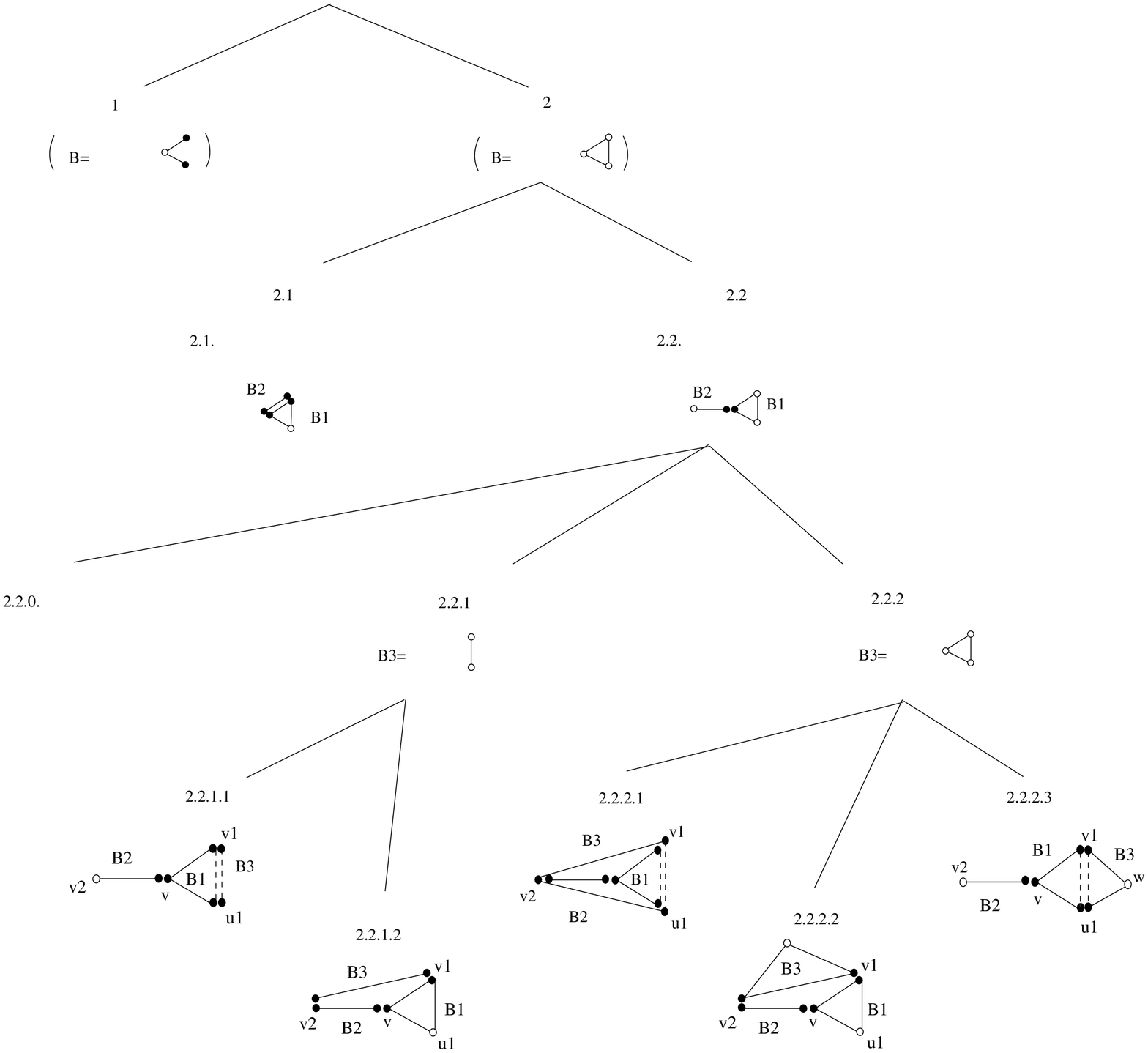,width=1.0\linewidth}
\end{center}
\end{table}

\medskip
\noindent
{\bf Case 2:} $B_1\in\B_{\rm{II}}$. \ Consider the following cases.

\smallskip
\noindent
{\bf Case 2.1:} $B_1$ and $B_2$ have two points in common. \ We may assume that $v_1=v_2$. Since ${\rm Val}_S(v)=3$ and $v\perp x$, $v_2$ and $v$ are joined by a double edge. By Lemma~\ref{wrong4}, no vertex of valence $4$ may be
incident to a double edge.
Thus, $S\setminus\l u_1,v,v_2\r\perp \l v,v_2\r$ and
$S\setminus u_1$ is not connected. Since valence of $u_1$ in $S$ does not exceed $4$, at least one connected component of  $S\setminus u_1$ has more than $2$ vertices (as in the proof of Lemma~\ref{no3}). Therefore, $S$ is block-decomposable by Proposition~\ref{razval1} applied to $u_1$.

\smallskip
\noindent
{\bf Case 2.2:} Vertex $v$ is the only common vertex of $B_1$ and $B_2$.
\ $\Theta_{S\setminus x}(v)$ may consist of two or three blocks.
To prove the statement we need to exclude the option of three blocks.
It is done in remaining part of the proof. Assume that $\Theta_{S\setminus x}(v)$
contains an additional block $B_3$. Clearly, $B_3$ is either of the first or of the second type.

\smallskip
\noindent
{\bf Case 2.2.1:} $B_3\in \B_{\rm{I}}$. \ There are two ways to attach $B_3$ to $B_1$ and $B_2$.

\smallskip
\noindent
{\bf Case 2.2.1.1:} $B_3$ does not intersect $B_2$. \ In this case $v_1$ and $u_1$ are either joined by a double edge or are not joined
in $S$ at all. If they are not joined, then
 both $v_1$ and $u_1$ are dead ends of $\Theta_{S\setminus x}(v)$ and
$S$ is block-decomposable (as in Case~1). If they are joined by a double edge, then all three vertices
$v_1,u_1$ and $v$ have valence $3$ in $S$, so  $\l v_1,u_1,v\r\perp x$,
and $S$ is block-decomposable by Proposition~\ref{razval1} applied to $v_2$.

\smallskip
\noindent
{\bf Case 2.2.1.2:} $B_3$ has one point in common with each of $B_1$ and $B_2$. \ We can assume that $B_3$ consists of $v_1$ and $v_2$. Since $L_S(v_1)$ contains a connected component of order at least $3$, Lemma~\ref{wrong4} yields
that ${\rm Val}_S(v_1)<4$, so ${\rm Val}_S(v_1)=3$.
Hence, $S\setminus\l\Theta_{S\setminus x}(v)\r\perp \l v,v_1\r$.
We apply Corollary~\ref{after2} to $S=\left\l S_1=\l v,v_1\r,u_1,v_2,S_2=S\setminus\l S_1,u_1,v_2\r
\right\r$ with $c_1=v$ to show that $S$ is block-decomposable.

\smallskip
\noindent
{\bf Case 2.2.2:} $B_3\in\B_{\rm{II}}$. \ There are three ways to attach $B_3$ to $B_1$ and $B_2$.

\smallskip
\noindent
{\bf Case 2.2.2.1:} $B_3$ contains all the three vertices $v_1,u_1,v_2$. \ Then all the vertices of $\Theta_{S\setminus x}(v)$ are dead ends, so $S$ is block-decomposable by Proposition~\ref{razval1} applied to $x$.

\smallskip
\noindent
{\bf Case 2.2.2.2:} $B_3$ has one point in common with each of $B_1$ and $B_2$. \ We can assume that $B_3$ contains $v_1$.
Then ${\rm Val}_S(v_1)=4$, and $L_S(v_1)$ is connected contradicting Lemma~\ref{wrong4}.


\smallskip
\noindent
{\bf Case 2.2.2.3:} $B_3$ does not intersect $B_2$. \ Denote by $w$ the third vertex of $B_3$. Let us prove first that $v_2$ is not joined with $w$ in $S$. If they are contained in a block of the first type in $S\setminus x$, then all the vertices of $\l
v,v_1,u_1,w,v_2\r$ are dead ends, so $S$ is block-decomposable by Proposition~\ref{razval1} applied to $x$. If $w$ and $v_2$ are contained in a block of second type, then ${\rm Val}_S(w)=4$,
but $L_S(w)$ consists of three connected components contradicting Lemma~\ref{wrong4}.
Therefore, $v_2\perp w$.
Observing that $v_1$ and $u_1$ are dead ends of $\Theta_{S\setminus x}(v)$, we see that they are joined only with $v,w$, and probably $x$.

Take any $y\notin \l\Theta_{S\setminus x}(v),x\r$ and consider $S_1=S\setminus y$ with a block decomposition.
We will prove that $x$ is joined with neither of $v_1$ and $u_1$. This implies that $S$ is block-decomposable by Corollary~\ref{after2}
applied to $S=\left\l S_1'=\l v_1,u_1\r,v,w,S_2'=S\setminus\l S_1',v,w\r\right\r$ with $c_1=v_1$.

Suppose that $x\not\perp v_1$. Since ${\rm Val}_S(v)=3$ and $v\perp \l x,w\r$, $v$ is not contained in one block with any of $x$ and $w$. Thus, $v_1,x$ and $w$ compose a block of second type and $v,v_1$ is a block of first
type. This implies that $v,v_2,u_1$ is a block of second type.

Since $v_2\perp u_1$ in $S$, in order to avoid the edge $(v_2,u_1)$
 there is another block containing $v_2$ and $u_1$.
 Since $u_1\not\perp w$, this block should contain $w$ also.
 But then $v_2\not\perp w$, which is already proved to be false.

\medskip

By exhausting all cases we completed the proof of the lemma.

\end{proof}

\begin{cor}
\label{nodouble}
Minimal non-decomposable quiver $S$ does not contain double edges.

\end{cor}

\begin{proof}
Let $v$ and $u$ be joined by a double edge. Take any non-neighbor $x$ of $v$ and consider $S\setminus x$ with some block decomposition. By Lemmas~\ref{wrong4} and~\ref{wrong3},
valences of $u$ and $v$ do not exceed $2$. Thus, they are joined with each other only and disconnected from the rest of $S$.

\end{proof}

\begin{proof}[Proof of Theorem~\ref{g8}]

Consider a quiver $S$ satisfying assumptions of the theorem and having at least $8$ vertices. By Lemmas~\ref{no5} and~\ref{no4}, valence of any vertex of $S$ does not exceed $4$. By Lemmas~\ref{wrong4} and~\ref{wrong3}, link of any vertex of
valence $4$ consists of two disjoint edges, and link of any vertex of valence $3$ consists of one edge and one vertex. By Corollary~\ref{nodouble}, $S$ does not contain double edges.

Now take all cycles of order $3$ in $S$ and paint all their edges and vertices in red (we assume all the remaining edges and vertices to be black). By Lemmas~\ref{wrong4} and~\ref{wrong3}, any red edge belongs to a unique cycle of order $3$. Any red
vertex is contained either in four red edges, or in two red edges and at most one black edge. Notice also that due to Lemmas~\ref{wrong4} and~\ref{wrong3} each cycle of order $3$ is cyclically oriented.

Denote by $S_1$ the quiver obtained by deleting all red edges from $S$. Let us show that $S_1$ is a forest. Indeed, $S_1$ does not contains vertices of valence $3$ or more. Further, if $S_1$ contains a cycle $C$, then each vertex of this cycle is
contained in two black edges, so the cycle does not contain any red vertex. This implies that no vertex of $S\setminus C$ is joined with $C$ in $S$, so either $S$ is not connected or $S=C$. In the latter case $S$ is block-decomposable.

Take a block decomposition of $S_1$: any edge is a block. It is well defined since there are no vertices of valence $3$ or more. Clearly, any red vertex is an outlet. Now consider each cycle of order $3$ as a block of second type, and glue it to $S_1$. We
obtain a block decomposition of $S$, which contradicts the assumptions of the theorem.

\end{proof}

Now we are able to prove the main results of the section.

\begin{theorem}
\label{min}
The only mutation-finite quivers satisfying assumptions of Theorem~\ref{g8} are ones mutation-equivalent to one of the two quivers $X_6$ and $E_6$ shown on Figure~\ref{minfig}.

\begin{figure}[!h]
\begin{center}
\psfrag{E}{$E_6$}
\psfrag{X}{$X_6$}
\epsfig{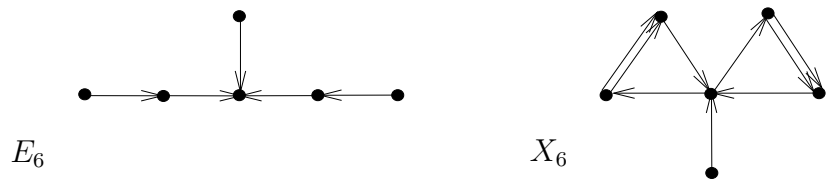}
\caption{Minimal non-decomposable mutation-finite quivers}
\label{minfig}
\end{center}
\end{figure}

\end{theorem}

\begin{remark}
\label{preservedecomp}
We recall that the property of quiver $S$ to be block-decomposable is preserved by mutations. Indeed, according to~\cite{FST}, $S$ is block-decomposable if and only if it corresponds to an ideal triangulation of a punctured bordered surface, and any mutation corresponds to a flip of the triangulation. Thus, any quiver mutation-equivalent to $S$ arises from some triangulation, too. In particular, this implies that the set of non-decomposable quivers is invariant under mutations either. At the same time, the property to be {\it minimal} non-decomposable may not be preserved by mutations {\it a priori}. However, while proving Theorem~\ref{min}, we see that minimal non-decomposable quivers are invariant anyway.

\end{remark}

\begin{proof}[Proof of Theorem~\ref{min}]
The two quivers shown on Figure~\ref{minfig} are mutation-finite and not-decomposable (see \cite[Propositions 4 and 6]{DO}). To prove the theorem, it is sufficient to show that all other mutation-finite quivers on at most $7$ vertices either are block-decomposable, or contain subquivers which are mutation-equivalent to one of $X_6$ and $E_6$. In particular, this will imply that all the quivers mutation-equivalent to $X_6$ or $E_6$ are also minimal non-decomposable.

Let $S$ be a minimal non-decomposable mutation-finite quiver. By Theorem~\ref{g8}, $|S|\le 7$. Since the mutation class of $S$ is finite, multiplicities of edges of $S$ do not exceed $2$ (see Theorem~\ref{less3}). The number of quivers on at most $7$ vertices with bounded multiplicities of edges is finite. This means that we can use a computer to list all quivers, choose mutation-finite ones, and check which of them are block-decomposable.  However, the number of quivers in consideration is large. To reduce the time required for computations, we organize the check as follows.

First, we list all mutation classes of connected  mutation-finite quivers of order $3$, there are three of them (see  \cite[Theorem 7]{DO}), and choose one representative in each class. They all are block-decomposable. Clearly, any connected mutation-finite quiver of order $4$ contains a proper subquiver mutation-equivalent to one of these $3$ quivers of order three.

 Next, we add a vertex and join it with each of the $3$ quivers by edges of multiplicities at most $2$ in all possible ways (we use {\bf C++} program~\cite{programs}). For each obtained quiver we check if its mutation class is finite, and choose one representative from each finite mutation class (here we use {\rm Java} applet for quivers mutations \cite{K}). The resulting list contains $5$ quivers of order $4$, they all are block-decomposable.

Continuing in the same way we get $7$ finite mutation classes of order $5$, again all are block-decomposable.  Then we get $13$ classes of order $6$, exactly two of them consist of non-decomposable quivers, namely quivers mutation-equivalent to $X_6$ and quivers mutation-equivalent to $E_6$ .  Since all mutation-finite quivers with at most $5$ vertices are block-decomposable, all quivers mutation-equivalent $X_6$ or $E_6$ are minimal non-decomposable. Attaching a vertex to representatives of all classes of order $6$, we get $15$ finite mutation classes of quivers of order $7$, three of them  consist of non-decomposable quivers, namely classes containing $E_7$, or $\widetilde E_6$, or $X_7$.   These three mutation classes consist of $416$, $132$, and $2$ quivers respectively. Each quiver from the first two classes contains a subquiver mutation-equivalent to $E_6$,  each quiver from the third class contains a subquiver mutation-equivalent to $X_6$. Therefore, none of them is minimal non-decomposable.

\end{proof}

The following immediate corollary of Theorem~\ref{min} is the main tool in the classification of  mutation-finite quivers.

\begin{cor}
\label{contmin}

Every non-decomposable mutation-finite quiver contains a subquiver mutation-equivalent
to $E_6$ or to $X_6$.

\end{cor}

\section{Classification of non-decomposable quivers}
\label{class}

In this section we use Corollary~\ref{contmin} to classify all non-decomposable  mutation-finite quivers.

\begin{theorem}
\label{all}
A connected non-decomposable mutation-finite quiver of order greater than $2$ is mutation-equivalent
to one of the eleven quivers $E_6$, $E_7$, $E_8$, $\widetilde E_6$, $\widetilde E_7$,
$\widetilde E_8$, $X_6$, $X_7$, $E_6^{(1,1)}$, $E_7^{(1,1)}$, $E_8^{(1,1)}$ shown on Figure~\ref{allfig}.

\begin{figure}[!h]
\begin{center}
\psfrag{1}{$E_6$}
\psfrag{2}{$E_7$}
\psfrag{3}{$E_8$}
\psfrag{1_}{$\widetilde E_6$}
\psfrag{2_}{$\widetilde E_7$}
\psfrag{3_}{$\widetilde E_8$}
\psfrag{1__}{$E_6^{(1,1)}$}
\psfrag{2__}{$E_7^{(1,1)}$}
\psfrag{3__}{$E_8^{(1,1)}$}
\psfrag{4}{$X_6$}
\psfrag{5}{$X_7$}
\epsfig{file=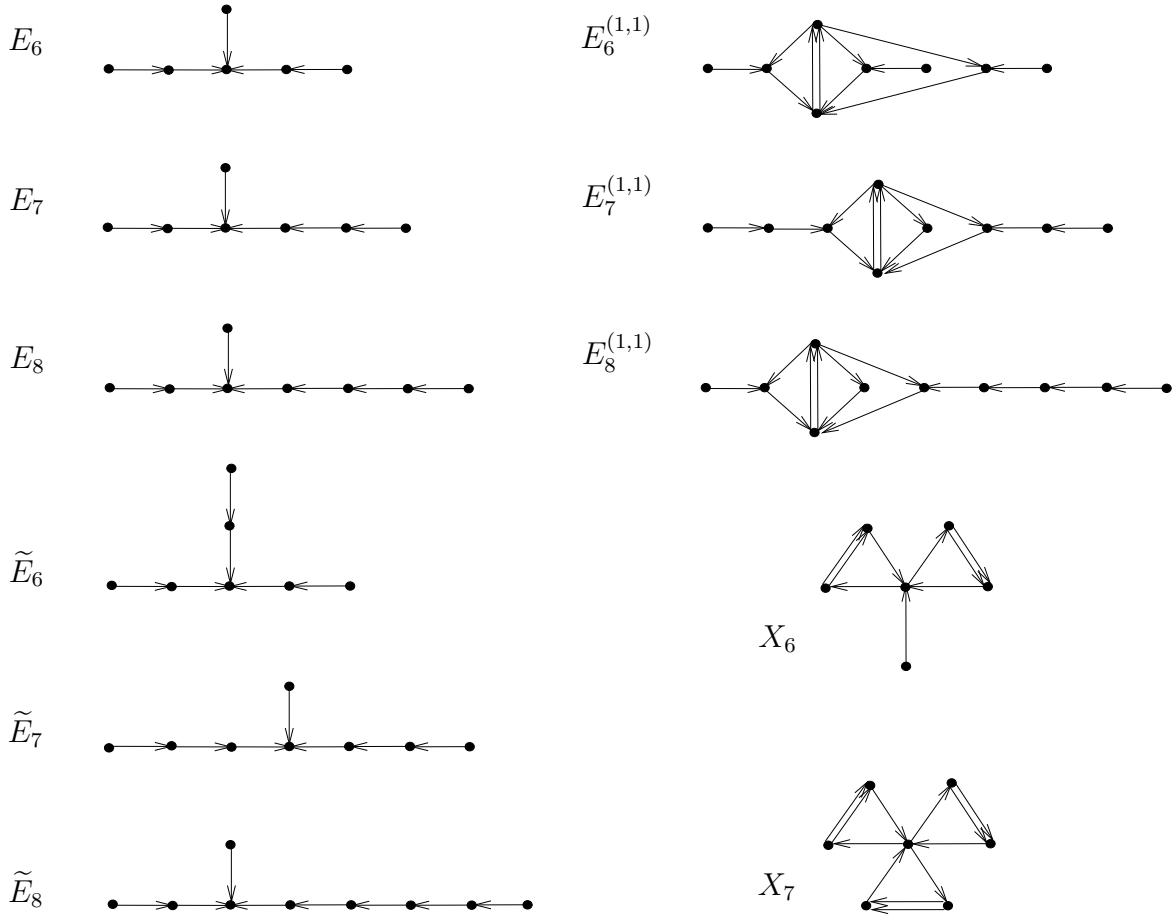,width=1.0\linewidth}
\caption{Non-decomposable mutation-finite quivers of order at least $3$}
\label{allfig}
\end{center}
\end{figure}

\end{theorem}

All these quivers have finite mutation class~\cite{DO} and are non-decompos\-able since each of them contains a subquiver mutation-equivalent to $E_6$ or $X_6$. We need only to prove that this list is complete.

We prove several elementary preparatory statements first.

\begin{lemma}
\label{m1}
Let $S$ be a non-decomposable quiver of order $d\ge 7$ with finite mutation class. Then $S$ contains a non-decomposable mutation-finite subquiver $S_1$ of order $d-1$.

\end{lemma}

\begin{proof}
According to Corollary~\ref{contmin}, $S$ contains a subquiver $S_0$ mutation-equivalent to $E_6$ or $X_6$. Let $S_1$ be any connected subquiver  of $S$ of order $d-1$ containing $S_0$. Clearly, $S_1$ is non-decomposable, and its mutation class is finite.

\end{proof}

\begin{cor}
\label{p1}
Suppose that for some $d\ge 7$ there are no non-de\-com\-po\-sable  mutation-finite quivers of order $d$. Then order of any non-decomposable  mutation-finite quiver does not exceed $d-1$.

\end{cor}

A proof of the following lemma is evident.

\begin{lemma}
\label{submt}
Let  $S_1$ be a proper subquiver of $S$, let $S_0$ be a quiver mutation-equivalent to $S_1$. Then there exists a quiver $S'$ which is mutation-equivalent to $S$ and contains $S_0$.

\end{lemma}

\begin{proof}[Proof of Theorem~\ref{all}]

According to Theorem~\ref{min}, there are exactly two finite mutation classes of non-decomposable quivers of order $6$, namely classes of $E_6$ and $X_6$. Due to Corollary~\ref{contmin}, all other non-decomposable mutation-finite quivers have at least $7$ vertices. By  Lemma~\ref{submt}, in each finite mutation class of non-decomposable quivers there is a representative containing a subquiver $E_6$ or $X_6$.

Therefore, to find all finite mutation classes of non-decomposable quivers of order $7$ we need to attach a vertex to $E_6$ and $X_6$ in all possible ways (i.e. by edges of multiplicity at most two, with all orientations), and to choose amongst obtained $7$-vertex quivers
all mutation-finite classes. This has been done by using {\rm Java } applet~\cite{K} and an elementary {\bf C++} program~\cite{programs} (in fact, the same algorithm was used in the proof of Theorem~\ref{min}). In this way we get $3$ finite mutation classes of quivers of order $7$ with representatives $X_7$, $E_7$, and $\widetilde E_6$.

Now, Lemmas~\ref{submt} and~\ref{m1} allow us to continue the procedure.  To list all finite mutation classes of non-decomposable quivers of order $8$ we attach a vertex to $X_7$, $E_7$, and $\widetilde E_6$ in all possible ways, and choose again all mutation-finite classes.  The result consists of $3$ mutation classes with representatives $E_8$,  $\widetilde E_7$,  and $E_6^{(1,1)}$.

In the same way we analyzed the quivers of order $9$ and obtained $2$ mutation classes with representatives $\widetilde E_8$  and $E_7^{(1,1)}$. To find all finite mutation classes of non-decomposable quivers of order $10$, we apply  the same procedure  to $\widetilde E_8$ and $E_7^{(1,1)}$.
 The result is a unique mutation class containing $E_8^{(1,1)}$.

Finally, the same procedure applied to $E_8^{(1,1)}$ gives no mutation-finite quivers at all. This implies that there are no non-decomposable mutation-finite quiver of order $11$.
 Now Corollary~\ref{p1} yields immediately the Theorem statement.

\end{proof}

\bigskip

\section{Minimal mutation-infinite quivers}
\label{inf}

The main goal of this section is to provide a criterion for a quiver to be mutation-finite, namely to prove Theorem~\ref{crit}. For given quiver $S$, this criterion allows to check if $S$ is mutation-finite in a polynomial in $|S|$ time.

\begin{definition}
\label{definf}
A \emph{minimal mutation-infinite quiver} $S$ is a quiver that
\begin{itemize}
\item has infinite mutation class;
\item any proper subquiver of $S$ is mutation-finite.
\end{itemize}
\end{definition}

\begin{example}
Any mutation-infinite quiver of order $3$ is minimal. This is caused by the fact that any quiver of order at most $2$ is mutation-finite.
\end{example}

Clearly, any minimal mutation-infinite quiver is connected. Notice that the property to be minimal mutation-infinite is not mutation invariant.
Indeed, any mutation-infinite class contains quivers with arbitrary large multiplicities of edges. If $|S|>3$, then taking a connected subquiver of $S$ of order $3$ containing an edge of multiplicity greater than $2$ we get a proper subquiver of $S$ which is mutation-infinite (see~Theorem~\ref{less3}). Note also that minimal mutation-infinite quiver of order at least $4$ does not contain edges of multiplicity greater than two.

We will deduce the criterion from the following lemma.

\begin{lemma}
\label{le10}

Any minimal mutation-infinite quiver contains at most $10$ vertices.

\end{lemma}

The bound provided in the lemma is sharp: there exist numerous minimal mutation-infinite quivers of order $10$. We show some of them below. One source of such examples are simply-laced Dynkin diagrams of root systems of hyperbolic Kac-Moody algebras with any orientations of edges. There are two such diagrams of order $10$ (e.g., see~\cite{Kac}), examples of corresponding quivers are shown on Fig.~\ref{e10}.

\begin{figure}[!h]
\begin{center}
\epsfig{file=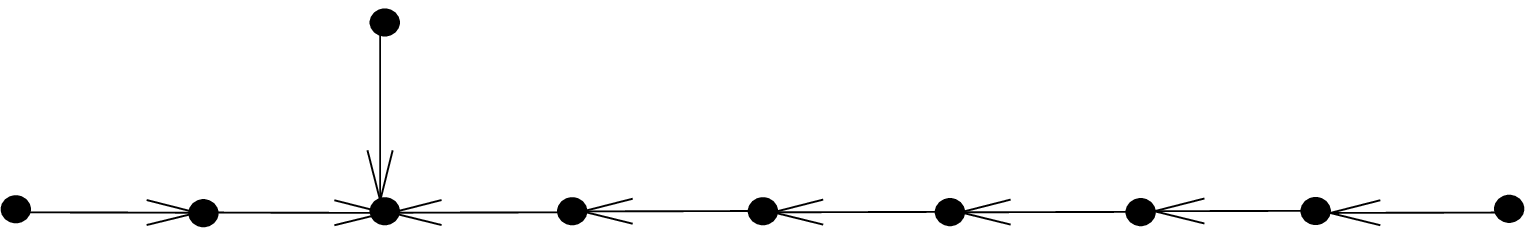,width=0.44\linewidth}\qquad\qquad\epsfig{file=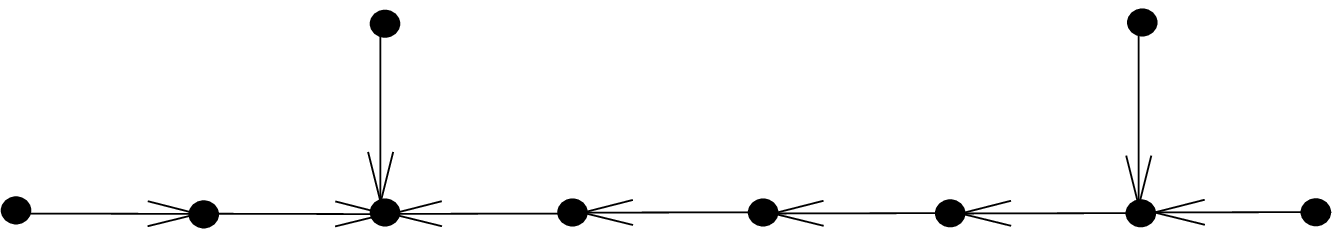,width=0.4\linewidth}
\caption{Minimal mutation-infinite quivers of order $10$ coming from Dynkin diagrams}
\label{e10}
\end{center}
\end{figure}

\begin{remark}
There is no general algorithm to determine if two infinite-mutational quivers are mutation-equivalent. However, for {\it acyclic} quivers (i.e., containing no oriented cycles) the following result is known (see~\cite[Corollary~4]{K3}): if two acyclic quivers are mutation-equivalent, then there exists a sequence of mutations from one of them to another via acyclic quivers only. In particular, this implies that two quivers shown on Fig.~\ref{e10} are not mutation-equivalent. Indeed, they both are trees, but it is easy to see that the only way to change the topological type of tree by mutation is to create an oriented cycle.

\end{remark}

Another series of examples can be obtained from the quiver shown on Fig.~\ref{newinf}.

\begin{figure}[!h]
\begin{center}
\epsfig{file=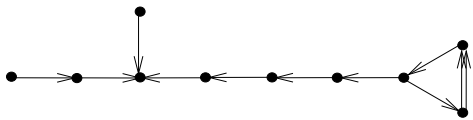,width=0.4\linewidth}
\caption{Minimal mutation-infinite quiver of order $10$}
\label{newinf}
\end{center}
\end{figure}

Note that it is not clear if the quiver shown on the Fig.~\ref{newinf} is not mutation-equivalent to one of the quivers shown on Fig.~\ref{e10}.

\begin{proof}[Proof of Lemma~\ref{le10}]
Let $S$ be a minimal mutation-infinite quiver.

First, we prove a weaker statement, i.e. we show that $|S|\le 11$. In fact, this bound follows immediately from Theorems~\ref{g8} and~\ref{all}.
Indeed, either all the proper subquivers of $S$ are block-decomposable, or $S$ contains a proper finite mutational non-decomposable subquiver of order $|S|-1$ (we can assume that this quiver is connected: if it is not connected but non-decomposable, it contains a non-decomposable connected component $S_0$, and any connected subquiver of $S$ of order $|S|-1$ containing $S_0$ is non-decomposable). In the former case $|S|\le 7$ according to Theorem~\ref{g8} (again, we emphasize that we did not require $S$ to be mutation-finite in the assumptions of Theorem~\ref{g8}). In the latter case $|S|-1\le 10$ due to Theorem~\ref{all}, which proves inequality  $|S|\le 11$.

Now suppose that $|S|=11$. Then $S$ contains a proper finite mutational non-decomposable
subquiver $S'$ of order $10$. According to Theorem~\ref{all}, $S'$ is mutation-equivalent
to $E_{10}^{(1,1)}$. The mutation class of $E_{10}^{(1,1)}$ consists of $5739$ quivers,
which can be easily computed using Keller's \rm{Java} applet~\cite{K}. In other words, we see that $S$
contains one of $5739$ quivers of order $10$ as a proper subquiver.

Hence, we can list all minimal mutation-infinite quivers of order $11$ in the following
way. To each of $5739$ quivers above we add one vertex in all possible ways (we can do that since
the multiplicity of edge is bounded by two; the sources of the program can be found in~\cite{programs}). For every obtained quiver we check whether all its proper subquivers of order $10$ (and, therefore, all the others) are mutation-finite. However, the resulting set of the procedure above is empty: every obtained quiver has at least one mutation-infinite subquiver of order $10$, so it is not minimal.

\end{proof}

As a corollary of Lemma~\ref{le10}, we get the criterion for a quiver to be mutation-finite.

\begin{theorem}
\label{crit}
A quiver $S$ of order at least $10$ is mutation-finite if and only if all subquivers of $S$ of order $10$ are mutation-finite.

\end{theorem}

\begin{proof}
According to Definition~\ref{definf}, every mutation-infinite quiver contains some minimal mutation-infinite quiver as a subquiver. Thus, a quiver is mutation-finite if and only if it does contain minimal mutation-infinite subquivers. By Lemma~\ref{le10}, this holds if and only if all subquivers of order at most $10$ are mutation-finite. Since a subquiver of a mutation-finite quiver is also mutation-finite, the latter condition, in its turn, holds if and only if all subquivers of order $10$ are mutation-finite, which completes the proof.

\end{proof}

\section{Growth of skew-symmetric cluster algebras}
\label{growth}

We recall the definition of growth of cluster algebra~\cite[Section~11]{FST}.
\begin{definition}
A cluster algebra has \emph{polynomial growth} if the number of
distinct seeds which can be obtained from a fixed initial seed by at most $n$ mutations is bounded from above by
a polynomial function of~$n$.  A cluster algebra has \emph{exponential growth}
if the number of such seeds is bounded from below by an exponentially
growing function of~$n$.
\end{definition}

In~\cite[Proposition~11.1]{FST} a complete classification of block-decomposable quivers corresponding to algebras of polynomial growth is given.
It occurs that growth is polynomial if and only if the surface corresponding to a quiver is a sphere with at most three punctures and boundary components in total.

We prove the following theorem.

\begin{theorem}
\label{inf-exp}
Any mutation-infinite skew-symmetric cluster algebra has exponential growth.
\end{theorem}

Combining these two results, we see that to classify all cluster algebras of polynomial growth we need only to
determine the growth of $11$ exceptional mutation-finite algebras listed in Theorem~\ref{all}. Three of them, namely
$E_6$, $E_7$ and $E_8$ are of finite type, so they have finite number of seeds. Other three ($\widetilde E_6$, $\widetilde E_7$ and $\widetilde E_8$)
are of affine type, so they have linear growth (according to H.~Thomas). Therefore, there are still $5$ algebras for which the growth is unknown. 

In other words, to complete the classification of cluster algebras by the growth rate 
it remains to ascertain the rates of growth in the following five cases $X_6$, $X_7$, $E_6^{(1,1)}$, $E_7^{(1,1)}$, and $E_8^{(1,1)}$.

A sequence of cluster transformations preserving the exchange matrix defines \emph{a group-like element}. The set of all group-like elements form \emph{generalized modular group}.

Using ideas similar to the proof of famous Tits alternative, it can be proved that in all five cases the growth rate of the generalized modular group is exponential. More precisely, studying the attracting points of some induced action it can be proved that the generalized modular group 
contains as a subgroup the free group of rank two.

Details on the study of growth rates in exceptional cases will be published elsewhere.

\begin{cor}
A skew-symmetric cluster algebra of rank at least $3$ has a polynomial growth if and only if
\begin{itemize}
  \item it is associated with triangulation of either a sphere with three punctures, either a disk with two punctures, either an annulus with one puncture, or a pair of pants;
  \item it is one of the following exceptional affine cases: $\tilde E_6,\tilde E_7,\tilde E_8$. 
\end{itemize}
\end{cor}

\begin{remark} 
Note that by construction any cluster algebra of rank $2$ is either of finite type of has a linear (i.e., polynomial) growth rate.
\end{remark}








\smallskip

The rest of this section is devoted to the proof of Theorem~\ref{inf-exp}.

\begin{remark}
\label{inf3}
It is sufficient to prove  Theorem~\ref{inf-exp} for cluster algebras corresponding to mutation-infinite quivers of order $3$.
Indeed, any mutation-infinite quiver $S$ has a mutation-equivalent quiver $S'$ with an edge of weight at least $3$.
According to Theorem~\ref{less3}, any connected subquiver $S_0\subset S$ of order $3$ containing that edge is mutation-infinite.
Therefore, it is enough to show that the algebra corresponding to $S_0$ grows exponentially.
\end{remark}

In fact, we prove some stronger result. Denote by $S^{(n)}$ the set of quivers which can be obtained from a quiver $S$ by at most $n$ mutations.
According to Remark~\ref{inf3}, the following lemma implies Theorem~\ref{inf-exp}.

\begin{lemma}
\label{exp}
Let $S$ be a mutation-infinite quiver of order $3$. Then the order $|S^{(n)}|$ grows exponentially with respect to $n$.
\end{lemma}

The proof of Lemma~\ref{exp} splits into two steps. We start by proving the following lemma.
By saying that a quiver is {\it oriented} we mean that all the cycles are oriented.

\begin{lemma}
\label{out}
For any mutation-infinite quiver $S$ of order $3$ there exists a sequence of at most $4$ mutations taking $S$ to $S'$ such that

{\rm(1)} $S'$ is oriented;

{\rm(2)} all the weights of edges of $S'$ are greater than $1$;

{\rm(3)} an edge of maximal weight is unique.

\end{lemma}

\begin{proof}
First, we make $S$ oriented without empty edges. For this, we need at most $2$ mutations.
Indeed, if $S$ is non-oriented without empty edges, then the mutation in a unique vertex which is neither sink nor source leads to an oriented quiver.
If $S$ has an empty edge, then by at most one mutation we put sink and source to the ends, and then, mutating in the middle vertex, we get a required quiver.

Thus, we can assume now that $S$ is oriented without empty edges, and we have two mutations left to satisfy conditions $(2)$ and $(3)$.
Denote the weights of $S$ by $(a,b,c)$ with $a\ge b\ge c$.
Since $S$ is mutation-infinite, $S$ does not coincide with any of the three mutation-finite quivers of order $3$ shown on Fig.~\ref{fin-mut3}.
\begin{figure}[!h]
\begin{center}
\epsfig{file=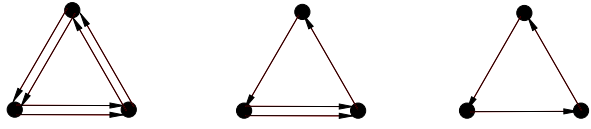,width=0.5\linewidth}
\caption{Oriented mutation-finite quivers of order $3$ without empty edges}
\label{fin-mut3}
\end{center}
\end{figure}
In particular, $a\ge b\ge 2$. We may assume that either $a=b$ or $c=1$.
If $c=1$, then making a mutation preserving $a$ and $b$, we obtain an oriented quiver with weights $(a,b,c'=ab-c)$.
Clearly, $c'=ab-c\ge 3$, so the condition $(2)$ holds.

Now we have a quiver satisfying the first two conditions, and one mutation left to satisfy condition $(3)$.
Again, let  the weights of $S$ be $(a,a,c)$ with $a\ge c$. If $c=2$, then we make a mutation changing $c$ to get
a quiver with weights $(a,a,c'=a^2-c)$. Since $S$ is mutation-infinite, $a>2=c$, therefore $c'=a^2-2>a$, so the third condition is satisfied.
If $c>2$, then mutating in any of the two other vertices, we get a quiver with weights  $(a,b=ac-a,c)$. Clearly, $b=ac-a>a$ since $c>2$.

\end{proof}

The last step in proving Lemma~\ref{exp} is Lemma~\ref{tree}.
We say that a sequence of mutations is {\it reduced} if it does not contain two consecutive mutations in the same vertex. Note also that we differ quivers with the same weights but different orientations.

\begin{lemma}
\label{tree}
Let $S$ be mutation-infinite quiver of order $3$ satisfying conditions $(1)-(3)$ of Lemma~\ref{out}, denote by $(a,b,c)$ the weights of edges of $S$, $a>b\ge c$. Let $S_1$ and $S_2$ be quivers obtained from $S$ by different reduced sequences of mutations, such that the first mutation in each sequence preserves the weight $a$ of $S$. Then $S_1$ and $S_2$ are distinct.

\end{lemma}

Clearly, Lemma~\ref{tree} together with Lemma~\ref{out} imply Lemma~\ref{exp}. Before proving Lemma~\ref{tree}, we provide the following auxiliary statement.

\begin{lemma}
\label{max}
Let $S$ fit into assumptions of Lemma~\ref{tree}. Suppose that $S'$ is obtained from $S$ by mutation preserving the weight $a$. Then

$1)$ the maximal weight of $S'$ is greater than $a$;

$2)$ $S'$ satisfies conditions $(1)-(3)$ of Lemma~\ref{out}.

\end{lemma}

\begin{proof}
To prove the first statement, compute the weights of $S'$. If we preserve weights $a$ and $b$, then weights of $S'$ are $(a,b,c'=ab-c)$, so $c'>a$ since $b\ge 2$ and $c<a$. If we preserve weights $a$ and $c$, then weights of $S'$ are $(a,b'=ac-b,c)$, so $b'>a$ since $c\ge 2$ and $b<a$.

Now the second statement is evident.

\end{proof}

The following immediate corollary of Lemma~\ref{max} is a partial case of Lemma~\ref{tree}.

\begin{cor}
\label{dif}
Let $S$ fit into assumptions of Lemma~\ref{tree}, and let $\mu_n\dots\mu_1$ be a reduced sequence of mutations, where $\mu_1$ preserves the maximal weight of $S$. Denote by $S_i$ the quiver $\mu_i\dots\mu_1S$. Then all the quivers $S,S_1,\dots,S_n$ are distinct.

\end{cor}

\begin{proof}[Proof of Lemma~\ref{tree}]
Suppose $S_1$ and $S_2$ coincide. We may assume that any two quivers $S_1'$ and $S_2'$ in the sequences of quivers from $S$ to $S_1$ and $S_2$ respectively are distinct. Consider preceding quivers $S_1'$ and $S_2'$ in the sequences of quivers from $S$ to $S_1$ and $S_2$. By Lemma~\ref{max}, both $S_1$ and $S_2$ satisfy the following property: the edge of the maximal weight is opposite to the vertex in which the last mutation was made. Therefore, $S_1'$ and $S_2'$ coincide also. The contradiction proves the lemma.

\end{proof}

\section{Quivers of order $3$}
\label{q3}

The structure of mutation classes of quivers of order $3$ was described
in~\cite{ABBS} and~\cite{BBH}. These papers provide complete classification of
mutation classes containing quivers without oriented cycles given in different terms.

Define the {\it total weight} of a quiver as the sum of the weights of edges. It is proved
in~\cite{ABBS} (see also~\cite[Lemma~2.1]{BBH}) that if a mutation class does not contain quivers without oriented cycles, then the mutation class contains a unique (up to duality) quiver of minimal total weight, and any other mutation-equivalent quiver can be reduced to that one in a unique way.

We complete the description of mutation classes containing quivers without oriented cycles by a similar statement. We use notation from~\cite{BBH}: a quiver $S$ of order $3$ is called {\it cyclic} if it is an oriented cycle, and {\it acyclic} otherwise; $S$ is called {\it cluster-cyclic} if any quiver mutation-equivalent to $S$ is cyclic, and {\it cluster-acyclic} otherwise.

\begin{theorem}
\label{umin}
Let $S$ be a connected cluster-acyclic quiver of order $3$. Then

$1)$ mutation class of $S$ contains a unique (up to change of orientations of edges) quiver $S_0$ without oriented cycle;

$2)$ the total weight of $S_0$ is minimal amongst all the mutation class;

$3)$ any sequence of mutations decreasing total weight at each step applied to $S$ ends in $S_0$.
\end{theorem}

We use the following notation. By $(a,b,c)^-$ we mean a non-oriented cycle with weights $(a,b,c)$. We may not fix orientations of edges since any two such quivers are mutation-equivalent under mutations in sources or sinks only. Similarly, by $(a,b,c)^+$ we mean an oriented cycle with weights $(a,b,c)$.

\begin{proof}[Proof of Theorem~\ref{umin}]
Consider a connected acyclic quiver $S_0=(a,b,c)^-$ with $a\ge b\ge c$. Denote by $\mathcal S$ the set of quivers satisfying conditions \rm{(1)}--\rm{(3)} of Lemma~\ref{out}.

Suppose first that none of the weights is equal to $1$. If $c=0$, then the only quiver we can get by one mutation different from $S_0$ is $(a,b,ab)^+$, which is contained in $\mathcal S$. If $c\ge 2$, then we can obtain three possibly different quivers $(a,b,ab+c)^+$, $(a,ac+b,c)^+$ and $(bc+a,b,c)^+$ all belonging to $\mathcal S$. According to Lemma~\ref{max}, all other quivers from mutation class of $S_0$ also belong to $\mathcal S$, which proves the first statement. Moreover, Lemmas~\ref{tree} and~\ref{max} imply that for any quiver from the mutation class of $S_0$ belonging to $\mathcal S$ there is a unique reduced sequence of mutations decreasing the total weight at each step, and the minimal element is in $\mathcal S$ if and only if the entire mutation class is contained in $\mathcal S$ (this is also proved in~\cite[Lemma~2.1]{BBH}). Since $S_0$ is the only quiver not contained in $S$, all the statements are proved.

Now suppose that at least one of the weights of $S_0$ is $1$. We may assume that $S_0$ is mutation-infinite (there are exactly two acyclic mutation-finite quivers of order $3$, namely $\widetilde A_2$ and $A_3$, and the theorem is evident for them). Our aim is to show that almost all quivers mutation-equivalent to $S_0$ belong to $S$, then looking at the remaining quivers all the statements become evident. Indeed, $S_0$ is of one of the following three types:
$(a,1,0)^-$, $(a,1,1)^-$, or $(a,b,1)^-$, where $a\ge b\ge 2$. We list the quivers which can be obtained from that three by mutations.

The only way to change $(a,1,0)^-$ is to obtain $(a,a,1)^+$, from which we may get $(a,a,a^2-1)^+$ only which is in $\mathcal S$.

The quiver $(a,1,1)^-$ can be mutated into $(a+1,1,1)^+$ and $(a+1,a,1)^+$. The first one then can be mutated into the second one only, and the latter into the first one or into $(a,a+1,a^2+a-1)^+\in\mathcal S$.

The quiver $(a,b,1)^-$ can be mutated either into $(a,b,ab+1)^+\in\mathcal S$, or into $(a+b,b,1)^+$ or $(a+b,a,1)^+$. These two can be mutated either one into another or into quivers belonging to $\mathcal S$.

Therefore, in the mutation class of $S_0$ there is at most $3$ quivers not belonging to $\mathcal S$, and $S_0$ has minimal total weight amongst them. Applying the same arguments as in the first case, we complete the proof.

\end{proof}


\end{document}